\documentclass{amsart}
\usepackage{amsmath}
\usepackage{amsfonts}
\usepackage{amssymb,amscd,amsthm}
\usepackage{geometry}
\usepackage{mathrsfs}
\usepackage[bookmarksnumbered, colorlinks, linktocpage, plainpages]{hyperref}
\usepackage{nameref}
\usepackage{wasysym,amssymb,indentfirst}
\usepackage{xcolor}
\usepackage{enumitem}
\usepackage{mathtools}
\usepackage{multirow}
\usepackage{appendix}
\usepackage{booktabs}
\usepackage{ragged2e}
\usepackage{setspace}
\usepackage{array}

\definecolor{RevisionPurple}{RGB}{92,38,128}
\definecolor{RevisionGreen}{RGB}{0,112,70}

\newcommand{\revisionnew}[1]{{\color{RevisionPurple}#1}}

\usepackage{cases}
\usepackage{graphicx}
\usepackage{ragged2e}

\usepackage{fmtcount}

\numberwithin{equation}{section}

\newcommand{\algma}{\mathcal{U}}
\newcommand{\Hom}{\text{Hom}}

\newcommand{\lpa}{-liftable }
\newcommand{\sinv}{sphere invariant}

\newcommand{\spc}{\mathbb{C}}
\newcommand{\spr}{\mathbb{R}}
\newcommand{\spo}{\mathbb{O}}

\newcommand{\re}{\operatorname{Re}}

\newcommand{\lif}{\mathrm{lif}\,}
\newcommand{\ext}{\mathrm{ext}\,}
\newcommand\fx[2]{\left<#1,#2\right>}%

\newcommand\fsh[1]{\|#1\|}

\def\O{\mathbb{O}}
\def\R{\mathbb{R}}

\def\abs#1{\left|#1\right|}

\newtheorem{mydef}{Definition}[section]
\newtheorem{rem}[mydef]{Remark}
\newtheorem{eg}[mydef]{Example}
\newtheorem{cor}[mydef]{Corollary}
\newtheorem{prop}[mydef]{Proposition}
\newtheorem{lemma}[mydef]{Lemma}
\newtheorem{thm}[mydef]{Theorem}

\newtheorem{step }[stp]{Step }

\begin{document}
	\title[Octonionic Riesz-Dunford  functional calculus]{Octonionic Riesz-Dunford  functional calculus}	
	\author{Qinghai Huo}
	
	\email[Q.~Huo]{hqh86@mail.ustc.edu.cn}
	\address{School of Mathematics,		Hefei University of Technology,		 Hefei 230601, China}
	
	\author{Guangbin Ren}
	\email[G.~Ren]{rengb@ustc.edu.cn}
	\address{Department of Mathematics, University of Science and Technology of China, Hefei 230026, China}

	\author{Irene Sabadini}

	\email[I.~Sabadini]{irene.sabadini@polimi.it}
	\address{Politecnico di Milano, Dipartimento di Matematica, Via E. Bonardi 9, 20133 Milano, Italy}
	
	\author{Zhenghua Xu}
	
	\email[Z.~Xu]{zhxu@hfut.edu.cn}
	\address{School of Mathematics,		Hefei University of Technology,		 Hefei 230601, China}

	\date{}
	

	\thanks{Q. Huo is  supported  partially by the National Natural Science Foundation of China (No 12301097) and the Fundamental Research Funds for the Central Universities (No. 	JZ2025HGTB0171).  G. Ren is  supported  partially by the National Natural Science Foundation of China (No. 12571090). I. Sabadini  is partially supported by PRIN 2022 {\em Real and Complex Manifolds: Geometry and Holomorphic Dynamics} and is member of GNASAGA of INdAM. Z. Xu is partially supported by  Anhui Provincial Natural Science Foundation (No. 2308085MA04), Fundamental Research Funds for the Central Universities (No. JZ2025HGTG0250) and China Scholarship Council (Nos. 202506690055, 202506690052).}

	\begin{abstract}

		The Riesz-Dunford functional calculus over the algebra of octonions, denoted by $\mathbb{O}$, has long been an open problem due to the nonassociativity of octonions. Two core obstacles hinder its development: first, the generalization of the resolvent operator series identity  produces unexpected associator terms that invalidate standard expansions; second, the nonassociativity spoils the analyticity of the resolvent operator, a key property for defining a functional calculus via Cauchy integrals. 	In this paper,	we  initiate the study of the Riesz-Dunford functional calculus for bounded  power-associative para-linear operators in Banach octonionic  bimodules. To address the above issues, we introduce several pivotal concepts: power-associative operators (to eliminate the unwanted associator terms and recover valid resolvent series expansions), the notions of regular inverse of $R_s-T$ for  $s\in \O$ (which serve as the octonionic versions of the resolvent operator),  $\mathbb{C}_J$-extendable power-associative operators, and $\mathbb{C}_J$-liftable power-associative operators (to characterize the slice regularity of the resolvent operators). Based on these notions, we define two types of octonionic spectra: the pull-back spectrum $\sigma^*(T)$ and the push-forward spectrum $\sigma_*(T)$. These give rise to  the  left and right slice regular functional calculi of bounded power-associative para-linear operators, respectively.
		This theory unifies the Riesz-Dunford functional calculus over  division algebras ($ \mathbb{C}, \mathbb{H}, \mathbb{O}$) and fills the six-decade-long gap in octonionic (nonassociative) functional analysis.

	\end{abstract}

	\maketitle
	
	\tableofcontents

	\section{Introduction}
	\label{sec:intro}
	
	\subsection{Background and Motivation}

	The development of spectral theory over division algebras has driven progress in operator theory, algebraic geometry, and quantum mechanics. Among the four normed division algebras — real numbers ($\mathbb{R}$), complex numbers ($\mathbb{C}$), quaternions ($\mathbb{H}$), and octonions ($\mathbb{O}$) — the algebra of  octonions stands out as the unique nonassociative and noncommutative one.	
	In contrast to the well-established spectral theory for the complex case and the rapidly developing spectral theory for the quaternionic case, this distinctive algebraic structure of octonions led to the absence of a systematic octonionic spectral theory for six decades prior to recent advancements.
	
	The Riesz-Dunford functional calculus \cite{Dunford1958} stands as a cornerstone of modern operator theory and spectral theory, bridging complex analysis and linear operator theory in Banach spaces. Proposed by F. Riesz and N. Dunford, this theory enables the definition of operator-valued functions $f(T)$ for bounded linear operators $T$ via a canonical Cauchy integral formula: 
	$$f(T)=\frac{1}{2\pi i}\int_{\gamma}f(\lambda)(\lambda \mathcal{I}-T)^{-1}d\lambda.$$ Here, $f$ denotes a function holomorphic in a neighborhood of $T$'s spectrum \(\sigma(T)\), \(\gamma\) represents a simple closed contour enclosing \(\sigma(T)\), and $\mathcal{I}$ represents the identity operator. By establishing this rigorous connection, it resolves the fundamental challenge of extending scalar function operations to operators. This calculus preserves algebraic properties of scalar functions, ensuring coherence between function composition and operator multiplication. Its applications permeate spectral theory, differential equations, and mathematical physics, providing essential tools for analyzing operator behavior. As a foundational framework, it also lays the groundwork for generalizations to unbounded operators and multi-operator systems.

	For associative division algebras $\mathbb{C}$ and  $\mathbb{H}$, spectral theory, the Riesz-Dunford functional calculus theory, is well-known:
	\begin{enumerate}
		\item
		
		Complex spectral theory hinges on the holomorphic functional calculus \cite{Dunford1958}, which is constructed upon two core properties: first, the resolvent operator series identity given by
		\begin{eqnarray}\label{eqintr:complex resol op series}
			(s\mathcal{I} - T)^{-1} = \sum_{n=0}^\infty T^n s^{-(n+1)}
		\end{eqnarray}
		for all complex numbers \( s \) satisfying \( |s| > \|T\| \); and second, the holomorphicity of the function \( (s\mathcal{I} - T)^{-1} \) with respect to the complex variable \( s \) in the resolvent set \( \rho(T) \).
		Here $T$ is a bounded complex linear operator on a complex Banach space.
		\item Quaternionic spectral theory extends the complex case via slice hyperholomorphic functional calculus and it is based on the so-called S-spectrum. It is nowadays a well established field of research, see the books  \cite{AlpColSab2016,AlpColSaba2020,ColGanbook2019,ColGanKimbook2018,colombo2011noncomfunctcalculus} and  references therein, adapting to noncommutativity while resting on associativity.
		The quaternionic spectral theory rooted in the slice-Cauchy integral formula \cite[Section 4.4, 4.5]{colombo2011noncomfunctcalculus}, which inspired the discovery of the so-called S-resolvent operator series identities
		\begin{eqnarray}
			\sum_{n \geq 0} T^n s^{-1 - n} &= -\left( T^2 - 2\operatorname{Re}[s]T + |s|^2 \mathcal{I} \right)^{-1} (T - \overline{s}\mathcal{I}),\qquad |s| > \|T\|;\label{eqintr:H res op l}\\
			\sum_{n \geq 0} s^{-1 - n} T^n &= -\left( T - \overline{s}\mathcal{I}\right) \left( T^{ 2} - 2\operatorname{Re}[s]T + |s|^2 \mathcal{I} \right)^{-1},\qquad |s| > \|T\|.\label{eqintr:H res op r}
		\end{eqnarray}
		Here $T$ is a bounded right quaternionic linear operator in a two-sided quaternionic Banach space and the scalar multiplications are defined by $$(sT)(v):=L_sT(v):=s(Tv),\qquad (Ts)(v):=(TL_s)(v)=T(sv)$$ for any $s\in \mathbb H$ and any element $v$ in the two-sided quaternionic Banach space.
		This  gives rise to the  definitions of  left $S$-resolvent operator as
		\begin{eqnarray}\label{eqintr:SL-1}
			S_{\rm L}^{-1}(s, T) := -\left(T^2 - 2\operatorname{Re}[s]T + |s|^2 \mathcal{I}\right)^{-1}(T - \overline{s}\mathcal{I}),
		\end{eqnarray}
		and the right $S$-resolvent operator as	
		\begin{eqnarray}\label{eqintr:SR-1}
			S_{\rm R}^{-1}(s, T) := -(T - \overline{s}\mathcal{I})\left(T^2 - 2\operatorname{Re}[s]T + |s|^2 \mathcal{I}\right)^{-1}.
		\end{eqnarray}
		These definitions further lead to the crucial notion of S-spectrum defined as \begin{eqnarray}
			\sigma_S(T) = \left\{ s \in \mathbb{H} \, : \, T^2 - 2 \operatorname{Re}[s]T + |s|^2 \mathcal{I} \text{ is not boundedly invertible} \right\}.		
		\end{eqnarray}
		It also proves that the $S$-resolvent operator acting on some elements in the two-sided quaternionic Banach space is a slice regular function with respect to the quaternion variable \( s \) in the resolvent set \( \rho_S(T):=\mathbb{H}\setminus \sigma_S(T) \).
		
		Based on these results and the recently developed theory of slice regular function over quaternions \cite{Gentili2007slice,colombo2011noncomfunctcalculus,Gentili2013slice}, the Riesz-Dunford functional calculus theory, as well as other spectral theories, can be well established in the quaternionic case.
		This method can also be extended to the case of associative Clifford algebras \cite{Colombo2009IsraelJM}.						
		
	\end{enumerate}
	
	In contrast, octonionic nonassociativity presents significant obstacles within these frameworks.
	The field of octonionic functional analysis originated by Goldstine and Horwitz in 1964 \cite{goldstine1964hilbert}, who conducted foundational investigations on octonionic Hilbert spaces. However, recent studies have revealed that one of the axioms in Goldstine and Horwitz’s definition of an octonionic Hilbert space, see \cite{goldstine1964hilbert}, can be derived from the other axioms. This insight has led to the introduction of the crucial concept of para-linearity \cite{huoqinghai2022Riesz}, which is essential for addressing the octonionic generalization of the Riesz representation theorem. It turns out that the appropriate subject of study in octonionic functional analysis is para-linear operators \cite{huo2025BLMSHB,huo2025Trendsmath,huoqinghai2025Oselfadjoint}.

	A preliminary breakthrough in octonionic spectral theory was only achieved recently: for bounded self-adjoint para-linear operators with strong eigenvalues on octonionic Hilbert spaces, we have developed a preliminary spectral theory \cite{huoqinghai2025Oselfadjoint}, based on para-linearity—a nonassociative analog of linearity \cite{huoqinghai2022Riesz}.
	
	However, the spectral theory for general bounded octonionic para-linear operators acting on Banach $\mathbb{O}$-bimodules remains an open problem. Resolving this issue has profound mathematical significance, as it would complete the spectral theories on division algebras, and bears practical relevance for extending nonassociative analysis to domains such as octonionic quantum mechanics.

	\subsection{Core Challenges  and Resolutions}
	
	It is a significant goal in octonionic functional analysis to establish the Riesz-Dunford functional calculus theory for octonionic para-linear operators.
	
	Let $V$ be a Banach $\mathbb{O}$-bimodule (see Definition \ref{def:O-bimod}). A bounded real linear operator $T \in \mathscr{B}_{\mathbb{R}}(V)$ is called \textbf{right para-linear} if
	\[
	\operatorname{Re} B_p(T,x) = 0, \quad \forall \, p \in \mathbb{O}, \, \forall \, x \in V,
	\]
	where the so-called \textbf{second right associator} $B_p(T,x)$ is defined by
	\[
	B_p(T,x) := T(x)p - T(xp),
	\]
	and $\operatorname{Re}: V \to \operatorname{Re}(V)$ denotes the real part operator (see Definition \ref{def:real part}). The set of all {bounded} right para-linear operators on $V$ is denoted by $\mathscr{B}_{\mathcal{RO}}(V)$.		
	The obstacle to generalizing classical spectral theory to the octonionic setting  arises from two fundamental issues, both rooted in the nonassociativity of $\mathbb{O}$ that we discuss below.
	
	\begin{enumerate}[label=(\arabic*), leftmargin=*, itemsep=0.3cm]
		\item \textbf{Generalization of the Resolvent Operator Series Identity}: In light of \eqref{eqintr:complex resol op series} \eqref{eqintr:H res op l} and \eqref{eqintr:H res op r}, for bounded complex linear or right quaternionic linear operators $T$, the power series $$\sum_{n=0}^\infty T^n s^{-(n+1)}, \text{ or }\ \sum_{n=0}^\infty s^{-(n+1)}T^n $$ will converge to a complex linear or \textbf{right quaternionic linear} operator provided $\abs{s}>\fsh{T}$.
		To mimic this approach, we first note that to obtain an \textbf{octonionic right para-linear} operator, we   consider series of the form
		\begin{eqnarray}\label{eqintr:resolv op O}
			\sum_{n\geqslant 0} T^{\circledcirc n} \odot s^{-1-n} \quad \text{or} \quad \sum_{n\geqslant 0} s^{-1-n} \odot T^{n\circledcirc},
		\end{eqnarray}
		where both the multiplication by a scalar and the composition should be suitably replaced to preserve the right para-linearity. Specifically, we use the \textbf{octonionic scalar multiplication 	``$\odot$''} (as defined in \eqref{eqdef:pf} and \eqref{eqdef:fp}),  the \textbf{regular composition ``$\circledcirc$''} (Definition \ref{def:right mod regular composition}) and the \textbf{regular power}  $T^{\circledcirc n}$ or  $T^{n\circledcirc}$ (Definition \ref{def:ncirc}).
		
		However, for general para-linear operators, there appears  no easily expressible {closed} form of the operator defined by the series  \eqref{eqintr:resolv op O}. Indeed, the result we obtained is far more complicated than {in all the associative cases}.
		
		The \textbf{composition associator} associated with the para-linear maps $f$ and $g$ is  defined as
		\[
		[f, g, x]_{\circledcirc} := (f \circledcirc g)(x) - (f \circ g)(x).
		\]
		For any $s \in \O$, let $R_s: V \to V$ denote the right scalar multiplication operator defined by $R_s(v) = vs$. \\
		Nevertheless, we have proved the following result, in which some unexpected associators emerge (see Theorem \ref{thm:Rs-T}):
		\begin{thm}
			Let $V$ be a Banach $\O$-bimodule and $T$ be a bounded para-linear operator on $V$. Suppose $s\in \spc_J\subseteq\O$ for some $J\in \mathbb S$ such that $|s|>\fsh{T}.$ Then we have
			\begin{enumerate}
				\item For all $x\in \spc_J(V):=\re V\oplus J\re V$, we have	\begin{eqnarray}\label{eqintr:1resolveformula}
					(R_s-T)\sum_{n\geqslant 0}\left(T^{\circledcirc n}\odot s^{-1-n}\right)(x)=x+\alpha(s,T)(x).
				\end{eqnarray}
				Here $\alpha(s,T):\spc_J(V)\to V$ is defined by $\displaystyle 	\alpha(s,T)(x):=\sum_{n\geqslant 0}[T,T^{\circledcirc n}, xs^{-1-n}]_{\circledcirc}.$
				\item For all $x\in V$, we have
				\begin{eqnarray}\label{eqintr:2resolveformula}
					\pi_J \sum_{n\geqslant 0}(s^{-1-n}\odot T^{n\circledcirc})(R_s-T)(x)=\pi_J x+\beta(s,T)(x).
				\end{eqnarray}
				Here   $\beta(s,T):V\to \spc_J(V)$ is defined by $\displaystyle \beta(s,T)(x):=\pi_J \sum_{n\geqslant 0}[T^{n \circledcirc},T,x]_{\circledcirc}s^{-1-n}$ and $\pi_J$ is the projection onto $\spc_J(V) $.
			\end{enumerate}
		\end{thm}
		
		The \textbf{identities \eqref{eqintr:1resolveformula} and \eqref{eqintr:2resolveformula} play a key role} in establishing the octonionic version of the Riesz-Dunford functional calculus theory.
		
		To deal with the unexpected associator terms $\alpha(s,T)$ and $\beta(s,T)$, inspired by Lemma \ref{lem:powass} below, we introduce the notion of \textbf{power-associative operators} (Definition \ref{def:powerass}) for which both  $\alpha(s,T)$ and $\beta(s,T)$ vanish.
		It is important to note that the assumption of \textbf{power-associativity} is automatically satisfied  when working in the associative quaternionic case.  Moreover, since an octonionic para-linear operator $T$ can be written as
		\[
		T = \sum_{i=0}^7 e_i \odot T_i,\qquad T_i \text{ is $\O$-linear for }i=0,\dots,7,
		\]
		then $T$ is power-associative (Proposition \ref{prop:commutativepowass})
		also in the case in which it has commuting components $T_i$.   Compared with \cite{Taylor1970acta}  in this article we first address the  functional calculus in which the eight components are  considered in one operator with non-associative  coefficients.

		For power-associative operators  (Corollary \ref{cor:reolv op ser}) we proved that:
		\begin{thm}
			Let $T\in \mathscr{B}_{\mathcal{RO}}(V)$ be a power-associative operator and   $s\in \O$ such that $|s|>\fsh{T}$. Then  we have
			\begin{eqnarray}
				(R_s-T)^{\circledcirc -}&=&\sum_{n\geqslant 0}T^{ n}\odot s^{-1-n}\label{intreq:powextinvRs-T};\\
				(R_s-T)^{-\circledcirc }&=&\sum_{n\geqslant 0}s^{-1-n}\odot T^{n} \label{intreq:powlifinvRs-T}.
			\end{eqnarray}
			Here the \textbf{regular inverses} $(R_s-T)^{\circledcirc -}$ and $(R_s-T)^{-\circledcirc }$ are defined in Definition \ref{def:ncirc}.
			
		\end{thm}
		The so-called \textbf{regular inverses} are defined by two crucial bijections: {$\lif$ }and $\ext$ (see Definition \ref{def:lif and ext}), which serve as a bridge connecting real linear operators to para-linear operators. It is worth noting that \textbf{$R_s$ is not, in general, right para-linear} (and it is not right quaternionic linear in the quaternionic case). {Thus $R_s$ will always be considered as a real linear operator.}

		We use the \textit{regular inverses} $(R_s-T)^{\circledcirc -}$ and $(R_s-T)^{-\circledcirc }$ to define the left and right \textbf{octonionic resolvent operators}, respectively (see Definition \ref{def:resolvent op}). Roughly speaking, we employ $(R_s-T)^{\circledcirc -}$ to \textbf{pull back} to the real part $\re V$ or $\spc_J$-slice $\spc_J(V)$, and then \textbf{extend} the result to the entire space by means of the bijection $\ext$; symmetrically, we use $(R_s-T)^{-\circledcirc }$ to \textbf{push forward} to the projection onto the real part $\re V$ or  $\spc_J(V)$, and then \textbf{lift} the outcome to the full target space via the bijection $\lif$. This approach runs  through in this framework.
		
		We further note that this approach is equally applicable to the quaternionic case, where the resolvent operators defined in this manner coincide with those presented in \eqref{eqintr:SL-1} and \eqref{eqintr:SR-1} (see Remark \ref{rem:resop}).

		\item \textbf{Slice Regularity of the Resolvent Operator}: 			
		Classical spectral theory relies on the holomorphicity (or slice regularity in the case of quaternions) of resolvent operators, which  forms the foundation of functional calculus through Cauchy integrals. For octonionic para-linear operators, nonassociativity gives rise to associators when commuting the  operator  $\left( \frac{\partial}{\partial x} + \frac{\partial}{\partial y}R_J \right)$  with $T$ and $R_s$. This phenomenon causes the fact that the resolvent operators $(R_s - T)^{\circledcirc -}$ and $(R_s - T)^{-\circledcirc }$ lose their slice regularity, in general. To characterize in which cases the resolvent operators are slice regular, we introduce the concepts of \textbf{$\mathbb{C}_J$-extendable operators} and \textbf{$\mathbb{C}_J$-liftable operators} (see Definitions \ref{def:ext pa} and \ref{def:lpa}), which induce the concepts of octonionic \textbf{pull-back spectrum} $\sigma^{*}(T)$ and \textbf{push-forward spectrum} $\sigma_{*}(T)$ (Definitions \ref{def:pullback spec} and \ref{def:pushforward sepc}).
\begin{mydef}[Octonionic spectrum]\label{intrdef:pullback spec}
Let $T\in \mathscr{B}_{\mathcal{RO}}(V)$ be a bounded right para-linear operator. For $J\in\mathbb S$, define
\[
\rho_J^*(T):=\{s\in\mathbb C_J:\ R_s-T\text{ is boundedly invertible and }(R_s-T)^{-1}
\text{ is }\mathbb C_J\text{-extendable power associative}\},
\]
\[
\sigma_J^*(T):=\mathbb C_J\setminus\rho_J^*(T),\qquad
\sigma^*(T):=\bigcup_{J\in\mathbb S}\sigma_J^*(T),\qquad
\rho^*(T):=\mathbb O\setminus\sigma^*(T).
\]
Equivalently,
\[
\rho^*(T)=
\left(\bigcup_{J\in\mathbb S}(\rho_J^*(T)\setminus\mathbb R)\right)
\cup\left(\bigcap_{J\in\mathbb S}\rho_J^*(T)\right).
\]
The push-forward slice spectrum, global spectrum and global resolvent set are defined analogously from the $\mathbb C_J$-liftability condition.
\end{mydef}

		We point out that the requirement of being $\mathbb{C}_J$-extendable or $\mathbb{C}_J$-liftable holds automatically  for the operator $(R_s-T)^{-1}$ in the quaternionic case. Thus the octonionic spectrum can be compared with the quaternionic the S-spectrum (see Remark \ref{rem:pull-back spec}).
		
		Now we can establish the slice regularity for  resolvent operators on the  resolvent set as in classical cases, see Theorems \ref{thm:rhosliceopen} \ref{thm:Rs-Tsliceregular}, \ref{lem:sigmaTbd}, \ref{thm:extspec cpt} for the  pull-back spectrum and Theorems \ref{thm:pushforw slice open}, \ref{thm:liftRs-Tsliceregular}, \ref{thm:pushforw bound}, \ref{thm:pushfw cpt}  for the push-forward spectrum.
\begin{thm}\label{intrthm:Rs-Tsliceregular}
Let $T\in \mathscr{B}_{\mathcal{RO}}(V)$ be a bounded right para-linear operator. For every $J\in\mathbb S$, the sets $\rho_J^*(T)$ and ${\rho_*}_J(T)$ are open in $\mathbb C_J$. In general the global resolvent sets need not be slice-open. Moreover:
\begin{enumerate}
\item For any $v\in\re V$ and $\phi\in V^{*_{\O}}$, the function
\[
s\longmapsto \phi((R_s-T)^{\circledcirc-}(v))
\]
is holomorphic on each $\rho_J^*(T)$; consequently it is right slice regular on every slice-open set contained in $\rho^*(T)$.
\item For any $v\in V$, $J\in\mathbb S$, and $\phi\in V^{*_{\O}}$, the function
\[
g(s):=\phi(\pi_J(R_s-T)^{-\circledcirc}(v)),\qquad s\in {\rho_*}_J(T),
\]
is holomorphic from ${\rho_*}_J(T)$ to $\mathbb C_J$.
\end{enumerate}
If $T$ is power-associative and $V\neq\{0\}$, then $|s|>\|T\|$ implies $s\in\rho^*(T)\cap\rho_*(T)$; each slice spectrum $\sigma_J^*(T)$ and ${\sigma_*}_J(T)$ is nonempty and compact, while the global spectra are nonempty and bounded subsets of $\{s\in\mathbb O:|s|\leq\|T\|\}$. 
\end{thm}

	\end{enumerate}
	
	\subsection{Main Results}

	Having overcome these challenges, we establish the octonionic version of the Riesz–Dunford functional calculus theory.
	
	We first generalize the notion of $T$-admissible sets from the quaternionic  \cite{colombo2011noncomfunctcalculus} to the octonionic case (see Definition \ref{def:T-admiss}).
	We denote by $\mathcal{SR}^L(\Omega)$ and $\mathcal{SR}^R(\Omega)$ the set of all left and right  slice regular functions on $\Omega$ (see Definition \ref{def:slice reg}), respectively. The set of
	slice preserving  regular functions is denoted by $\mathcal{SR}_{\R}(\Omega)$ (see Definition \ref{def:slice pres}).
	We remark that for an axially symmetric domain $\Omega$, both $\mathcal{SR}^L(\Omega)$ and $\mathcal{SR}^R(\Omega)$ are $\mathbb{O}$-bimodules, where the $\mathbb{O}$-scalar multiplications are defined via the so-called slice products $\bullet^L$ and $\bullet^R$ (see Definition \ref{def:slice product}). Furthermore, their real parts coincide with $\mathcal{SR}_{\mathbb{R}}(\Omega)$.
{When replacing $\Omega$ by an octonionic spectrum, the notation is understood in the local sense of Definition \ref{def:T-admiss}: the function is slice regular on a slice domain containing the closure of a suitable admissible domain around the spectrum (see Definition \ref{def:Tadmiss}).}
	
	Our main results  (Theorems \ref{thm:f*TindependentU} and \ref{thm:main thm2}) are:
	\begin{thm}\label{intrthm:f*TindependentU}
		Let $T\in \mathscr{B}_{\mathcal{RO}}(V)$ be a power-associative operator.	Let $U\subseteq \O$  be a \revisionnew{bounded }$T$-left-admissible (resp. $T$-right-admissible) domain and $f\in  {\mathcal{SR}_{\R}}({\sigma^*(T)})$ (resp. $f\in  {\mathcal{SR}_{\R}}({\sigma_*(T)})$) be a slice preserving function. For any $J\in \mathbb S$, set ${d}s_J = -{d}sJ$. Then the integral
		\begin{eqnarray*}\label{intreq:intfTJ}
			f^*(T)_J:&=&\dfrac{1}{2\pi} \int_{\partial (U\cap \spc_J)}  (R_s-T)^{\circledcirc -}\odot ({d}s_Jf(s))\\
			\Bigg( \text{resp. } f_*(T)_J:&=&\dfrac{1}{2\pi} \int_{\partial (U\cap \spc_J)}(f(s){d}s_J)   \odot(R_s-T)^{-\circledcirc }\Bigg)
		\end{eqnarray*}
		does not depend on the choice of  $U$.
		Moreover, the map
		\begin{eqnarray*}\label{intreqdef:f(T)}
			f^*(T) \quad(\text{resp. } f_*(T)):\mathbb S&\to& \mathscr{B}_{\mathcal {RO}}(V)\\
			J&\mapsto& f^*(T)_J \quad(\text{resp. } f_*(T)_J)\notag
		\end{eqnarray*} is continuous, and  for any $I,J\in \mathbb S$, we have
		$$\operatorname{Re}_{\mathscr{B}_{\mathcal{RO}}(V)}\, f^*(T)_J=\operatorname{Re}_{\mathscr{B}_{\mathcal{RO}}(V)}\, f^* (T)_I\qquad (\text{resp. }\operatorname{Re}_{\mathscr{B}_{\mathcal{RO}}(V)}\, f_*(T)_J=\operatorname{Re}_{\mathscr{B}_{\mathcal{RO}}(V)}\, f_* (T)_I).$$
		Here $\operatorname{Re}_{\mathscr{B}_{\mathcal{RO}}(V)}$ is the real part operator on the $\O$-bimodule $\mathscr{B}_{\mathcal{RO}}(V)$.
	\end{thm}
	
	Let $\Gamma(\mathbb S, \mathscr{B}_{\mathcal{RO}}(V))$ denote the set of continuous right para-linear operator-valued sections on the six dimensional sphere $\mathbb S$. The $\O$-bimodule structure of $\mathscr{B}_{\mathcal{RO}}(V)$   induces an $\O$-bimodule structure on $\Gamma(\mathbb S, \mathscr{B}_{\mathcal{RO}}(V))$. Theorem \ref{intrthm:f*TindependentU} gives rise  to the definition of octonionic functional calculus (Definitions \ref{def:ofunccal} and \ref{def:lifofunccal}).

	
	\begin{mydef}[The octonionic  functional calculus]\label{intrdef:ofunccal}
		Let $T\in \mathscr{B}_{\mathcal{RO}}(V)$ be a power-associative operator.	Define
		\begin{eqnarray*}
			(\Phi_T)_0:{\mathcal{SR}_{\R}}({\sigma^*(T)})&\to& \Gamma(\mathbb S, \mathscr{B}_{\mathcal{RO}}(V))\\
			f&\mapsto& (f^*(T))(J):=f^*(T)_J=\dfrac{1}{2\pi} \int_{\partial (U\cap \spc_J)}  (R_s-T)^{\circledcirc -}\odot ({d}s_Jf(s)), \notag
		\end{eqnarray*}	
		where $U\subseteq \O$  is a $T$-left-admissible domain.
		
		We then define the octonionic left slice regular functional calculus as $$\Phi_T:=\ext 	(\Phi_T)_0:\mathcal{SR}^L({\sigma^*(T)})\to \Gamma(\mathbb S, \mathscr{B}_{\mathcal{RO}}(V)).$$
		Analogous definition for the octonionic right slice regular functional calculus
		$$\Psi_T:\mathcal{SR}^R({\sigma_*(T)})\to \Gamma(\mathbb S, \mathscr{B}_{\mathcal{RO}}(V)).$$
	\end{mydef}
	For concrete examples, we refer the reader to Section \ref{sec:eg}. {We note that, by definition, the functional calculi depend on $J\in \mathbb S$.}
	We call an operator	$T$ is  \textbf{(left)  \sinv} if for any $f\in \mathcal{SR}^L({\sigma^*(T)})$, $f^*(T)_J$ is independent of $J\in \mathbb S$. Similar definition can be given for \textbf{(right) \sinv}. It is worth noting that  if we regard the real part operator $\operatorname{Re}_{\mathscr{B}_{\mathcal{RO}}(V)}$ as the projection onto the nucleus, where the nucleus coincides with the original space in the associative case,  then every right quaternionic linear operator $T$ is automatically ``\sinv''.
	
	We emphasize that although the properties of $\Phi_T$ and $\Psi_T$ turn out to be similar, their proofs proceed in entirely different ways: one relies on the pull-back method, while the other is based on the push-forward method. Hence, in this work, we present both proofs.
	
	For general slice regular functions,   the functional calculus is as follows (see Theorems \ref{thm:f(T)} and \ref{thm:rslice fc}):
	\begin{thm}
		Let $T\in \mathscr{B}_{\mathcal{RO}}(V)$ be a power-associative operator.
		\begin{enumerate}
			\item 	Let	$U\subseteq \O$  be a  $T$-left-admissible domain.
			Suppose $f=\sum_{i=0}^7f_{(i)}\bullet^L e_i\in\mathcal{SR}^L({\sigma^*(T)}) $ with $f_{(i)}$ being slice preserving for $i=0,\dots,7$. Then 
			\begin{eqnarray*}\label{intreq:f(T)J}
				f^*(T)_J	&=&\dfrac{1}{2\pi} \int_{\partial (U\cap \spc_J)} 	(R_s-T)^{\circledcirc -}\odot ({d}s_Jf(s)) + \\
				&&	\dfrac{1}{2\pi}\sum_{i=1}^{7} \int_{\partial (U\cap \spc_J)} \big[(R_s-T)^{\circledcirc -},{d}s_Jf_{(i)}(s),e_i\big]_{\mathscr{B}_{\mathcal{RO}}(V)}.\notag
			\end{eqnarray*}
			\item Let	$U\subseteq \O$  be a  $T$-right-admissible domain.
			Suppose $f=\sum_{i=0}^7e_i\bullet^R f_{(i)}\in\mathcal{SR}^R({\sigma_*(T)}) $ with $f_{(i)}$ being slice preserving for $i=0,\dots,7$. Then
			\begin{eqnarray*}\label{intrlifeq:f(T)J}
				f_*(T)_J	&=&\dfrac{1}{2\pi} \int_{\partial (U\cap \spc_J)} (f(s){d}s_J)\odot 	(R_s-T)^{-\circledcirc } - \\
				&&	\dfrac{1}{2\pi}\sum_{i=1}^{7} \int_{\partial (U\cap \spc_J)} \big[e_i,f_{(i)}(s){d}s_J,(R_s-T)^{-\circledcirc }\big]_{\mathscr{B}_{\mathcal{RO}}(V)}.\notag
			\end{eqnarray*}

		\end{enumerate}
	\end{thm}

	Finally, we investigate the algebraic properties of the octonionic functional calculus.
	Note that both $\mathscr{B}_{\mathcal{RO}}(V)$ and $\mathcal{SR}^L(\Omega)$ (resp. $\mathcal{SR}^R(\Omega)$) are not only $\mathbb{O}$-bimodules, but also $\mathbb{O}$-algebras (see Theorems \ref{thm:slice algebra}, \ref{thm:BROV is Oalg}  and Proposition \ref{prop:SR O-alg}), where the notion of  $\mathbb{O}$-algebra (Definition \ref{def:O alg}) is introduced in \cite{huoqinghai2025Oselfadjoint}.
	By definition, $\Phi_T$ (resp. $\Psi_T$) are right (resp. left) para-linear maps. The following result illustrates how the algebraic structure is preserved:
\begin{thm}\label{inrtthm:algebra prop}
Let $V$ be a Banach $\O$-bimodule and let $T\in\mathscr{B}_{\mathcal{RO}}(V)$ be power-associative and both left and right \sinv. Then:
\begin{enumerate}
\item If $f\in\mathcal{SR}_{\R}(\sigma(T))$ and $g\in\mathcal{SR}^L(\sigma(T))$, then
\[
\operatorname{Re}f_{*}(T)\circledcirc g^*(T)=\operatorname{Re}(fg)^{*}(T).
\]
\item If $f\in\mathcal{SR}^R(\sigma(T))$ and $g\in\mathcal{SR}_{\R}(\sigma(T))$, then
\[
\operatorname{Re}f_{*}(T)\circledcirc g^*(T)=\operatorname{Re}(fg)_{*}(T).
\]
\end{enumerate}
Here $\operatorname{Re}$ is the real-part operator on $\Gamma(\mathbb S,\mathscr{B}_{\mathcal{RO}}(V))$.
\end{thm}

		A further phenomenon, which has no direct analogue in the classical
		complex or quaternionic $S$-spectral setting, is that the global
		octonionic spectrum need not be axially symmetric.  This is not merely
		a technical possibility.  In Example~\ref{eg:non sphere}, for the
		power-associative operator
		\[
		T=
		\begin{pmatrix}
			0&-e_1\\
			e_1&0
		\end{pmatrix}
		\quad\text{on }\mathbb O^2,
		\]
		we compute explicitly that
		\[
		\sigma^*(T)
		=
		\left\{
		s\in\mathbb O:
		|s|=1,\;
		\langle s,e_1\rangle_{\mathbb R}=0
		\right\}.
		\]
		Thus $\sigma^*(T)$ is the unit $6$-sphere in the
		$7$-dimensional real hyperplane $e_1^{\perp_{\mathbb R}}$.
		In particular,
		\[
		\frac{1+e_2}{\sqrt2}\in\sigma^*(T),
		\qquad
		\frac{1+e_1}{\sqrt2}\notin\sigma^*(T),
		\]
		although these two octonions have the same real part and the same
		modulus of their imaginary parts.  Hence, unlike the quaternionic
		$S$-spectrum, the octonionic pull-back spectrum is not determined only
		by the pair
		\[
		\bigl(\operatorname{Re}s,|\operatorname{Im}s|\bigr),
		\]
		and may genuinely detect the direction of the imaginary part.
		This shows that the dependence on the slice parameter
		$J\in\mathbb S$ is an intrinsic feature of octonionic spectral theory,
		rather than an artifact of the construction.

		Example~\ref{eg:non sphere} also has two further consequences for the
		structure of the theory.  First, although the present functional
		calculus is formulated on axially symmetric admissible domains, the
		spectrum to be enclosed by such a domain may itself be genuinely
		non-axially symmetric.  This indicates that a future functional
		calculus directly adapted to non-axially symmetric spectral domains
		would be natural in the octonionic setting.  Second, the same example
		shows that the classical global spectral mapping identity cannot be
		expected without additional hypotheses.  Indeed,  we have $T^2=I$ in Example~\ref{eg:non sphere}, and hence
		\[
		\sigma^*(T^2)=\{1\},
		\]
		whereas
		\[
		\{s^2:s\in\sigma^*(T)\}=\sigma^*(T).
		\]
		Thus even the polynomial $p(s)=s^2$ exhibits a genuinely different
		global spectral behavior.

	To summarize, this article establishes a unified framework for investigating the Riesz-Dunford functional calculus theory over the division algebras \(\mathbb{C}\), \(\mathbb{H}\), and \(\mathbb{O}\). The key results of the three distinct Riesz-Dunford functional calculus theories are displayed in Tables \ref{tab:spec}, \ref{tab:formula}, and \ref{tab:prop} in Section \ref{sec:conclusion}.

	\section{Preliminaries}
	\label{sec:algebraic_foundations}
	
	In this section, we review  and establish the algebraic framework for para-linear operators within Banach octonionic modules, while also summarizing the key results in the theory of octonionic slice regular functions.
	\subsection{Octonions}
	The octonions $\mathbb{O}$ form an $8$-dimensional normed division algebra over $\mathbb{R}$ with basis $\{e_0=1, e_1, \dots, e_7\}$. The basis elements multiply via the cyclic rules
	\[
	e_i e_j = \epsilon_{ijk} e_k - \delta_{ij}, \quad i,j,k = 1,\dots,7,
	\]
	where $\epsilon_{ijk} = 1$ for oriented triples $(123)$, $(145)$, $(176)$, $(246)$, $(257)$, $(347)$, $(365)$. These rules are encoded in \textbf{Fano Plane}:
	\
	
	\noindent \small{\textbf{Fig.1} Fano mnemonic graph}
	\begin{flushright}
		\centerline{\includegraphics[width=3cm]{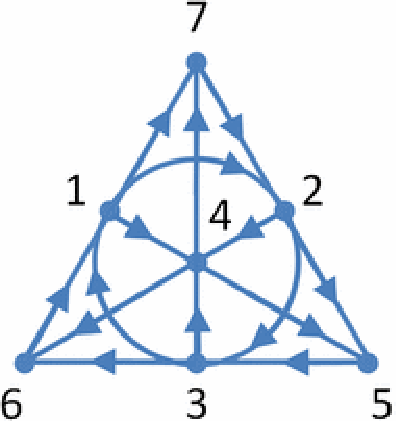}}
	\end{flushright}
	\normalsize

	Some other basic properties include:
	
	\begin{enumerate}
		\item \textit{Conjugation and Norm:} For $x = x_0 + \sum_{i=1}^7 x_i e_i$,
		\[
		\overline{x} = x_0 - \sum_{i=1}^7 x_i e_i, \quad |x| = \sqrt{x\overline{x}}.
		\]
		
		\item \textit{Non-Associative Structure:} The associator $[x,y,z] = (xy)z - x(yz)$ is fully antisymmetric. The commutator $[x,y] = xy - yx$ measures non-commutativity.
		
		\item \textit{Moufang Identities:} For all $x,y,z \in \mathbb{O}$:
		\[
		(xyx)z = x(y(xz)), \quad
		z(xyx) = ((zx)y)x, \quad
		x(yz)x = (xy)(zx).
		\]
		
		\item \textit{Five-Terms Identity:} For all $x,y,z,w \in \mathbb{O}$:
		
		\begin{equation}
			\label{eq:five_term}
			x[y,z,w] + [x,y,z]w = [xy,z,w] - [x,yz,w] + [x,y,zw].
		\end{equation}

		\item
		Denote by $\mathbb S $ the set of all imaginary units of $\O$.	
		It is well-known that the automorphism  group of the algebra of octonions is $G_2$.	For any $J\in \mathbb S$, there exists a (not unique) automorphism $\tau\in G_2$  such that $\tau(e_4)=J$. Denote $J_i:=\tau(e_i)$ for $i=1,\dots,7$. Then
		\begin{eqnarray}\label{eqdef:orthbasis}
			\{1,\ J_1,\ J_2,J_3=J_1J_2,\ J_4=J,\ J_5=J_1J,\ J_6=J_2J,\ J_7=(J_1J_2)J\}
		\end{eqnarray}
		is an orthonormal basis of $\O$. Such an orthonormal basis \eqref{eqdef:orthbasis} is called a \textbf{standard orthonormal basis (related to $J$)}.
		
	\end{enumerate}
	
	\subsection{Octonionic Module Theory}
	We now define module structures which are compatible with octonionic non-associativity.
	
	\begin{mydef}
		An $\mathbb{R}$-vector space $M$ is a \textbf{right $\mathbb{O}$-module} if equipped with a right scalar multiplication satisfying the alternativity  of right associators:
		\[
		[x,p,q]_M := (xp)q - x(pq) = -[x,q,p]_M, \quad \forall x \in M, \, p,q \in \mathbb{O}.
		\]
		Similar definition can be introduced for \textbf{left $\mathbb{O}$-module}. If there are no confusions, we shall omit the subscript $M$ for simplicity.
	\end{mydef}
	
	\begin{mydef}
		A right $\mathbb{O}$-module $M$ is an \textit{$\mathbb{O}$-bimodule} if  there is a left   multiplication satisfying
		\[
		[p,q,x]_M = [q,x,p]_M = [x,p,q]_M = -[q,p,x]_M, \quad \forall p,q \in \mathbb{O}, \, x \in M.
		\]
		Here the left associators and middle associators are defined canonically. The subscript $M$ will be omitted for simplicity if no confusions.
	\end{mydef}

	The regular $\O$-bimodule, which as a set is just $\O$, equipped with the natural module structure, is the unique irreducible $\mathbb{O}$-bimodule by Schafer \cite{Schafer1952repaltalg} and   Jacobson \cite{jacobson1954structure,Springer2000jordanalg}. This enables the definition of the real part of an $\mathbb{O}$-bimodule \cite{huoqinghai2020nonass}.
	The following definition is slightly different from that in \cite{huoqinghai2020nonass}, yet equivalent to the one presented therein. This adjustment is made to reproduce the quaternionic results when replacing $\mathbb{O}$ with $\mathbb{H}$.
	\begin{mydef}\label{def:real part}
		For an $\mathbb{O}$-bimodule $M$, the \textbf{real part} $\operatorname{Re} M$ is defined as
		\[
		\operatorname{Re} M := \{m \in M : [p,q,m] =0,\, [p,m]:=pm-mp= 0,  \ \forall p,q \in \mathbb{O}\}.
		\]
		Any element $m\in 	\operatorname{Re}  M$ is called an \textbf{associative} element.
		The \textbf{real part operator} is defined as the  projection $\operatorname{Re}_M: M \to \operatorname{Re} M$. If no confusion, we shall omit the subscript $M$.
		
		For any $J\in \mathbb S$, we denote by
		\begin{eqnarray}\label{eq:CJM}
			\mathbb C_J(M):=\operatorname{Re}  M\oplus J\operatorname{Re} M
		\end{eqnarray}
		the \textbf{$\mathbb C_J$-slice} of $M$.
	\end{mydef}
	
	With slight modifications,   we rewrite the results on properties of real part operator  in \cite[Section 3.1]{huoqinghai2020nonass} as follows.
	\begin{prop}[Properties of Real Part Operator]
		\label{prop:real_part}
		Let $M$ be an $\mathbb{O}$-bimodule. Then the real part has the following properties.
		\begin{enumerate}
			\item $\operatorname{Re} M =  \{m \in M : [p,q,m] = 0,  \ \forall p,q \in \mathbb{O}\}= \{m \in M : pm = mp,  \ \forall p \in \mathbb{O}\}$.
			\item $\re\, [p,q,m]=\re\, [p,m]=0$ for any $p,q\in \O$ and $m\in M$.
			\item $M \cong \operatorname{Re} M \otimes_{\mathbb{R}} \mathbb{O}$ as $\mathbb{O}$-bimodules.
			\item 	Each $x$ can be decomposed into
			\begin{eqnarray}\label{eq:redec}
				x=\sum_{i=0}^{7}e_ix_i
			\end{eqnarray}
			with $x_i\in \re M$ for $i=0,\dots,7$.
			\item $\operatorname{Re}$ is given explicitly by
			\begin{equation}
				\label{eq:real_part_formula}
				\operatorname{Re} x = \frac{5}{12}x - \frac{1}{12} \sum_{i=1}^7 e_i x e_i.
			\end{equation}
			
		\end{enumerate}
		Moreover, one can substitute the standard orthonormal basis  $\{1,J_1,\dots,J_7\}$ related to some $J\in \mathbb S$ for $\{1,e_1,\dots,e_7\}$ in \eqref{eq:redec} and \eqref{eq:real_part_formula}.
	\end{prop}
	
	\begin{rem}\label{rem:uniform law}
		In general, let $M$ be an alternative bimodule over an alternative algebra $\mathbb K$, where $\mathbb K=\O$ or $\mathbb H$  (for definition of alternative bimodules, see \cite{jacobson1954structure}). The \textbf{(commutative) center} of $M$ is defined as
		$$\mathscr{Z}_{\mathbb K}(M):=\{m\in M\mid pm=mp\text{ for all } p\in \mathbb K\}$$ and the \textbf{(right) nucleus} of $M$ consisting of all associative elements is defined as
		$$\mathscr{A}_{\mathbb K}{(M)}:=\{m\in M\mid [m,p,q]=0 \text{ for all } p\in \mathbb K\}.$$
		If  $\mathbb K=\mathbb H$, then $$\mathscr{A}_{\mathbb H}{(M)}=M,\qquad \mathscr{Z}_{\mathbb H}(M)=\re M.$$
		If  $\mathbb K=\mathbb O$, then by Proposition \ref{prop:real_part},  $$\mathscr{A}_{\mathbb O}{(M)}= \mathscr{Z}_{\mathbb O}(M)=\re M.$$
		
		{Although we confine ourselves to the octonionic case, all the discussions can go back to the quaternionic case, by taking care of the  counterpart of the real part $\re M$ in that framework: it may be the nucleus $\mathscr{A}_{\mathbb H}{(M)}$ or the commutative center $\mathscr{Z}_{\mathbb H}{(M)}$, i.e., $M$ or $\re M$. Hence the quaternionic counterpart of the  real part operator may be the projection operator on $M$ or $\re M$.}
		
	\end{rem}

	Let $M,M'$ be two $\O$-bimodules. 	Denote by $\operatorname{Hom}_{\mathbb{R}}(M,M') $ the set of all real linear   maps from $M$ to $M'$.
	A real linear map $f\in \operatorname{Hom}_{\mathbb{R}}(M,M') $ is called \textit{left (resp. right) octonionic linear} if
	$$f(px)=pf(x)\qquad  (\text{resp.}\ \ f(xp)=f(x)p )$$ for all $x\in M$ and all $p\in \O$.
	By Remark 3.2 in \cite{huo2024aacasubmod},  \textbf{the left octonionic linearity of  $f$ is equivalent to the right octonionic linearity}. Hence we shall just use the terminology ``octonionic linearity".
	Combining with \eqref{eq:real_part_formula}, we immediately obtain that,   for an octonionic linear map $f$ we have
	\begin{eqnarray}\label{eq:refx=frex}
		\re f(x)=f(\re x)
	\end{eqnarray}
	for all $x\in M$.
	Denote by  $\operatorname{Hom}_{\O}(M,M') $ the set of all  octonionic linear maps from $M$ to $M'$.
	
	We introduce the notion of Banach $\O$-module \cite{huoqinghai2022Riesz,huo2025BLMSHB}.
	\begin{mydef}\label{def:O-bimod}
		A normed left $\spo$-module consists of a left $\O$-module $M$ paired with a function $\fsh{x}$ mapping elements of $M$ to $\mathbb{R}$, satisfying:
		\begin{itemize}
			\item $\fsh{x} \geqslant 0 $ for all $x\in M$ with equality if and only if $x = 0$.
			\item  $\fsh{p x}=|p|\fsh{x}$ for all $p\in \spo$ and all $x\in M$.
			\item  $\fsh{x+y}\leqslant\fsh{x}+\|{y}\|$ for all $x,y\in M$.
		\end{itemize}
		A normed left $\O$-module becomes a Banach left $\O$-module if it is complete with respect to its natural distance induced by the norm. Similarly, normed right $\O$-modules, normed $\O$-bimodules and Banach $\O$-bimodules can be defined.
	\end{mydef}
	
	\subsection{Slice regular functions}
	Gentili and Struppa \cite{Gentili2006slice, Gentili2007slice} proposed slice regular functions—a new regularity notion for quaternionic functions inspired by Cullen \cite{Cullen1965intrinsic}—which have notable properties and find applications in functional calculus and the mathematical foundations of quantum mechanics \cite{colombo2011noncomfunctcalculus, ghiloni2013slicefct} {which began already in \cite{Viswanath} but without an appropriate notion of spectrum valid in the general case}; later, slice regularity was extended to Clifford algebras \cite{Colombo2009IsraelJM} and octonions \cite{Gentili2010RMJM}, and \cite{Ghiloni2011ADV} presented a slightly different approach, although equivalent on suitable open sets, applicable to all real alternative *-algebras. {This approach is based on the concept of stem functions, see \cite{TransSceBook,Fueter,Ghiloni2011ADV, Sce}.}

	Let us recall that the complexification of $\O$ is the real algebra given by the tensor product $$\O_{\spc}:=\O\otimes_{\R} \spc\simeq \O^2$$that can be described by setting $1 := (1, 0) \in \O^2$ and $ {\imath}:= (0, 1) \in \O^2$, so that every $z=(x, y) \in \O^2 $ can be uniquely written in the form $z = x + y\imath$. Thus the product induced by tensor product  can be described as
	$$(x + y\imath)(x' + y'\imath) = (x x' - yy') + (x y' + yx')\imath.$$
	This makes $\O_{\spc}$ an alternative algebra as well.
	The \textit{complex conjugation} of $ z=x + y\imath\in \O_{\spc}$ is defined by $\overline{z}:=x - y\imath$.
	
	Now let $D$ be a subset of $\spc$, invariant under the complex conjugation. Define  $$\Omega_D := \{r + s J \in \O : r, s \in \R, J\in \mathbb S, r + s\sqrt{-1}\in D\}.$$
	A subset of $\O$ is said to be \textbf{axially symmetric} if it is equal to $\Omega_D$ for some set $D$ as above.
	
	\begin{mydef}[octonionic stem functions and slice functions]
		A function $F=F_1+\imath F_2:D\subseteq \spc \to \O_{\spc}$ is called \textit{stem function} if $F(\overline{z})=\overline{F(z)}$ for all $z\in D$.
		
		The stem function $F=F_1+\imath F_2$ on $D$ induces a left slice function $\mathcal{I}^L(F) : \Omega_D\to \O$ on $\Omega_D$ as follows. Let  $x\in \Omega_D$. By definition of  $\Omega_D$, there exist $r, s \in \R$ and $J\in \mathbb S$ such that $x = r + s J$. Then we set:  $$\mathcal{I}^L(F) (x) := F_1(z) + J F_2(z),$$ where $z = r + s\sqrt{-1}\in  D$. The set of all left slice functions defined on some axially symmetric set $\Omega$ is denoted by
		$\mathcal {S}^L(\Omega)$.
		
		Similarly, a stem function $F=F_1+\imath F_2$ on $D$ also  induces a right slice function $\mathcal{I}^R(F):= F_1(z) +  F_2(z)J $. The set of all right slice functions defined on some axially symmetric set $\Omega$ is denoted by
		$\mathcal {S}^R(\Omega)$.

	\end{mydef}
	It can be verified that	the definition  is well-posed. \textbf{From now on let $\Omega$ be an axially symmetric subset of $\O$.}

	\begin{mydef}\label{def:slice pres}
		Let $F=F_1+\imath F_2$ be a stem function. The left slice function $f = \mathcal{I}^L(F)$ induced by $F$ is said to be \textbf{slice preserving} if $F_1$ and $F_2$ are real-valued.
		The set of all   slice preserving functions in $\mathcal {S}^L(\Omega)$ is denoted by $\mathcal {S}_{\R}(\Omega)$.
	\end{mydef}
	In the literature these functions are also known as	intrinsic functions.
	
	\begin{mydef}[slice product]\label{def:slice product}
		Let $f=\mathcal{I}^L(F),g=\mathcal{I}^L(G)\in \mathcal {S}^L(\Omega)$ be two left slice functions on $\Omega$. We define the \textbf{left slice product} $f\bullet^L g$ as the left slice function $\mathcal{I}^L(FG)$.
		
		In particular, for any $p\in \O$, 
		$$p\bullet^L f=\mathcal{I}^L(pF);\qquad f\bullet^L p=\mathcal{I}^L(Fp).$$
		
		Similar definition of   \textbf{right slice product} $\bullet^R$ can also be given by stem functions. If there are  no confusions, we shall omit  the superscript $L$ or $R$.
	\end{mydef}
	\begin{prop}\label{prop:slicepre}
		\begin{enumerate}
			\item 	A left slice function $f\in \mathcal {S}^L(\Omega)$ is slice preserving if and only if $f (\Omega\cap \spc_J)\subseteq \spc_J$ for every $J\in\mathbb{S}$.
			\item 	Let $f\in \mathcal {S}_{\R}(\Omega),g\in \mathcal {S}^L(\Omega)$, then
			\begin{eqnarray}\label{eq:sliceprod}
				f\bullet g=g\bullet f=fg.
			\end{eqnarray}
		\end{enumerate}

		Similar results hold for right slice functions.
	\end{prop}
	
	It turns out that the set of slice functions forms an alternative algebra \cite{Ghiloni2017TAMStaosf}. Through direct verification, we obtain the following results:
	
	\begin{thm}\label{thm:slice algebra}
		The following properties hold:
		\begin{enumerate}
			
			\item 	Endowed with the slice product, $(\mathcal {S}^L(\Omega),\bullet)$ is an alternative algebra.
			\item With the  scalar  multiplication defined by slice product,	$\mathcal {S}^L(\Omega)$ is an $\O$-bimodule.
			\item The nucleus of  the alternative algebra $(\mathcal {S}^L(\Omega),\bullet)$ is $\mathcal {S}_{\R}(\Omega)$, also  coincides with its commutative center.
			\item The real part of the $\O$-bimodule	$\mathcal {S}^L(\Omega)$ is $\mathcal {S}_{\R}(\Omega)$.
		\end{enumerate}
		
	\end{thm}

	The notion of \textbf{slice topology} is introduced in \cite{Dou2023JEMS}.
	\begin{mydef}
		A subset $U$ of $\O$ is called \textbf{slice-open} if $U\cap \spc_J$ is open in $\spc_J$ for any $J\in\mathbb S$. All slice-open sets form  a topology on $\O$, which is referred to as  the \textbf{slice topology} of $\O$.	
	\end{mydef}
	\begin{mydef}\label{def:s-dom}
		A set $U$ in $\mathbb{O}$ is called a \textbf{slice domain}, for short \textbf{s-domain}, if it is a domain in the Euclidean topology, $\mathbb{R} \cap U \neq \emptyset$, and $U\cap \spc_J$ is a domain in $\mathbb{C}_J$ for any $J \in \mathbb{S}$.
		
	\end{mydef}

	\begin{mydef}[Slice regular functions]\label{def:slice reg}
		Let  $U\subseteq \O$ be a slice-open set. A $C^1$ function $f:U\to \O$  function is called \textbf{left slice regular}	{if for all $J\in \mathbb S$, the restriction $f|_{U\cap \spc_J}$satisfies:  $$\left(\dfrac{\partial}{\partial x}+L_J\dfrac{\partial}{\partial y}\right)f|_{U\cap \spc_J}=0  .$$Here $L_J$  is the left  multiplication operator by $J$.}
		A similar definition holds for right slice regular
		functions by replacing above “left” by “right” and the operator $L_J$ by $R_J$, the right multiplication operator by $J$.
	\end{mydef}

The set of all left and right  slice regular functions on the slice open $U$ are  denoted by $\mathcal{SR}^L(U)$ and $\mathcal{SR}^R(U)$ respectively. The set of
slice preserving  regular functions is denoted by $\mathcal{SR}_{\R}(U)$.	It is well-known that all  slice regular functions  defined on an axially symmetric slice domain are
	slice functions \cite{Gentili2013slice} {in the quaternionic setting. The octonionic setting also  applies.}

	For any $ s\in \O$,	the \textit{characteristic polynomial} of $s$ is the slice preserving slice regular function $Q_{s}$
	\begin{eqnarray}
		Q_{s}(q)=q^2-2(\re s)q+\abs{s}^2.
	\end{eqnarray}
	It follows that  $Q_{s}^{-1}$ is also slice preserving.
	
	Suppose $s=x+yI$ for some $I\in \mathbb S $, and $x,y\in \R$. Then the set of zeroes of $Q_{s}$ is
	$$[s]:=\{x+yJ: J\in \mathbb S\}.$$
	Define the left and right Cauchy kernels as
	\begin{eqnarray}
		{S_L}^{-1}(s,q):&=	Q_{s}(q)^{-1}(\overline{s}-q);\\
		{S_R}^{-1}(s,q):&=	(\overline{s}-q)Q_{s}(q)^{-1}.
	\end{eqnarray}
	
	Let $U=\O\setminus [s]$. It is immediate to verify that  ${S_L}^{-1}(s,\cdot)\in \mathcal{SR}^L(U)$ and ${S_R}^{-1}(s,\cdot)\in \mathcal{SR}^R(U)$. Moreover, Proposition \ref{prop:slicepre} gives
	\begin{eqnarray}
		{S_L}^{-1}(s,q)&=	Q_{s}(q)^{-1}\bullet_q^L (\overline{s}-q);\label{eq:SL-1}\\
		{S_R}^{-1}(s,q)&= (\overline{s}-q)	\bullet_q^R Q_{s}(q)^{-1},\label{eq:SR-1}
	\end{eqnarray}
	where $\bullet_q$ denotes the  slice product of functions with respect to the variable
	$q$.
	As in the quaternionic case, we also have
	\begin{eqnarray}\label{eq:SL=-SR}
		{S_L}^{-1}(s,q)=-{S_R}^{-1}(q,s).
	\end{eqnarray}

	\begin{thm}[Slice Cauchy integral formula \cite{Ghiloni2017caot}]\label{thm:cauchy int}
		Let $D\subseteq \spc$ be a bounded domain {with $D\cap \R\neq \emptyset$}, $J\in \mathbb S$ and $D_J:=\Omega_D\cap \spc_J$. Let $\partial D_J$ denote the boundary of $D_J$ in $\spc_J$ and assume that it is piecewise $C^1$.
		
		If $f\in \mathcal{SR}^L(\Omega_D)\cap C^1(\overline{\Omega_D})$, then
		\begin{eqnarray}\label{eq:lslicecauchy}
			f(q)=\frac{1}{2\pi}\int_{\partial D_J}{S_L}^{-1}(s,q)\bullet_q^L(ds_Jf(s))
		\end{eqnarray}
		for all $q\in \Omega_D$.
		
		If $f\in \mathcal{SR}^R(\Omega_D)\cap C^1(\overline{\Omega_D})$, then
		\begin{eqnarray}\label{eq:rslicecauchy}
			f(q)=\frac{1}{2\pi}\int_{\partial D_J}(f(s)ds_J)\bullet_q^R{S_R}^{-1}(s,q)
		\end{eqnarray}
		for all $q\in \Omega_D$.
		
	\end{thm}
	\begin{cor}\label{cor:slicecauchy}
		Under the same assumptions of Theorem \ref{thm:cauchy int}, if
		$f\in \mathcal{SR}_{\R}(\Omega_D)\cap C^1(\overline{\Omega_D})$, then
		\begin{eqnarray}\label{eq:slicecauchy}
			f(q)=\frac{1}{2\pi}\int_{\partial D_J}{S_L}^{-1}(s,q)ds_Jf(s)=\frac{1}{2\pi}\int_{\partial D_J}f(s)ds_J{S_R}^{-1}(s,q)
		\end{eqnarray}
		for all $q\in \Omega_D$.

	\end{cor}
	\begin{proof}
		Suppose $q\in \spc_I$ for some $I\in \mathbb S$. Denote by $\mathbb H_{I,J}$ the associative subalgebra  of $\O$ generated by $I,J$.
		In view of \eqref{eq:SL-1} and noting $Q_{s}(\cdot)^{-1}$ is slice preserving,  for any $p\in \mathbb H_{I,J}$ we have
		\begin{align*}
			{S_L}^{-1}(s,q)\bullet_q^L p&=	(Q_{s}(q)^{-1}\bullet_q^L (\overline{s}-q))\bullet_q^L p\\
			&=Q_{s}(q)^{-1}\bullet_q^L ((\overline{s}-q)\bullet_q^L p)\\
			&=Q_{s}(q)^{-1} ((\overline{s}-q) p)\\
			&=(Q_{s}(q)^{-1} (\overline{s}-q)) p-[	Q_{s}(q)^{-1} ,\overline{s}-q, p)]\\
			&={S_L}^{-1}(s,q) p.
		\end{align*}
		Combining with formula \eqref{eq:lslicecauchy}, we obtain for 	$f\in \mathcal{SR}_{\R}(\Omega_D)$
		$$f(q)=\frac{1}{2\pi}\int_{\partial D_J}{S_L}^{-1}(s,q)ds_Jf(s),\qquad \text{$q\in \Omega_D,$}$$ as desired.
		The second formula can be proved with a similar reasoning.
	\end{proof}
	
	\subsection{Para-Linear Maps}

	Para-linearity generalizes linearity to the non-associative setting by constraining the associator's real part. The notion of para-linearity is first introduced  in \cite{huoqinghai2022Riesz} and then extended into a more general setting in \cite{huoqinghai2020nonass}.
	
	\begin{mydef}[Right Para-Linear Map]
		\label{def:para_linear}
		Let $M$ be a right $\mathbb{O}$-module and $M'$ an $\mathbb{O}$-bimodule. A map $f \in \operatorname{Hom}_{\mathbb{R}}(M,M')$ is \textbf{right para-linear} if
		\[
		\operatorname{Re} B_p(f,x) = 0, \quad  p \in \mathbb{O}, \, x \in M,
		\]
		where the \textbf{second right associator} is
		\[
		B_p(f,x) := f(x)p - f(xp),
		\]
		and $\operatorname{Re}: M' \to \operatorname{Re} M'$ is the real part operator. The set of right para-linear maps is denoted by $\operatorname{Hom}_{\mathcal{RO}}(M,M')$. Left para-linearity is defined analogously and the set of left para-linear maps is denoted by $\operatorname{Hom}_{\mathcal{LO}}(M,M')$.
	\end{mydef}
	
	For any $f \in \operatorname{Hom}_{\mathbb{R}}(M,M')$, we denote
	\begin{eqnarray}\label{eqdef:fR}
		f_{\R}(x):=\re f(x).
	\end{eqnarray}
	
	\begin{thm}[Para-Linearity Characterization I]
		\label{thm:para_linear_char}
		Let $M$ be a right $\mathbb{O}$-module and $M'$ an $\mathbb{O}$-bimodule. For $f \in \operatorname{Hom}_{\mathbb{R}}(M,M')$ decomposed as $f(x) = f_{\mathbb{R}}(x) + \sum_{i=1}^7 f_i(x)e_i$, the following are equivalent:
		\begin{enumerate}
			\item $f \in \operatorname{Hom}_{\mathcal{RO}}(M,M')$,
			\item $f_i(x) = -f_{\mathbb{R}}(x e_i)$ for $i=1,\dots,7$,
			\item $B_p(f,x) = \sum_{i=1}^7 f_{\mathbb{R}}([x,p,e_i])e_i$ for all $p\in \O$.
		\end{enumerate}
		
	\end{thm}
	
	Recall the definition \eqref{eq:CJM} of  ${\mathbb C_J(M)} $. The following result indicates that the para-linearity can be viewed as ``slice (complex) linearity'' in some sense.
	\begin{thm}[Para-Linearity Characterization II]\label{lem:paralinear-char2}
		Let $M,M'$ be two $\O$-bimodules and $f \in \operatorname{Hom}_{\mathbb{R}}(M,M')$. Then the following are equivalent:
		\begin{enumerate}
			\item $f\in \operatorname{Hom}_{\mathcal{RO}}(M,M')$;
			\item $B_p{(f,x)}=0  \text{ for all } p\in \O \text{ and all } x\in \re M$;
			\item for any $J\in \mathbb S$, $B_p(f,z)=0$ for all $p\in \mathbb C_J$ and all $z\in \mathbb C_J(M) $;
			\item for any $J\in \mathbb S$,  $f|_{\mathbb C_J(M)}$ is right $\mathbb C_J$-linear, i.e., $$f(zp)=f(z)p$$for all $p\in \mathbb C_J$ and all $z\in \mathbb C_J(M) $.
		\end{enumerate}
		
	\end{thm}
	\begin{proof}
		(1) $\implies$ (2):
		If  $f \in \operatorname{Hom}_{\mathcal{RO}}(M,M')$, then for any $x\in \re M$, it follows from Theorem \ref{thm:para_linear_char} that $$B_p(f,x) = \sum_{i=1}^7 f_{\mathbb{R}}([x,p,e_i])e_i=0.$$
		
		(2) $\implies$ (1):
		For any $x=\sum_{i=0}^7x_ie_i$ with $x_i\in \re M$,   we have
		$$\re B_p{(f,x)}=\re B_p{(f,\sum_{i=0}^7x_ie_i)}=\re \sum_{i=0}^7[f(x_i),e_i,p]=0.$$
		This proves  $f \in \operatorname{Hom}_{\mathcal{RO}}(M,M')$ as desired.
		
		(2) $\implies$ (3): Let $z=z_0+z_1J\in \mathbb C_J(M)$ with $z_0,z_1\in \re (M)$.
		Then by part (2), we obtain
		$$B_p(f,z)=B_p(f,z_1J)=f(z_1J)p-f((z_1J)p)=(f(z_1)J)p-f(z_1(Jp))=[f(z_1),J,p]=0.$$

		It is  obvious that part (3) implies (2). That  part (3) is equivalent to (4) is trivial. This completes the proof.
	\end{proof}
	The second associators satisfy some identities (see \cite{huoqinghai2020nonass} for the second left associators  case and \cite{huoqinghai2025Oselfadjoint} for the right case.)
	\begin{lemma}\label{lem:secass prop}
		Let $M,M'$ be two $\O$-bimodules and $f \in \operatorname{Hom}_{\mathbb{R}}(M,M')$. Then  for any $r\in \O$, we have
		\begin{eqnarray}
			B_r(f,x)&=&-B_{\overline{r}}(f,x),\label{eq:Br=-Brbar}\  \\
			B_r(f,xr)&=&B_r(f,x)\overline{r}\ =\ rB_r(f,x) ,\label{eq:B1}\\
			B_r(f,x)r&=&B_r(f,rx).\label{eq:Brtxr}
		\end{eqnarray}
	\end{lemma}

	For later use, we define  for any  $J\in \mathbb S$ the \textbf{$\spc_J$-projection operator} $\pi_J$ on an $\O$-bimodule $M$ as follows:
	\begin{eqnarray}\label{eqdef:piJ}
		\pi_J:M&\to& \spc_J(M)\\
		x&\mapsto& \re x+J\re (\overline{J}x).\notag	
	\end{eqnarray}
	\begin{prop}
		Let $M$ be an $\O$-bimodule and 	$J\in \mathbb S$. the projection operator $\pi_J$ is both left and right $ \spc_J$-linear, i.e.,
		\begin{eqnarray}\label{eq:piJ}
			\pi_J(xs)=\pi_J(x)s=s\pi_J(x)=\pi_J(sx), \qquad x\in M, s\in \spc_J.
		\end{eqnarray}
	\end{prop}
	\begin{proof}
		Since $\pi_J$ is real linear, we can just assume that $s=J$.
		This can be checked directly.
	\end{proof}
	
	We present the following useful lemma regarding the properties of $\pi_J$.
	\begin{lemma}\label{lem:piJBpfx=0}
		Let $M,M'$ be two $\mathbb{O}$-bimodules and  $f\in \operatorname{Hom}_{\mathcal{RO}}(M,M')$. Then for any  $J\in \mathbb S$ and $x\in M$, we have
		\begin{eqnarray}\label{eqpf:piJB=0}
			\pi_J(B_{p}(f,x))=0, \qquad x\in M,  p\in \spc_J.
		\end{eqnarray}
	\end{lemma}
	\begin{proof}
		Let $p=p_0+p_1J\in \spc_J$ with $p_0,p_1\in \R$, it follows from Lemma \ref{lem:secass prop} that
		\begin{eqnarray*}
			\pi_J(B_{p}(f,x))=\pi_J (\overline{J}B_{p_1J}(f,x))=-p_1\re (JB_{J}(f,x))=0.
		\end{eqnarray*}
		This completes the proof.
	\end{proof}

	The following Uniqueness Lemma will be used several  times in the sequel.
	\begin{lemma}[Uniqueness Lemma \cite{huoqinghai2020nonass}]\label{lem:Uniqueness Lemma}
		\label{cor: A_p(x,f)=0 and f(px)=pf(x)}
		Let $M,M'$ be two $\mathbb{O}$-bimodules	and  $f \in \operatorname{Hom}_{\mathcal{RO}}(M,M')$. Then:
		\begin{enumerate}
			\item $f_{\mathbb{R}} = 0$ if and only if  $f = 0$.
			\item  $f|_{\operatorname{Re} M} = 0$ if and only if  $f = 0$.
		\end{enumerate}
	\end{lemma}
	Uniqueness Lemma \ref{lem:Uniqueness Lemma} gives rise to two bijections  \cite{huoqinghai2020nonass}.
	\begin{mydef}[lifting and extension maps]\label{def:lif and ext}
		Let $M,M'$ be two $\mathbb{O}$-bimodules.	The (right) lifting map is defined by
		\begin{eqnarray}\label{eq:lif}
			\operatorname{lif}_R:\operatorname{Hom}_{\R}(M,\re M')&\to& \operatorname{Hom}_{\mathcal{RO}}(M,M')\notag\\
			f&\mapsto &(\operatorname{lif}_R f)(x):=\sum_{i=0}^7f(x\overline{e_i})e_i
		\end{eqnarray} and the  (right) extension map is defined by
		\begin{eqnarray}\label{eq:ext}
			\operatorname{ext}_R:\operatorname{Hom}_{\R}(\re M,M')&\to& \operatorname{Hom}_{\mathcal{RO}}(M,M')\notag\\
			f&\mapsto& (\operatorname{ext}_R f)\left(\sum_{i=0}^7x_ie_i\right):=\sum_{i=0}^7f(x_i)e_i.
		\end{eqnarray}
		Similarly, we can also define the left lifting map $$\operatorname{lif}_L:\operatorname{Hom}_{\R}(M,\re M')\to \operatorname{Hom}_{\mathcal{LO}}(M,M')$$ and  the left extension map
		$$	\operatorname{ext}_L:\operatorname{Hom}_{\R}(\re M,M')\to \operatorname{Hom}_{\mathcal{LO}}(M,M').$$
		Unless explicitly specified, we shall use the abbreviated form 	$\operatorname{lif}$ and $	\operatorname{ext}$ instead of $\operatorname{lif}_R$ and $	\operatorname{ext}_R$, respectively.
	\end{mydef}

	\begin{thm}[$\mathbb{O}$-Bimodule Structure \cite{huoqinghai2020nonass}]
		\label{thm:para_bimodule}
		Let  $M,M'$ be two $\mathbb{O}$-bimodules. The space $\operatorname{Hom}_{\mathcal{RO}}(M,M')$ is an $\mathbb{O}$-bimodule with respect to the following multiplications with $p\in\mathbb{O}$:
		\begin{eqnarray}
			(p \odot f)(x) :&= p f(x) + B_p(f,x),\label{eqdef:pf} \\
			(f \odot p)(x) :&= f(px) - B_p(f,x).\label{eqdef:fp}
		\end{eqnarray}
		
		Moreover, \begin{eqnarray}
			\re \operatorname{Hom}_{\mathcal{RO}}(M,M')= \operatorname{Hom}_{\O}(M,M')
		\end{eqnarray}
		and
		\begin{eqnarray}\label{eq:re f}
			\re_{\operatorname{Hom}_{\mathcal{RO}}(M,M')}  f = \ext ( f_{\R}|_{\re M}),\qquad f\in \operatorname{Hom}_{\mathcal{RO}}(M,M').
		\end{eqnarray}
	\end{thm}

	\begin{rem}
		In view of Theorem \ref{thm:para_bimodule}, we conclude that a right para-linear $f$ is associative in the bimodule $\Hom_{\mathcal {RO}}(M',M'')$ (see Definition \ref{def:real part}) if and only if $f$ is octonionic linear.
	\end{rem}

	We next introduce the notion of the  regular composition for para-linear maps. This notion has been given in \cite{huoqinghai2020nonass} and \cite{huoqinghai2025Oselfadjoint}. We aim  to give a modification to this notion which is more natural in the sequel. Denote by $\circ  $ the classical  composition as usual.  To this end, we firstly give the following lemma.
	
	\begin{lemma}\label{lem:lif=ext}
		Let $M,\ M',\ M''$  be  $\spo$-bimodules. If  $f\in  \Hom_{\mathcal{RO}}(M',M'')$ and $g\in  \Hom_{\mathcal{RO}}(M,M')$, then $$\lif (f\circ g)_{\R}=\ext (f\circ g)|_{\re M}.$$
		
	\end{lemma}
	\begin{proof}
		Let $f\in  \Hom_{\mathcal{RO}}(M',M'')$ and $g\in  \Hom_{\mathcal{RO}}(M,M')$.	For any $x\in \re M$, in view of Definition \ref{def:lif and ext}  and Theorem   \ref{lem:paralinear-char2}, we have
		\begin{eqnarray*}
			\lif (f\circ g)_{\R}(x)&=&\re f(g(x))+\sum_{i=1}^7e_i\re f(g(x\overline{e_i}))\\
			&=&\re f(g(x))+\sum_{i=1}^7e_i\re f(g(x)\overline{e_i})\\
			&=&(\lif f_{\R})(g(x))\\
			&=&f(g(x))\\
			&=&\ext (f\circ g)|_{\re M}(x).
		\end{eqnarray*}
		By  Uniqueness Lemma \ref{lem:Uniqueness Lemma}, we prove the assertion.
	\end{proof}

	By Lemma \ref{lem:lif=ext}, we give the following notion of the regular composition.
	\begin{mydef}\label{def:right mod regular composition}
		Let $M,\ M'$,  $\ M''$ be  $\spo$-bimodules.
		We define the \textbf{(right) regular composition} as
		$$	f \circledcirc_R g:=
		\begin{cases}
			\lif (f\circ g)_{\R}&f\in  \Hom_{\mathcal {RO}}(M',M''),\, g\in  \Hom_{\spr}(M,M');\\
			\ext (f\circ g)|_{\re M}&f\in  \Hom_{\R}(M',M''),\, g\in  \Hom_{\mathcal {RO}}(M,M'),
		\end{cases}
		$$
		and define the \textbf{(right) composition associator}  as
		$$ [f,g,x]_{\circledcirc_R}:=(f\circledcirc_R g)\,(x)-(f\circ g)(x).$$
		It is easy to verify
		\begin{equation}\label{eq:[fgx]}
			[f,g,x]_{\circledcirc_R}=
			\begin{cases}
				\sum_{i=1}^7 e_i \operatorname{Re}\big(f(B_{e_i}(g, x))\big)&f\in  \Hom_{\mathcal {RO}}(M',M''),\, g\in  \Hom_{\spr}(M,M');\\
		{\sum_{i=1}^7 B_{e_i}(f, g(\re (x\overline{e_i}))) }&f\in  \Hom_{\R}(M',M''),\, g\in  \Hom_{\mathcal {RO}}(M,M').
			\end{cases}
		\end{equation}
		If there is no confusion, we shall  omit  the subscript $R$  for simplicity.
	\end{mydef}
	
	Identity \eqref{eq:[fgx]}  implies that if one of $f,g,x$ is associative (see Definition \ref{def:real part} for the notion of associative element), then the associator $[x,f,g]_{\circledcirc}$ vanishes.
	\begin{lemma}[Vanishing Composition Associator]\label{lem:vanishing[f,g,x]=0}
		{Let $M,\ M'$,  $\ M''$ be  $\spo$-bimodules.}
		Let  $f\in\Hom_{\mathcal {RO}}(M',M'')$ and  $g\in  \Hom_{\mathcal {RO}}(M,M')$. We have $\re [x,f,g]_{\circledcirc} = 0$.
		Moreover, the composition associator $[x,f,g]_{\circledcirc} = 0$ and hence $(f \circledcirc g)(x) = f(g(x))$ under any of the following:
		\begin{enumerate}
			\item $g\in \Hom_{\spo}(M,M')$;
			\item  $x\in \re{M}$;
			\item $f\in  \Hom_{\spo}(M',M'')$.

		\end{enumerate}
	\end{lemma}
	
		%
		%
		
		\section{Octonionic Para-linear bounded operators}

		In this section we assume that \textbf{$V$ is a Banach $\O$-bimodule} (see Definition \ref{def:O-bimod}). Let $\mathscr{B}_{\R}(V)$ be the set of all bounded real linear operators. For any real linear operators $T,S\in \mathscr{B}_{\R}(V)$, we denote their standard composition, consistent with the classical case, by
		$$T\circ S, \text{  or simply } TS.$$
		Denote by $\mathscr{B}_{\O}(V)$ the set of all bounded octonionic linear operators  and  $\mathscr{B}_{\mathcal{RO}}(V)$ the set of all bounded right  octonionic para-linear operators respectively. For any $T\in \mathscr{B}_{\R}(V)$, the norm $\fsh{T}$ of $T$ is defined as in classical case which is actually independent with the octonionic scalar:
		$$\fsh{T}:=\sup_{\fsh{x}=1} \fsh{Tx}.$$
		
		Recall the definition of \textit{$\mathbb{O}$-algebra} introduced in \cite{huoqinghai2025Oselfadjoint}:
		\begin{mydef}
			\label{def:O alg}
			An $\mathbb{O}$-bimodule $\mathcal{U}$ with multiplication $\mathcal{U} \times \mathcal{U} \to \mathcal{U}$ is an \textit{$\mathbb{O}$-algebra} if:
			\begin{enumerate}
				\item The product  is para-bilinear, i.e., the left (right) product operator is right (left) para-linear;
				\item There exists an {element} $ e\in \operatorname{Re}\mathcal{U}$, {called  the \textbf{unit} of $\algma$,} such that $$(pe)x=px,\qquad  x(pe)=xp$$
				for all $x\in \algma$ and  $ p\in\spo$.
			\end{enumerate}
		\end{mydef}

	Using the same approach as the proof in the Hilbert case \cite{huoqinghai2025Oselfadjoint}, we can also prove the following theorem, and we omit its proof for brevity.
		\begin{thm}\label{thm:BROV is Oalg}
			Let $V$ be a Banach $\O$-bimodule. Then	$(\mathscr{B}_{\mathcal{RO}}(V),\circledcirc)$ forms an $\O$-algebra.
		\end{thm}
		
		\begin{lemma}\label{lem:tensor of BROV}
			Let $V$ be a Banach $\O$-bimodule. Then	as $\O$-algebras, we have	$$\mathscr{B}_{\mathcal{RO}}(V)\simeq \mathscr{B}_{\mathbb{O}}(V)\otimes \O.$$
		\end{lemma}
		\begin{proof}
			Let $T,S\in  \mathscr{B}_{\mathbb{O}}(V)$ be two octonionic linear operators. Note that by Theorem \ref{thm:para_bimodule}, the real part of the bimodule $\mathscr{B}_{\mathcal{RO}}(V)$, which coincides with its commutative center and associative center,  is $\mathscr{B}_{\mathbb{O}}(V)$. For any $p,q\in \O$, it follows that
			\begin{align*}
				(p\odot T)\circledcirc (q\odot S)&=	(p\odot T)\circledcirc (S\odot q)\\
				&=((p\odot T)\circledcirc S)\odot q\\
				&=(p\odot (T\circledcirc S))\odot q,
			\end{align*}
			where the second  equality follows from the right para-linearity of the map of the left product
			$$L_{p\odot T}:\mathscr{B}_{\mathbb{O}}(V)\to \mathscr{B}_{\mathbb{O}}(V), \qquad R\mapsto (p\odot T)\circledcirc R,$$  and the third equality follows from the left para-linearity of the map of the of the  right product induced by $S$.
			It is easy to check that $T\circledcirc S=TS$ is also  octonionic linear.
			Hence $$	(p\odot (T\circledcirc S))\odot q=	( (T\circledcirc S)\odot p)\odot q= (T S)\odot (p q)=(pq)\odot(TS).$$
			Therefore \begin{eqnarray}\label{eqpf:tensor BROV}
				(p\odot T)\circledcirc (q\odot S)=(p q)\odot (T S).
			\end{eqnarray}
			This induces an $\O$-algebra isomorphism 	$\mathscr{B}_{\mathcal{RO}}(V)\simeq \mathscr{B}_{\mathbb{O}}(V)\otimes \O$.
		\end{proof}

		\begin{mydef}\label{def:invertible}
			A real linear operator		$T\in \mathscr{B}_{\R}(V)$ is called \textbf{invertible} if there exists $S\in \mathscr{B}_{\R}(V)$ such that
			$$TS=ST=\mathcal{I}.$$  Such  $S$  is denoted by $T^{-1}$ as usual, which is the inverse of $T$ in the {real}  algebra $(\mathscr{B}_{\R}(V),\circ)$.
			$T\in \mathscr{B}_{\R}(V)$ is called \textbf{bounded invertible} if $T$ is invertible and $T^{-1}$ is bounded.
		\end{mydef}

		We next discuss a key notion, the \textit{regular inverse} of an  operator, which plays the crucial role in the octonionic Riesz-Dunford theory.
		
		{To introduce that notion, we recall that for any $T\in \mathscr{B}_{\R}(V)$, we set
			$$T_{\R}(x):=\re T(x),$$
			moreover, below we shall write ${T^n}|_{\re V}$ in place of $( {T^n})|_{\re V}$.
		}

		\begin{mydef}[Regular powers and regular inverses]\label{def:ncirc}
			For any real linear operator $T\in \mathscr{B}_{\R}(V)$, define for $n\in \mathbb N$ the left and right  \textbf{regular powers} of $T$, respectively,
			$$T^{n\circledcirc}:=\lif (T^n)_{\R},\qquad T^{\circledcirc n}:=\ext (T^n|_{\re V}).$$
			
			If  $T$ is invertible in $\mathscr{B}_{\R}(V)$,  for $n\in \mathbb N$, we define
			$$T^{(-n)\circledcirc}:=(T^{-1})^{n\circledcirc},\qquad T^{\circledcirc (-n)}:=(T^{-1})^{\circledcirc n}.$$
			And denote $$ T^{-\circledcirc}:=T^{(-1)\circledcirc}, \qquad T^{\circledcirc -}:=T^{\circledcirc (-1)}.$$
			$T^{-\circledcirc}$ is called the \textbf{left regular inverse} of $T$ and 	$T^{\circledcirc -}$ is called the \textbf{right regular inverse} of $T$, respectively.
			{For $n=0$ both the regular powers equal the identity $\mathcal I$.}
		\end{mydef}
		{We note that, by definition, the regular powers of $T\in\mathscr{B}_{\mathbb{R}}(V)$ belong to $\mathscr{B}_{\mathcal{RO}}(V)$.}
		\begin{lemma}\label{lem:circledcirc}
			
			Let $T\in \mathscr{B}_{\R}(V)$. For any $n\in \mathbb N $, we have
			\begin{eqnarray}
				T^{n\circledcirc}&=&T^{(n-1)\circledcirc}\circledcirc T;\label{eq:Tnlif}\\
				T^{\circledcirc n}&=&T\circledcirc T^{\circledcirc (n-1)}.\label{eq:Tnext}
			\end{eqnarray}
			Further if $T$ is invertible in $\mathscr{B}_{\R}(V)$, then these  also hold for  $n\leqslant 0$.
		\end{lemma}

		\begin{proof}
			Recall the notation   in \eqref{eqdef:fR}.
			By Definition \ref{def:ncirc}, we have $$({T^{(n-1)\circledcirc}})_{\R}=({T^{n-1}})_{\R}.$$
			Noting $T^{(n-1)\circledcirc}\in \mathscr{B}_{\mathcal{RO}}(V)$ and $T \in \mathscr{B}_{\R}(V)$, it follows from  Definition \ref{def:right mod regular composition} that
			\begin{align*}
				T^{n\circledcirc}&=\lif ({T^n})_{\R}\\
				&=\lif ({{T^{n-1}})_{\R}\circ T}\\
				&=\lif ({{T^{(n-1)\circledcirc}})_{\R}\circ T}\\
				&=\lif ({T^{(n-1)\circledcirc}\circ T})_{\R}\\
				&=T^{(n-1)\circledcirc}\circledcirc T.
			\end{align*}
			
			Similarly,
			\begin{align*}
				T^{\circledcirc n}&=\ext( {T^n}|_{\re V})\\
				&=\ext (T\circ {T^{n-1}}|_{\re V})\\
				&=\ext (T\circ {T^{\circledcirc(n-1)}}|_{\re V})\\
				&= T\circledcirc {T^{\circledcirc(n-1)}}.
			\end{align*}
			%
			
			If $T$ is bounded invertible, the proof is also valid for  $n\leqslant 0$.
		\end{proof}
		\begin{rem}\label{rem:regular inv}
			By Lemma \ref{lem:circledcirc}, we conclude that if $T\in \mathscr{B}_{\R}(V)$ is bounded invertible, then
			\begin{eqnarray}
				T^{{-\circledcirc}}\circledcirc T&=\mathcal I\\
				T\circledcirc T^{\circledcirc-}&=\mathcal I.
			\end{eqnarray}
			This justifies the name of left and right regular inverses of $T^{{-\circledcirc}}$ and $T^{{\circledcirc-}}$, respectively.
		\end{rem}
		
		\section{The resolvent operator series identities}
		In this section, we establish some resolvent operator series identities. To this end, below $V$ denotes a Banach $\O$-bimodule, and $R_s:V\to V$ is the real linear map $v\mapsto vs$ for any $s\in \O$. Note that,  in general, $R_s$ is not a right para-linear operator.
		\begin{lemma}\label{lem:Rs-Tinv}
			Let $T\in \mathscr{B}_{\R}(V)$ and $s\in \O$ such that $|s|>\fsh{T}.$
			Then $R_s-T$ is invertible {in $\mathscr{B}_{\R}(V)$}.
		\end{lemma}
		\begin{proof}
			It is easy to verify that $\fsh{R_{s^{-1}}T}=\abs{s^{-1}}\fsh{T}<1$.
			Hence the operator
			\begin{align*}
				R_{s^{-1}}(R_s-T)=\mathcal I-R_{s^{-1}}T
			\end{align*}
			is invertible as a real linear operator.
			This implies that $R_s-T$ is invertible. More precisely, we have
			$$ (R_s-T)^{-1}=(\mathcal I-R_{s^{-1}}T)^{-1}{R_s}^{-1}$$
			where we use the fact that $R_sR_{s^{-1}}T=T$.
		\end{proof}

		Recall the notion of $\spc_J$-projection operator $\pi_J$ \eqref{eqdef:piJ} and 	 the definition \eqref{eq:CJM} of  ${\mathbb C_J(V)} $.	 We now formulate a key result, namely resolvent operator series identities.
		\begin{thm}\label{thm:Rs-T}
			Let $T\in \mathscr{B}_{\mathcal{RO}}(V)$ and let $s\in \spc_J\subseteq\O$ for some $J\in \mathbb S$ be such that $|s|>\fsh{T}.$
			\begin{enumerate}
				\item 	For all $x\in \spc_J(V)$, we have
				\begin{eqnarray}\label{eq:1resolveformula}
					(R_s-T)\sum_{n\geqslant 0}\left(T^{\circledcirc n}\odot s^{-1-n}\right)(x)=x+\alpha(s,T)(x).
				\end{eqnarray}
				Here $\alpha(s,T):\spc_J(V)\to V$ is defined by
				\begin{eqnarray}
					\alpha(s,T)(x):=\sum_{n\geqslant 0}[T,T^{\circledcirc n}, xs^{-1-n}]_{\circledcirc}.
				\end{eqnarray}
				\item For all $x\in V$, we have
				\begin{eqnarray}\label{eq:2resolveformula}
					\pi_J \sum_{n\geqslant 0}(s^{-1-n}\odot T^{n\circledcirc})(R_s-T)(x)=\pi_J x+\beta(s,T)(x).
				\end{eqnarray}
				Here $\beta(s,T):V\to \spc_J(V)$ is defined by
				\begin{eqnarray}
					\beta(s,T)(x):=\pi_J \sum_{n\geqslant 0}[T^{n \circledcirc},T,x]_{\circledcirc}s^{-1-n}.
				\end{eqnarray}
			\end{enumerate}
			
			All the series converge in norm.
		\end{thm}
		\begin{proof}
			\begin{enumerate}
				\item 	Fix $x\in \spc_J(V)$ arbitrarily, then Theorem \ref{lem:paralinear-char2} implies that $B_{p}(x,f)=0$ for any  right para-linear map $f$ and any $p\in\spc_J$. Hence by the  definition of scalar multiplication of para-linear operators in Theorem \ref{thm:para_bimodule}, we get for any $s\in \spc_J$ that
				\begin{eqnarray}\label{eqpf:Tncirc}
					( T^{\circledcirc n}\odot s^{-1-n})(x)=  T^{\circledcirc n}(s^{-1-n}x)=T^{\circledcirc n}(xs^{-1-n})=T^{\circledcirc n}(x)s^{-1-n}.
				\end{eqnarray}
				
				Since $\fsh{\re v}\leqslant \fsh{v}$ for all $v\in V$ and $x=\pi_J x=\re x+J\re\overline{J}x $ for $x\in\mathbb{C}_J(V)$, we obtain
				\begin{align*}
					\fsh{T^{\circledcirc n}(x)}&=\fsh{T^{\circledcirc n}(\re x)+T^{\circledcirc n}(\re \overline{J}x)J}\\
					&\leqslant\fsh{T^{\circledcirc n}(\re x)}+\fsh{T^{\circledcirc n}(\re \overline{J}x)}\\
					&=\fsh{T^{ n}(\re x)}+\fsh{T^{ n}(\re \overline{J}x)}\\
					&\leqslant2\fsh{T}^n\fsh{x}.
				\end{align*}
				Thus combining with \eqref{eqpf:Tncirc},  we have $$\fsh{( T^{\circledcirc n}\odot s^{-1-n})(x)}=\fsh{ T^{\circledcirc n}(x)}\abs{s^{-1-n}}\leqslant 2\fsh{T}^n\fsh{x}\abs{s}^{-1-n},$$
				so the series $\sum_{n\geqslant 0}\left(T^{\circledcirc n}\odot s^{-1-n}\right)(x)$ converges for $|s|>\fsh{T}$.
				\\	
				Similarly, keeping identity  \eqref{eqpf:Tncirc} in mind, we deduce
				\begin{align*}
					\fsh{[T,T^{\circledcirc n}, xs^{-1-n}]_{\circledcirc}}&=\fsh{(T\circledcirc T^{\circledcirc n})(x s^{-1-n})-T(T^{\circledcirc n}(x s^{-1-n}))}\\
					&\leqslant\fsh{T^{\circledcirc (n+1)}(xs^{-1-n})}+\fsh{T}\fsh{T^{\circledcirc n}(x s^{-1-n})}\\
					&=\fsh{T^{\circledcirc (n+1)}(x)s^{-1-n}}+\fsh{T}\fsh{T^{\circledcirc n}(x )s^{-1-n}}\\
					&\leqslant 4\fsh{T}^{ n+1}\fsh{x}\abs{s}^{-1-n}.
				\end{align*}
				This proves the convergence of the series $\sum_{n\geqslant 0}[T,T^{\circledcirc n}, xs^{-1-n}]_{\circledcirc}$ for $|s|>\fsh{T}$.
				\\	
				Next, we prove the equality \eqref{eq:1resolveformula}. In view of \eqref{eqpf:Tncirc}, we have
				\begin{align*}
					&(R_s-T)\sum_{n\geqslant 0}\left(T^{\circledcirc n}\odot s^{-1-n}\right)(x)\\=&	R_s\sum_{n\geqslant 0}T^{\circledcirc n}(x)s^{-1-n}-T\sum_{n\geqslant 0}T^{\circledcirc n}(xs^{-1-n})\\
					=&\sum_{n\geqslant 0}T^{\circledcirc n}(x)s^{-n}-\sum_{n\geqslant 0}\left((T\circledcirc T^{\circledcirc n})(xs^{-1-n})-[T,T^{\circledcirc n}, xs^{-1-n}]_{\circledcirc}\right)\\
					=&\sum_{n\geqslant 0}T^{\circledcirc n}(x)s^{-n}-\sum_{n\geqslant 0} T^{\circledcirc (n+1)}
					(x)s^{-1-n}+\sum_{n\geqslant 0}[T,T^{\circledcirc n}, xs^{-1-n}]_{\circledcirc}\\
					=&x+\alpha(s,T)(x).
				\end{align*}
				\item
				Note that by the definition of scalar of  para-linear operators in Theorem \ref{thm:para_bimodule}, for any $x\in V$, we have
				\begin{eqnarray}\label{oldeq:sodotT}
					\re (s^{-1-n}\odot T^{n\circledcirc})(x)&=&\re( s^{-1-n} T^{n\circledcirc}(x))\\&=&\re(  T^{n\circledcirc}(x)s^{-1-n})\notag \\&=&\re T^{n\circledcirc}( xs^{-1-n})\notag\\&=&\re T^{n}( xs^{-1-n}).\notag
				\end{eqnarray}

				By  Theorem \ref{thm:para_linear_char},  we conclude from \eqref{oldeq:sodotT} that
				$$(s^{-1-n}\odot T^{n\circledcirc})(x)=\sum_{i=0}^7e_i\re (s^{-1-n}\odot T^{n\circledcirc})(x\overline{e_i})=\sum_{i=0}^7e_i\re  T^{n}((x\overline{e_i})s^{-1-n}).$$
				Thus \begin{align*}
					\fsh{(s^{-1-n}\odot T^{n\circledcirc})(x)}&\leqslant\sum_{i=0}^7\fsh{\re  T^{n}((x\overline{e_i})s^{-1-n})}\\
					&\leqslant\sum_{i=0}^7\fsh{ T^{n}((x\overline{e_i})s^{-1-n})}\\
					&\leqslant\sum_{i=0}^7\fsh{ T}^{n}\fsh{(x\overline{e_i})s^{-1-n}}\\
					&=8\fsh{ T}^{n}\fsh{x}\abs{s^{-1-n}}.
				\end{align*}
				This proves the convergence of the series  $\sum_{n\geqslant 0}(s^{-1-n}\odot T^{n\circledcirc})(x)$ for any $x\in V$.
				\\		
				Similarly, one can verify that
				$$\fsh{[T^{n \circledcirc},T,x]_{\circledcirc}}\leqslant 16 \fsh{T}^{n+1}\fsh{x}.$$
				This shows the convergence of the series  $\sum_{n\geqslant 0}[T^{n \circledcirc},T,x]_{\circledcirc}s^{-1-n}$ for any $x\in V$.
				\\			
				Finally, 	we prove the equality \eqref{eq:2resolveformula}.
				Hence by the  definition of scalar multiplication of para-linear operators in Theorem \ref{thm:para_bimodule} and \eqref{eqpf:piJB=0} in Lemma \ref{lem:piJBpfx=0}, we have
				$$	\pi_J (s^{-1-n}\odot T^{n\circledcirc})(x)=	\pi_J (s^{-1-n} T^{n\circledcirc}(x))+\pi_JB_{s^{-1-n}}(T^{n\circledcirc},x)=\pi_J (s^{-1-n} T^{n\circledcirc}(x)).$$
				In view of properties of $\pi_J$ in \eqref{eq:piJ} and \eqref{eqpf:piJB=0}, we have
				\begin{eqnarray}\label{eq:sodotT}
					\pi_J\, (s^{-1-n}\odot T^{n\circledcirc})(x)&=&\pi_J\, (s^{-1-n} T^{n\circledcirc}(x))\\
					&=&\pi_J\, (T^{n\circledcirc}(x)s^{-1-n})\notag \\&=&\pi_J\, T^{n\circledcirc}( xs^{-1-n}).\notag
				\end{eqnarray}
				Thanks to \eqref{eq:sodotT} and \eqref{eqpf:piJB=0}, we deduce
				\begin{align*}
					&	\pi_J \sum_{n\geqslant 0}(s^{-1-n}\odot T^{n\circledcirc})(R_s-T)(x)\\
					=&	\pi_J \sum_{n\geqslant 0}(s^{-1-n}\odot T^{n\circledcirc})(xs-Tx)\\
					=&	\pi_J \sum_{n\geqslant 0} T^{n\circledcirc}(xs^{-n})-(T^{n\circledcirc}(Tx))s^{-1-n}+B_{s^{-1-n}}(T^{n\circledcirc},Tx)\\
					=&	\pi_J \sum_{n\geqslant 0} T^{n\circledcirc}(x)s^{-n}-\big((T^{n\circledcirc}\circledcirc T) (x)-[T^{n\circledcirc},T,x]_\circledcirc\big)s^{-1-n}\\
					=&		\pi_J \sum_{n\geqslant 0} T^{n\circledcirc}(x)s^{-n}-T^{{(n+1)}\circledcirc} (x)s^{-1-n}+[T^{n\circledcirc},T,x]_\circledcirc s^{-1-n}\\
					=&	\pi_J  x+\beta(s,T)(x).
				\end{align*}
			\end{enumerate}	
			This completes the proof.
		\end{proof}

		Recall the notion of left and right regular inverse of $T$    in Definition \ref{def:ncirc}. Below we prove that  for  $|s|>\fsh{T}$ the  resolvent operator series converges to  the left and right regular inverse of $R_s-T$ provided the associator terms $\alpha(s,T)$ and $\beta(s,T)$ vanish, respectively.
		\begin{thm}
			Let $T\in \mathscr{B}_{\mathcal{RO}}(V)$  and   $s\in \spc_J\subseteq \O$ for some $J\in \mathbb S$ be such that $|s|>\fsh{T}$.
			\begin{enumerate}
				\item  If $\alpha(s,T)$ in Theorem \ref{thm:Rs-T} vanishes, then  we have
				\begin{eqnarray}\label{eq:extinvRs-TcJv}
					(R_s-T)^{-1}|_{\spc_J(V)}&=&\sum_{n\geqslant 0}T^{\circledcirc n}\odot s^{-1-n}|_{\spc_J(V)}
				\end{eqnarray}
				and\begin{eqnarray}
					(R_s-T)^{\circledcirc -}&=&\sum_{n\geqslant 0}T^{\circledcirc n}\odot s^{-1-n}.\label{eq:extinvRs-T}
				\end{eqnarray}
				\item 	If $\beta(s,T)$ in Theorem \ref{thm:Rs-T} vanishes, then  we have
				\begin{eqnarray}
					\pi_J	(R_s-T)^{-1}(x)&=&\pi_J\sum_{n\geqslant 0}(s^{-1-n}\odot T^{n\circledcirc})(x) \label{eq:piJlifinvRs-T}
				\end{eqnarray} for all $x\in V$ and
				\begin{eqnarray}
					(R_s-T)^{-\circledcirc }&=&\sum_{n\geqslant 0}s^{-1-n}\odot T^{n\circledcirc} \label{eq:lifinvRs-T}.
				\end{eqnarray}
			\end{enumerate}

		\end{thm}
		\begin{proof}
			\begin{enumerate}
				\item 	In view of Theorem \ref{thm:Rs-T}, for all $x\in \spc_J(V)$ we have
				\begin{eqnarray}\label{eqpf:RS-T}
					(R_s-T)\sum_{n\geqslant 0}T^{\circledcirc n}\odot s^{-1-n}(x)=x+\alpha(s,T)(x)=x.
				\end{eqnarray}
				Lemma \ref{lem:Rs-Tinv} implies that $	R_s-T$ is invertible. Hence the desired equality  \eqref{eq:extinvRs-TcJv}  follows   immediately by applying $	(R_s-T)^{-1}$ on both sides of \eqref{eqpf:RS-T}. 	Note that since both sides of equality \eqref{eq:extinvRs-T} are right para-linear, by the Uniqueness Lemma \ref{lem:Uniqueness Lemma}, it suffices to verify \eqref{eq:extinvRs-T}  on $\re V$.
				By Definition \ref{def:ncirc}, for all $v\in \re V$  we have
				$$ (R_s-T)^{\circledcirc -}(v)=(R_s-T)^{-1}(v)=\sum_{n\geqslant 0}T^{\circledcirc n}\odot s^{-1-n}(v).$$		
				This proves  \eqref{eq:extinvRs-T}.
				
				\item  For any $x\in V$, set $(R_s-T)^{-1 }(x)=y$, thus $$x=(R_s-T)(y).$$
				It follows from \eqref{eq:2resolveformula} that
				$$\pi_J\sum_{n\geqslant 0}(s^{-1-n}\odot T^{n\circledcirc}) ((R_s-T)(y))=\pi_J\, y,$$
				that is,
				$$\pi_J \sum_{n\geqslant 0}(s^{-1-n}\odot T^{n\circledcirc}) (x)=
				\pi_J (R_s-T)^{-1 }(x).$$
				This proves \eqref{eq:piJlifinvRs-T}. Taking the real parts of both sides of \eqref{eq:piJlifinvRs-T}, we get
				\begin{eqnarray}
					\re \sum_{n\geqslant 0}(s^{-1-n}\odot T^{n\circledcirc}) (x)=
					\re (R_s-T)^{-1 }(x)
				\end{eqnarray}
				for all $x\in V$.
				Hence by the Uniqueness Lemma \ref{lem:Uniqueness Lemma}, we obtain  \eqref{eq:lifinvRs-T}.
			\end{enumerate}
		\end{proof}

		\begin{lemma}\label{lem:powass}
			Let $T\in \mathscr{B}_{\mathcal{RO}}(V)$. Then the following are equivalent:
			\begin{enumerate}
				\item $[T,T^{\circledcirc n},x]_{\circledcirc}=0$ for all $x\in V$ and $n=0,1,2,\dots$
				\item $T^n\in \mathscr{B}_{\mathcal{RO}}(V)$ for $n=0,1,2,\dots$
				\item $[T^{n\circledcirc},T,x]_{\circledcirc}=0$ for all $x\in V$ and $n=0,1,2,\dots$
			\end{enumerate}
			Moreover, under any of the above conditions, we have
			$$T^{\circledcirc n}=T^n=T^{n\circledcirc}$$  and hence
			$$T^n\circledcirc T^m=T^{m+n}.$$
		\end{lemma}
		\begin{proof}
			We note that all the assertions hold automatically for $n=0$ since $T^{\circledcirc 0}=T^{0\circledcirc}=\mathcal I$.
			
			$(1)$ $\implies$  $(2)$:
			We claim that $T^n=T^{\circledcirc n}.$ The  proof  is by induction on  $n$.
			Suppose it holds for $n$, then identity \eqref{eq:Tnext} implies that
			$$T^{\circledcirc n}=T\circledcirc T^{\circledcirc (n-1)}=T\circledcirc T^{ (n-1)}.$$
			Note the hypothesis, we thus have
			$$ T^{\circledcirc n}=T\circ T^{ (n-1)}=T^n.$$
			This proves the claim.
			
			$(2)$ $\implies$  $(1)$:
			Since both $T^n$ and $T^{\circledcirc n}$ are right para-linear operators and they coincide on the real part $\re V$, it follows by the Uniqueness Lemma that $$T^n= T^{\circledcirc n}.$$
			Therefore,
			$$[T,T^{\circledcirc n},x]_{\circledcirc}=T\circledcirc T^{\circledcirc n}(x)-T(T^{\circledcirc n}(x))=T^{\circledcirc (n+1)}(x)-T(T^nx)=0.$$
			
			The part $(2)\iff (3)$ can be proved with a similar method.
		\end{proof}

		Lemma \ref{lem:powass} inspires the following definition.
		\begin{mydef}\label{def:powerass}
			Let $T\in \mathscr{B}_{\mathcal{RO}}(V)$. $T$ is called \textbf{power-associative} if $T^n\in \mathscr{B}_{\mathcal{RO}}(V)$ for all $n\in \mathbb N$.
		\end{mydef}
		
		As a corollary, we immediately conclude from Lemma \ref{lem:powass} that
		\begin{cor}\label{cor:reolv op ser}
			Let $T\in \mathscr{B}_{\mathcal{RO}}(V)$ be a power-associative operator and   $s\in \O$ be such that $|s|>\fsh{T}$. Then  we have
			\begin{eqnarray}
				(R_s-T)^{\circledcirc -}&=&\sum_{n\geqslant 0}T^{ n}\odot s^{-1-n}\label{eq:powextinvRs-T};\\
				(R_s-T)^{-\circledcirc }&=&\sum_{n\geqslant 0}s^{-1-n}\odot T^{n} \label{eq:powlifinvRs-T}.
			\end{eqnarray}
			
		\end{cor}

		\section{$\mathbb C_J$-extendable and $\mathbb C_J$\lpa power associative operators.}
		To describe the slice regularity of the resolvent operator,   we introduce the new notions of  $\mathbb C_J$-extendable and $\mathbb C_J$\lpa  power associative operators in this section.
		
		\subsection{$\mathbb C_J$-extendable power associative operators}
		
		\begin{mydef}\label{def:ext pa}
			Let $T\in \mathscr{B}_{\R}(V)$ and $J\in \mathbb S$.
			$T$ is called \textbf{$\mathbb C_J$-extendable power associative}, if
			\begin{eqnarray}
				T^n(vs)=T^n(v)s,\qquad n=1,2,\dots
			\end{eqnarray}
			for any $v\in \re V$ and any $s\in \mathbb C_J$.
		\end{mydef}
		\begin{lemma}\label{lem:cy epa}
			An operator $T\in \mathscr{B}_{\R}(V)$ is  {$\mathbb C_J$-extendable power associative} if and only if for  {$n\in\mathbb{N}$}, $T^n|_{\mathbb C_J(V)}$ is right $\mathbb C_J$-linear, i.e.,
			$$T^n(zs)=T^n(z)s$$ for any $z\in \mathbb C_J(V)$ and any $s\in \mathbb C_J$.
		\end{lemma}
		\begin{proof}
			We only show the   necessity.
			Let $T$ be {$\mathbb C_J$-extendable power associative}.  For any $z=x+yJ\in \mathbb C_J(V)$ with $x,y\in \re V$, any $s=a+bJ\in \mathbb C_J$ with $a,b\in \R$, and {$n\in\mathbb{N}$},
			we have
			\begin{align*}
				T^n(zs)&=T^n((x+yJ)(a+bJ))\\
				&=aT^n(x)-bT^n(y)+(aT^n(y)+bT^n(x))J\\
				&=(T^n(x)+T^n(y)J)(a+bJ)\\
				&=T^n(z)s.
			\end{align*}
			This proves the lemma.
		\end{proof}

		\begin{rem}

			If $T$ is $\mathbb C_J$-extendable power associative for all $J\in \mathbb S$, then Lemma \ref{lem:cy epa} shows that $B_s(f,z)=0$ for all $s\in \mathbb C_J$ and all $z\in \mathbb C_J(V) $. It follows from Theorem \ref{lem:paralinear-char2} that   $T^n$ is right para-linear for all $n\in \mathbb N$. By Definition \ref{def:powerass}, this means $T$ is power associative.
		\end{rem}
		\begin{lemma}\label{lem:CJpower associative}
			Let $T$ be  $\mathbb C_J$-extendable power associative for some $J\in \mathbb S$. Then for all  $s\in \mathbb C_J$ and all $n\in \mathbb N$,  we have
			\begin{eqnarray}
				R_s T^n|_{\spc_J(V)}&=	 T^nR_s|_{\spc_J(V)},\label{eqpf:extasp}\\
				(R_sT)^n|_{\spc_J(V)}&=R_s^nT^n|_{\spc_J(V)}.\label{eq:RsTn=RsnTn}
			\end{eqnarray}
		\end{lemma}
		\begin{proof}
			Identity \eqref{eqpf:extasp} follows from Lemma \ref{lem:cy epa} immediately.
			We prove \eqref{eq:RsTn=RsnTn} by induction. It  clearly holds for $n=1$. Suppose \eqref{eq:RsTn=RsnTn} holds for $n-1\geqslant1$.
			For all $v\in \spc_J (V)$ and $s\in \spc_J$ , we conclude from \eqref{eqpf:extasp} and induction hypothesis that
			\begin{align*}
				(	(R_sT)^n|_{\spc_J(V)})(v)&=(R_sT)	(	(R_sT)^{n-1}|_{\spc_J(V)})(v)\\
				&=(R_sT)	(	R_s^{n-1}T^{n-1}(v))\\
				&=(R_sT)	(	T^{n-1}(vs^{n-1})\\
				&=T^n(vs^{n-1})s\\
				&=T^n(vs^{n}).
			\end{align*}	This proves the lemma. 	
		\end{proof}
		
		We introduce a sequence useful in the sequel, and study some properties.
		\begin{lemma}\label{lem:amn}
			The sequence $\{a_{m,n}\}$  defined by \begin{eqnarray}\label{def:amn}
				a_{m,n}:= \binom{m + n - 1}{m - 1},\qquad m\geqslant 1, n\geqslant 0,
			\end{eqnarray}         satisfies the following properties:
			\begin{enumerate}
				\item For any integer $m>1$,  we have
				\begin{eqnarray}\label{eqdef:amn}
					a_{m,n}=\sum_{k=0}^na_{m-1,k}, \quad n=0,1,2,\ldots
				\end{eqnarray}
				\item 
				For $m=2,3,\dots$, $n=1,2,\dots$, we have $$ a_{m,n}-a_{m,n-1}=a_{m-1,n}.$$
				\item For $n=1,2,\dots$, we have \begin{eqnarray}\label{ineq:amn}
					a_{m,n}\leqslant (1+n)^m.
				\end{eqnarray}
				\item For any real linear operators $A,B\in \mathscr{B}_{\R}(V)$, if $\fsh{A}\fsh{B}<1$, then  for any integer $m>1$, we have
				\begin{eqnarray}\label{eq:sumamnAnBn}
					\sum_{n\geqslant 0}a_{m-1,n}A^nB^{(m-1)+n}=\sum_{n\geqslant 0}a_{m,n}(A^nB^{(m-1)+n}-A^{n+1}B^{m+n}).
				\end{eqnarray}
				All the series converge in norm.
			\end{enumerate}
		\end{lemma}
		\begin{proof}
			Part (1) follows from the combinatorial identity			
			\[
			\sum_{k=0}^{n} \binom{m +k- 1}{m  - 1} = \binom{m + n}{m}, \quad n \in \mathbb{N}, m>0.			\]
			
			Part (2) follows from part (1) immediately. We prove part (3) by  induction on $m$. Clearly, $a_{1,n}=1\leqslant 1+n$ for all $n$. For $m\geqslant 2$, it follows from \eqref{eqdef:amn} and  the induction hypothesis that
			\begin{align*}
				a_{m,n}&=\sum_{k=0}^na_{m-1,k}\notag\\
				&\leqslant\sum_{k=0}^n(1+n)^{m-1}\notag\\
				&=(1+n)^m.
			\end{align*}
			This proves the inequality \eqref{ineq:amn}.
			
			The inequality \eqref{ineq:amn} enables us to get
			\begin{align*}
				\varlimsup_{n\to \infty}	\sqrt[n]{\fsh{a_{m,n}A^n{B}^{n}}}
				&\leqslant \varlimsup_{n\to \infty}	\sqrt[n] {(1+n)^m}\fsh{A}\fsh{B}\\
				&=\fsh{A}\fsh{B}<1.
			\end{align*}
			This actually shows the convergence of all the series in \eqref{eq:sumamnAnBn}.
			
			Finally, we prove \eqref{eq:sumamnAnBn}. Utilizing part (2) proved above, we obtain
			\begin{align*}
				&\sum_{n\geqslant 0}a_{m,n}(A^nB^{(m-1)+n}-A^{n+1}B^{m+n})\notag\\
				=&\sum_{n\geqslant 0}a_{m,n}A^n{B}^{(m-1)+n}-\sum_{n\geqslant 1}a_{m,n-1}A^{n}{B}^{m+(n-1)}\notag\\
				=&\sum_{n\geqslant 1}(a_{m,n}-a_{m,n-1})A^n{B}^{(m-1)+n}+a_{m,0}{B}^{m-1}\\
				=&\sum_{n\geqslant 1}a_{m-1,n}A^n{B}^{(m-1)+n}+a_{m-1,0}A^0{B}^{m-1}\\
				=&\sum_{{n\geqslant 0}}a_{m-1,n}A^n{B}^{(m-1)+n}.
			\end{align*}
			This completes the proof.
		\end{proof}
		\begin{lemma}\label{lem:Rs-TinvCpow}
			Let $T\in \mathscr{B}_{\mathcal{RO}}(V)$ be a power-associative operator and   $s\in\mathbb C_J\subseteq  \O$ such that $|s|>\fsh{T}$. Then $(R_s-T)^{-1}$ is $\mathbb C_J$-extendable power associative.
			
		\end{lemma}
		\begin{proof}
			We first claim for all $z\in \mathbb C_J(V)$ and $s\in\mathbb C_J$
			\begin{eqnarray}\label{eqpf:claimRs-Tm}
				(R_s-T)^m\sum_{n\geqslant 0}a_{m,n}T^n{R_s}^{-m-n}(z)=z, \qquad m=1,2,\ldots,
			\end{eqnarray}
			where $a_{m,n}$
 is given by \eqref{def:amn}.

			Let $z\in \mathbb C_J(V)$. Since $T$ is power-associative, by Theorem \ref{lem:paralinear-char2}, we obtain   that $B_p(T^n,z)=0$ for all $p\in \mathbb C_J$.
			It follows that
			$$T^{\circledcirc n}\odot s^{-1-n}(z)=T^{\circledcirc n}(s^{-1-n}z)-B_{s^{-1-n}}(T^{\circledcirc n},z)=T^{ n}(zs^{-1-n})=T^{ n}R_{s^{-1-n}}z.$$
			Thus \eqref{eq:extinvRs-TcJv} implies the case $m=1$.
			For $m>1$, suppose that the statement holds for $m-1 $, then we prove it for $m$.

			
			Since $\fsh{T}\fsh{{R_s}^{-1}}<1$, it follows from Lemma \ref{lem:amn} that the series $\sum_{n\geqslant 0}a_{m,n}T^n{R_s}^{-m-n}(z)$ is convergent.
			We get from   \eqref{eqpf:extasp} that
			\begin{align}\label{eqpf:Rs-Tm}
				&(R_s-T)^m\sum_{n\geqslant 0}a_{m,n}T^n{R_s}^{-m-n}(z)\notag\\
				=&(R_s-T)^{m-1}\sum_{n\geqslant 0}a_{m,n}(R_s-T)T^n{R_s}^{-m-n}(z)\notag\\
				=&(R_s-T)^{m-1}\sum_{n\geqslant 0}a_{m,n}{R_s}^{-m-n+1}T^n(z)-a_{m,n}T^{n+1}{R_s}^{-m-n}(z)\notag\\
				=&(R_s-T)^{m-1}\sum_{n\geqslant 0}a_{m,n}T^n{R_{s^{-1}}}^{(m-1)+n}(z)-a_{m,n}T^{n+1}{R_{s^{-1}}}^{m+n}(z)\notag\\
				=&(R_s-T)^{m-1}\sum_{n\geqslant 0}a_{m,n}\big(T^n{R_{s^{-1}}}^{(m-1)+n}-T^{n+1}{R_{s^{-1}}}^{m+n}\big)(z)\notag\\
				=&(R_s-T)^{m-1}\sum_{n\geqslant 0}a_{m-1,n}T^n{R_s}^{-(m-1)-n}(z),
			\end{align}
			where in the last line we used \eqref{eq:sumamnAnBn}.
			Combining the induction hypothesis with \eqref{eqpf:Rs-Tm}, we conclude  that
			$$(R_s-T)^m\sum_{n\geqslant 0}a_{m,n}T^n{R_s}^{-m-n}(z)=z,$$
			which proves  \eqref{eqpf:claimRs-Tm}.

			By identity \eqref{eqpf:claimRs-Tm}, we obtain
			\begin{align}\label{eq:rs-tn}
				(R_s-T)^{-m}|_{ \mathbb C_J(V)}=\sum_{n\geqslant 0}a_{m,n}T^n{R_s}^{-m-n}|_{ \mathbb C_J(V)}.
			\end{align}
			
			Note the operators  $T^n$ and ${R_s}$ are both right ${ \mathbb C_J}$-linear when restricted to ${ \mathbb C_J(V)}$, hence $(R_s-T)^{-m}|_{ \mathbb C_J(V)}$ is also right ${ \mathbb C_J}$-linear. In view of Lemma \ref{lem:cy epa}, we get $(R_s-T)^{-1}$ is $\mathbb C_J$-extendable power associative.
		\end{proof}
		
		\begin{rem}
			The quaternionic analog of identity \eqref{eq:rs-tn} has already been obtained in   \cite{huoxu2023power}. 	
		\end{rem}
		\subsection{$\mathbb C_J$\lpa power associative operators.}
		In this subsection
		we introduce and study  the notion of $\mathbb C_J$\lpa power associative operator.
		\begin{mydef}\label{def:lpa}
			Let $T\in \mathscr{B}_{\R}(V)$ and $J\in \mathbb S$.
			$T$ is called   \textbf{$\mathbb C_J$\lpa power associative}, if
			\begin{eqnarray}
				\re (T^n(vs))=\re (T^n(v)s), \qquad n=1,2,\dots
			\end{eqnarray}
			for all $v\in V$ and all $s\in \mathbb C_J$.
		\end{mydef}
		\begin{rem}
			If $T$ is $\mathbb C_J$\lpa power associative for all $J\in \mathbb S$, then it follows from definition of right para-linearity that   $T^n$ is right  para-linear and hence  $T$ is power associative.
		\end{rem}
		Recall the  definition \eqref{eqdef:piJ} of $\pi_J$. We prove the following:
		\begin{lemma}\label{lem:cy lpa}
			An operator $T\in \mathscr{B}_{\R}(V)$ is $\mathbb C_J$\lpa power associative if and only if for any  $s\in \mathbb C_J$ and any  $n\in \mathbb N$, it holds
			\begin{eqnarray}\label{eqpf:lieftdefeq}
				\pi_J\, R_sT^n&=	\pi_J\, T^n{R_s}.
			\end{eqnarray}
		\end{lemma}
		\begin{proof}
			We only show the   necessity.
			Let $T$ be  a {$\mathbb C_J$\lpa power associative} operator.  By definition we have
			\begin{eqnarray}\label{eqpf:reRsTn}
				\operatorname{Re}_VR_sT^n=	\operatorname{Re}_V T^n{R_s}.
			\end{eqnarray}
			Here we use the notation  $	\operatorname{Re}_V$ to emphasize that this is a map  defined on $V$, which can  be \textbf{composed} with other operators in $\mathscr{B}_{\R}(V)$.
			Recall that  $$\pi_J=	\operatorname{Re}_V+R_J	\operatorname{Re}_V R_{\overline{J}}.$$
			Using the equalities $R_sR_J=R_{Js}=R_{sJ}=R_JR_s$ for $s\in\mathbb{C}_J$  and \eqref{eqpf:reRsTn} we deduce that
			\begin{align*}
				\pi_J\, R_sT^n&=(\operatorname{Re}_V+R_J\operatorname{Re}_V R_{\overline{J}})(R_sT^n) \\
				&=\operatorname{Re}_VR_sT^n+R_J\operatorname{Re}_V R_{s\overline{J}}T^n\\
				&=\operatorname{Re}_VT^nR_s+R_J\operatorname{Re}_V T^nR_{s\overline{J}}\\
				&=\operatorname{Re}_VT^nR_s+R_J\operatorname{Re}_V T^nR_{\overline{J}}R_s\\
				&=\operatorname{Re}_VT^nR_s+R_J\operatorname{Re}_V R_{\overline{J}}T^nR_s\\
				&=\pi_J\, T^n{R_s}.
			\end{align*}
			This proves the lemma.
		\end{proof}

		\begin{lemma}\label{lem:CJliftpower associative}
			Let $T$ be  $\mathbb C_J$\lpa power associative for some $J\in \mathbb S$. Then for any  $s\in \mathbb C_J$ and any integer $n$,  we have
			\begin{eqnarray}\label{eq:liftRsTn=RsnTn}
				\pi_J\, (R_sT)^n&=\pi_J\, (R_s^nT^n).
			\end{eqnarray}
		\end{lemma}
		\begin{proof}
			We prove \eqref{eq:liftRsTn=RsnTn} by induction. It  clearly holds for $n=1$. Suppose \eqref{eq:liftRsTn=RsnTn} holds for $n-1\geqslant1$.
			We conclude from \eqref{eqpf:lieftdefeq} and the induction hypothesis that
			\begin{align*}
				\pi_J\, (R_sT)^n&=	\pi_J\,(R_sT)^{n-1}		(R_sT)\\
				&=	\pi_J\,R_s^{n-1}T^{n-1}		R_sT\\
				&=	\pi_J\,T^{n-1}R_s^{n-1}		R_sT\\
				&=	\pi_J\,T^{n-1}R_s^{n}	T\\
				&=	\pi_J\,R_s^{n}T^{n-1}	T\\
				&=	\pi_J\,R_s^{n}T^{n},
			\end{align*}	which proves the assertion.
			
		\end{proof}
		
		\begin{lemma}\label{lem:liftRs-TinvCpow}
			Let $T\in \mathscr{B}_{\mathcal{RO}}(V)$ be a power-associative operator and   $s\in\mathbb C_J\subseteq  \O$ such that $|s|>\fsh{T}$. Then $(R_s-T)^{-1}$ is $\mathbb C_J$\lpa power associative.
			
		\end{lemma}
		\begin{proof}
			We first claim for all ${v}\in V$,
			\begin{eqnarray}\label{eqpf:lifclaimRs-Tm}
				\pi_J	\sum_{n\geqslant 0}a_{m,n}T^n{R_s}^{-m-n}(v)=\pi_J\,(R_s-T)^{-m}(v),\qquad m=1,2,\ldots,
			\end{eqnarray}
		 where the numbers   $a_{m,n}$  is given by \eqref{def:amn}.

			Note that case $m=1$ follows from  \eqref{eq:piJlifinvRs-T} and $$\pi_J\left(s^{-1-n}\odot T^n\right)(v)=\pi_JT^n(vs^{-1-n})=\pi_JT^n{R_s}^{-1-n}(v),\qquad s\in \spc_J,v\in V,n\in \mathbb N.$$
			Suppose it holds for case $m-1 $ ($m>1$), we prove the case $m$.
			Since $\fsh{T}\fsh{{R_s}^{-1}}<1$, it follows from Lemma \ref{lem:amn}  the convergence of the series $	\sum_{n\geqslant 0}a_{m,n}T^n{R_s}^{-m-n}(x)$ for any $x\in V$.
			
			Moreover \eqref{eqpf:lieftdefeq} gives that $$\pi_J T^n{R_s}^{-m-n}T=\pi_J {R_s}^{-m-n}T^nT=\pi_J {R_s}^{-m-n}T^{n+1}=\pi_J T^{n+1}{R_s}^{-m-n}.$$
			Combining this and \eqref{eq:sumamnAnBn} we get
			\begin{align*}\label{eqpf:lifRs-Tm}
				&	\pi_J	\sum_{n\geqslant 0}a_{m,n}T^n{R_s}^{-m-n}(R_s-T)^m(v)\notag\\
				=&	\pi_J	\sum_{n\geqslant 0}a_{m,n}T^n{R_s}^{-m-n}(R_s-T)(R_s-T)^{m-1}(v)\notag\\
				=&	\pi_J	\sum_{n\geqslant 0}a_{m,n}\Big(T^n{R_s}^{-m-n+1}-T^{n+1}{R_s}^{-m-n}\Big)\Big((R_s-T)^{m-1}(v)\Big)\notag\\
				=&	\pi_J	\sum_{n\geqslant 0}a_{m-1,n}T^n{R_s}^{-(m-1)-n}\Big((R_s-T)^{m-1}(v)\Big)\notag\\
				=&\pi_J\, v.
			\end{align*}
			The last line used  the induction hypothesis.
			This implies   \eqref{eqpf:lifclaimRs-Tm}.

			Combining  \eqref{eqpf:lifclaimRs-Tm} and \eqref{eqpf:lieftdefeq}, we obtain
			\begin{align}
				\pi_J\,	(R_s-T)^{-m}(v\lambda)=\pi_J\,(R_s-T)^{-m}(v)\lambda
			\end{align}
			for all $\lambda\in \spc_J$ and $v\in V$.
			In view of Lemma \ref{lem:cy lpa}, we get $(R_s-T)^{-1}$ is $\mathbb C_J$\lpa power associative.
		\end{proof}

		\section{Octonionic spectra}
		As we shall explain in this section, for a right octonionic para-linear operator, two distinct types of octonionic spectra can be introduced. These spectra correspond to the calculus of left  and right slice regular functions, respectively. In this section, we assume that  $V$ is a nonzero Banach $\O$-bimodule.
		\subsection{Octonionic pull-back spectrum}
\begin{mydef}\label{def:pullback spec}
Let $T\in \mathscr{B}_{\mathcal{RO}}(V)$ be a bounded right para-linear operator. For $J\in\mathbb S$, define the \textbf{pull-back slice-resolvent set}
\[
\rho_J^*(T):=\{s\in\mathbb C_J:R_s-T\text{ is invertible in }\mathscr{B}_{\mathbb R}(V)\text{ and }(R_s-T)^{-1}\text{ is }\mathbb C_J\text{-extendable power associative}\},
\]
and the \textbf{pull-back slice spectrum}
\[
\sigma_J^*(T):=\mathbb C_J\setminus\rho_J^*(T).
\]
The \textbf{pull-back spectrum} and \textbf{pull-back resolvent set} are
\[
\sigma^*(T):=\bigcup_{J\in\mathbb S}\sigma_J^*(T),\qquad
\rho^*(T):=\mathbb O\setminus\sigma^*(T).
\]
Equivalently,
\[
\rho^*(T)=\left(\bigcup_{J\in\mathbb S}(\rho_J^*(T)\setminus\mathbb R)\right)
\cup\left(\bigcap_{J\in\mathbb S}\rho_J^*(T)\right).
\]
Thus, for nonreal $s\in\mathbb C_J$, $s\in\rho^*(T)$ if and only if $s\in\rho_J^*(T)$, whereas for $r\in\mathbb R$, $r\in\rho^*(T)$ if and only if $r\in\bigcap_{J\in\mathbb S}\rho_J^*(T)$.
\end{mydef}
		\begin{rem}\label{rem:pull-back spec}
			\begin{enumerate}
				\item  Suppose that $T$ is \textbf{octonionic linear}, then  $V$ can be viewed as a complex linear space, denoted by $V_J$, induced by the complex structure $R_J$, where $J\in \mathbb S$, i.e.,   $$\sqrt{-1}v:=vJ$$ for any $v\in V$.
				Since $T$ is $\O$-linear, $$T:V_J\to V_J$$ can be viewed as a $\spc_J$-linear operator. This feature \textbf{fails for general para-linear operators.}
				Noticing that for any $s\in \spc_J$,  $R_s:V_J\to V_J$  is a $\spc_J$-linear operator, if $R_s-T$ is invertible in $\mathscr{B}_{\R}(V)$, then also $(R_s-T)^{-1}:V\to V$ is $\spc_J$-linear and hence  the  $\mathbb C_J$-extendable power associativity of $(R_s-T)^{-1}$  automatically  holds.
				Indeed, suppose $s\in \spc_J$ and  $R_s-T$ is invertible in $\mathscr{B}_{\mathbb{R}}(V)$. For any $\lambda\in \spc_J$ and any $v\in V$, since
				$$(R_s-T)((R_s-T)^{-1}(v)\lambda)=(R_s-T)((R_s-T)^{-1}(v))\lambda=v\lambda,$$
				it follows that
				$$(R_s-T)^{-1}(v)\lambda=(R_s-T)^{-1}(v\lambda)$$ which shows that  $(R_s-T)^{-1}$ is a $\spc_J$-linear operator on $V_J$. Thus $(R_s-T)^{-n}$ is also $\spc_J$-linear for each integer $n$ and hence $(R_s-T)^{-1}$ is  $\mathbb C_J$-extendable power associative.
				
				{In other words, for an octonionic linear operator $T$, the pull-back slice-resolvent set of $T$  coincides with}
				$$\rho_J(T):=\{s\in \mathbb C_J: R_s-T \text{ is invertible in } \mathscr{B}_{\mathbb{R}} (V)\}.$$

				\item {The previous discussion applies also to quaternionic right linear operators.
					In other words, for a quaternionic linear operator $T$, the pull-back slice-resolvent set of $T$, for $J\in\mathbb{S}$, coincides with}
				$$\rho_J(T):=\{s\in \mathbb C_J: R_s-T \text{ is invertible in } \mathscr{B}_{\mathbb{R}} (V)    \}.$$
				\item 	In view of Lemma \ref{lem:cy epa}, when the operator $(R_s-T)^{-1}$ is $\mathbb C_J$-extendable power associative we deduce that
				$$(R_s-T)^{-n}|_{\spc_J(V)}\text{ is right $\spc_J$-linear}.$$	If we define the embedding map
				\begin{eqnarray}
					\iota_J:\spc_J(V)\to V,\qquad x \mapsto x
				\end{eqnarray}
				and	if we use this  to  pull back  the operator $(R_s-T)^{-1}$:
				$$\iota_J^*((R_s-T)^{-1}):=(R_s-T)^{-1} \circ \iota_J:\spc_J(V)\to V.$$
			{This is precisely why we adopt the notations of $\rho^*(T)$ and $\sigma^*(T)$,  and the operator $\iota_J^*((R_s-T)^{-1})$  essentially takes the place of the classical resolvent operator $(R_s-T)^{-1}$.}

			\end{enumerate}
			
		\end{rem}
		
		\begin{lemma}	\label{lem:m=1case}
			Let $T\in \mathscr{B}_{\mathcal{RO}}(V)$.	Let $s_0\in \rho_J^*(T)$ for some $J\in \mathbb S$. Define $\delta_0:=\fsh{(R_{s_0}-T)^{-1}}^{-1}$. Denote by $B(s_0,\delta_0)$  the open disk in $\spc_J$ centered at $s_0$ with radius $\delta_0$. Then for all $s\in B(s_0,\delta_0)$, $R_s-T$ is invertible {in $\mathscr{B}_{\mathbb{R}}(V)$} and
			\begin{eqnarray}\label{eq:R_s-T}
				(R_s-T)^{-1}|_{ \mathbb C_J(V)}=\sum_{n\geqslant 0}{(R_{s_0}-T)}^{-1-n}R_{s_0-s}^{n}|_{ \mathbb C_J(V)}.
			\end{eqnarray}
			
		\end{lemma}
		\begin{proof}
			The convergence of the series in \eqref{eq:R_s-T} is obvious.
			
			It follows from  $s_0\in\rho_J^*(T) $  that $(R_{s_0}-T)^{-1}$ is $\mathbb C_J$-extendable power associative. Hence for any $v\in \spc_J(V)$, we deduce from \eqref{eq:RsTn=RsnTn} that for all  $n\in \mathbb N$,  $${(R_{s_0-s}(R_{s_0}-T)^{-1})}^{n}(v)={(R_{s_0}-T)}^{-n}R_{s_0-s}^{n}{(v)}.$$
			Note that
			\begin{eqnarray}\label{eqpf:Rs-Tisinv}
				R_s-T=R_{s_0}-T-R_{s_0-s}=(R_{s_0}-T)(\mathcal{I}-(R_{s_0}-T)^{-1}R_{s_0-s}).
			\end{eqnarray}
			Since  $s\in  B(s_0,\delta_0)$, we get  that $(\mathcal{I}-(R_{s_0}-T)^{-1}R_{s_0-s})$ is invertible {in $\mathscr{B}_{\mathbb{R}}(V)$}.
			Combining this with \eqref{eqpf:Rs-Tisinv}, we deduce that also $	R_s-T$ is invertible.

			Since
			$$(R_s-T)^{-1}=(R_{s_0}-T-R_{s_0-s})^{-1}=(R_{s_0}-T)^{-1}(\mathcal{I}-R_{s_0-s}(R_{s_0}-T)^{-1})^{-1},$$
			it follows that
			\begin{align*}
				(R_s-T)^{-1}(v)&=(R_{s_0}-T)^{-1}(\mathcal{I}-R_{s_0-s}(R_{s_0}-T)^{-1})^{-1}(v)\\
				&=(R_{s_0}-T)^{-1}\sum_{n\geqslant 0}{(R_{s_0-s}(R_{s_0}-T)^{-1})}^{n}(v)\\
				&=(R_{s_0}-T)^{-1}\sum_{n\geqslant 0}{(R_{s_0}-T)}^{-n}R_{s_0-s}^{n}(v)\\
				&=\sum_{n\geqslant 0}{(R_{s_0}-T)}^{-1-n}R_{s_0-s}^{n}(v).
			\end{align*}
			This completes the proof.
		\end{proof}

\begin{thm}\label{thm:rhosliceopen}
Let $T\in\mathscr{B}_{\mathcal{RO}}(V)$. Then $\rho_J^*(T)$ is open in $\mathbb C_J$ for every $J\in\mathbb S$.
\end{thm}

		\begin{proof}
		We show that $\rho_J^*(T)$ is  open in $\mathbb C_J$ for any $J\in \mathbb S$.
			
			For any $s_0\in\rho_J^*(T) $, denote by $\delta_0:=\fsh{(R_{s_0}-T)^{-1}}^{-1}$ and denote by $B(s_0,\delta_0)$ the open disk center at $s_0$ with radius $\delta_0$.
			Claim: for every $s\in B(s_0,\delta_0)$, $R_s-T$ is invertible and
			\begin{eqnarray}\label{eq:pfRs-T}
				(R_s-T)^{-m}|_{ \mathbb C_J(V)}=\sum_{n\geqslant 0}a_{m,n}{(R_{s_0}-T)}^{-m-n}R_{s_0-s}^{n}|_{ \mathbb C_J(V)}.
			\end{eqnarray}
			
			We first show the convergence of the series in \eqref{eq:pfRs-T}.
			Since $s_0\in\rho_J^*(T) $, it follows that $(R_{s_0}-T)^{-1}$ is $\mathbb C_J$-extendable power associative and by Lemma \ref{lem:cy epa} we have for any $v\in \mathbb C_J(V)$ and any $m\in \mathbb N$
			$${(R_{s_0}-T)}^{-m-n}R_{s_0-s}^{n}(v)={(R_{s_0}-T)}^{-m-n}(v){{(s_0-s)}}^{n}.$$
			Combining  \eqref{ineq:amn} and noting ${\delta_0}^{-1}=\fsh{(R_{s_0}-T)^{-1}}$, we conclude that
			\begin{align*}
				\fsh{a_{m,n}{(R_{s_0}-T)}^{-m-n}R_{s_0-s}^{n}v}&\leqslant(1+n)^m{\delta_0}^{-m-n}\fsh{v}\abs{s_0-s}^n\\
				&=\left(\frac{1+n}{\delta_0}\right)^m\left(\frac{\abs{s_0-s}}{\delta_0}\right)^n\fsh{v}.
			\end{align*}
			Since $\abs{s_0-s}<\delta_0,$   we conclude the convergence of the series $$\sum_{n\geqslant 0}a_{m,n}{(R_{s_0}-T)}^{-m-n}R_{s_0-s}^{n}(v)$$ for any $v\in \mathbb C_J(V)$.
			
			We next prove \eqref{eq:pfRs-T} by induction on $m$.
			Case $m=1$ follows from Lemma \ref{lem:m=1case}.
			Set $A:=R_{s_0-s}$ and $B:={(R_{s_0}-T)}^{-1}$ and let	$v\in \mathbb C_J(V)$. We get from \eqref{eqpf:extasp}  that
			$$R_{s_0-s}B^{m+n}R_{s_0-s}^n(v)=R_{s_0-s}R_{s_0-s}^nB^{m+n}(v)={A}^{n+1}B^{m+n}(v)$$ for any integers $m,n$. Thus using identities \eqref{eqpf:extasp},  \eqref{eq:sumamnAnBn} and the induction hypothesis, we get
			\begin{align*}
				&	(R_s-T)^{m}\left(\sum_{n\geqslant 0}a_{m,n}{(R_{s_0}-T)}^{-m-n}R_{s_0-s}^{n}(v)\right)\\
				=&	(R_s-T)^{m-1}(-R_{s_0-s}+R_{s_0}-T)\left(\sum_{n\geqslant 0}a_{m,n}{(R_{s_0}-T)}^{-m-n}R_{s_0-s}^{n}(v)\right)\\
				=&(R_s-T)^{m-1}\left(\sum_{n\geqslant 0} -a_{m,n}R_{s_0-s}B^{m+n}{R_{s_0-s}}^n(v) +a_{m,n}{B}^{m+n-1}{R_{s_0-s}}^n(v)\right)\\
				=&(R_s-T)^{m-1}\left(\sum_{n\geqslant 0} a_{m,n}({A}^n{B}^{m+n-1}-{A}^{n+1}B^{m+n})(v)\right)\\
				=&(R_s-T)^{m-1}\sum_{n\geqslant 0}a_{m-1,n}{(R_{s_0}-T)}^{-(m-1)-n}R_{s_0-s}^{n}(v)\\
				=& v.
			\end{align*}
			This implies the claimed equality \eqref{eq:pfRs-T}.

			By \eqref{eq:pfRs-T}, we have $s\in\rho_J^*(T) $ for every $s\in B(s_0,\delta_0)$. 
		\end{proof}
		
		Let us denote by $$V^{*} \text{ and } V^{*_{\O}}$$ the set of all continuous \textit{right para-linear} functionals  and  all continuous  \textit{octonionic linear} functionals on $V$, respectively.
	
\begin{thm}\label{thm:Rs-Tsliceregular}
Let $T$ be a bounded right para-linear operator.  For any $v\in\re V$ and any octonionic linear functional $\phi\in V^{*_{\O}}$, the function
\[
s\longmapsto \phi((R_s-T)^{\circledcirc-}(v))
\]
is holomorphic on $\rho_J^*(T)$ for any $J\in\mathbb S$. Consequently, on every slice-open set $W\subseteq\rho^*(T)$, $s\longmapsto \phi((R_s-T)^{\circledcirc-}(v))$ is a right slice regular function.
\end{thm}

		\begin{proof}
Fix $J\in\mathbb S$ and set $g_J(s)=\phi((R_s-T)^{\circledcirc-}(v))$ for $s=x+yJ\in\rho_J^*(T)$. Since $v\in\re V$, $(R_s-T)^{\circledcirc-}(v)=(R_s-T)^{-1}(v)$. Denote
\[
G:=\{S\in\mathscr{B}_{\mathbb R}(V):S\text{ is invertible in }\mathscr{B}_{\mathbb R}(V)\}.
\]
The inverse map $G\to G$, $S\mapsto S^{-1}$, is continuous; hence $(R_s-T)^{-1}:\rho_J^*(T)\to G$ and $g_J$ are continuous on the fixed slice.

			By direct calculations, we have
			\begin{align*}
				&\qquad \frac{\partial}{\partial x}(R_s-T)^{ -1}(v)\\
				&=\lim_{\Delta_x\to 0}\frac{(R_{x+\Delta_x+yJ}-T)^{-1}-(R_{x+yJ}-T)^{-1}}{\Delta_x}(v)\\
				&=\lim_{\Delta_x\to 0}\frac{(R_{x+\Delta_x+yJ}-T)^{-1}((R_{x+yJ}-T)-(R_{x+\Delta_x+yJ}-T))(R_{x+yJ}-T)^{-1}}{\Delta_x}(v)\\
				&=-(R_s-T)^{-1}(R_s-T)^{-1}(v),
			\end{align*}
		{Here we use the continuity of $(R_s-T)^{-1}:\rho_J^*(T)\to G$ on the fixed slice.}
			Similarly,
			\begin{align*}
				&\qquad \frac{\partial}{\partial y}(R_s-T)^{-1}(v)\\
				&=\lim_{\Delta_y\to 0}\frac{(R_{x+(\Delta_y+y)J}-T)^{-1}-(R_{x+yJ}-T)^{-1}}{\Delta_y}(v)\\
				&=\lim_{\Delta_y\to 0}\frac{(R_{x+(\Delta_y+y)J}-T)^{-1}((R_{x+yJ}-T)-(R_{x+(\Delta_y+y)J}-T))(R_{x+yJ}-T)^{-1}}{\Delta_y}(v)\\
				&=-(R_s-T)^{-1}R_J(R_s-T)^{-1}(v).
			\end{align*}
			Since $s\in  \rho_J^*(T)$, it follows that $(R_s-T)^{-1}$ is $\mathbb C_J$-extendable power associative.
			Combining with  Lemma \ref{lem:CJpower associative},  for any octonionic linear functional  $\phi\in V^{*_{\O}}$ we get
			\begin{align*}
				&\qquad \left (\frac{\partial}{\partial x}+ \frac{\partial}{\partial y}R_J\right)\phi ((R_s-T)^{\circledcirc -}(v))\\
				&=\phi\Big(-(R_s-T)^{-1}(R_s-T)^{-1}(v)-R_J(R_s-T)^{-1}R_J(R_s-T)^{-1}(v)\Big)\\
				&=0.
			\end{align*}
			Since $J\in\mathbb S$ is arbitrary, this completes the proof.
		\end{proof}

		The pull-back spectrum of a power-associative operator intersected with a slice is a nonempty compact set, as we prove below:
		\begin{thm}\label{lem:sigmaTbd}
			Let $T\in \mathscr{B}_{\mathcal{RO}}(V)$ be a power-associative operator.	If $|s|>\fsh{T}$, then $s\in \rho^*(T)$.
		\end{thm}
		\begin{proof}
			This follows from Lemmas \ref{lem:Rs-Tinv} and  \ref{lem:Rs-TinvCpow} immediately.
		\end{proof}
		
		\begin{thm}\label{thm:extspec cpt}
			Let $T\in \mathscr{B}_{\mathcal{RO}}(V)$ be a power-associative operator. Then	$\sigma_J^*(T)$ is  compact  and nonempty in $\spc_J$ for any $J\in \mathbb S$.
		\end{thm}
		\begin{proof}
			By Theorem \ref{lem:sigmaTbd} and Theorem \ref{thm:rhosliceopen}, we get 	$\sigma_J^*(T)$ is  compact. 	Mimicking the  canonical method, we conclude $\sigma_J^*(T)$ is   nonempty from	Theorem \ref{thm:Rs-Tsliceregular}.
		\end{proof}
		
		\subsection{Octonionic push-forward spectrum}
\begin{mydef}\label{def:pushforward sepc}
Let $T\in\mathscr{B}_{\mathcal{RO}}(V)$ be a bounded right para-linear operator. For $J\in\mathbb S$, define
\[
{\rho_*}_J(T):=\{s\in\mathbb C_J:R_s-T\text{ is invertible in }\mathscr{B}_{\mathbb R}(V)\text{ and }(R_s-T)^{-1}\text{ is }\mathbb C_J\text{-liftable power associative}\},
\]
\[
{\sigma_*}_J(T):=\mathbb C_J\setminus{\rho_*}_J(T),\qquad
\sigma_*(T):=\bigcup_{J\in\mathbb S}{\sigma_*}_J(T),\qquad
\rho_*(T):=\mathbb O\setminus\sigma_*(T).
\]
Equivalently,
\[
\rho_*(T)=\left(\bigcup_{J\in\mathbb S}({\rho_*}_J(T)\setminus\mathbb R)\right)
\cup\left(\bigcap_{J\in\mathbb S}{\rho_*}_J(T)\right).
\]
\end{mydef}

		\begin{rem}
			\begin{enumerate}
				\item  Suppose $T$ is $\O$-linear. Reasoning as in Remark \ref{rem:pull-back spec}, if $R_s-T$ is invertible in $\mathscr{B}_{\R}(V)$, then    the  $\mathbb C_J$-liftable power associativity of $(R_s-T)^{-1}$  automatically  holds. Therefore, for $\O$-linear operators $T$, the push-forward slice resolvent of $T$  is
				$$\rho_J(T):=\{s\in \mathbb C_J: R_s-T \text{ is invertible in }   \mathscr{B}_{\R}(V) \}.$$
				This discussion can be repeated in the quaternionic case for (right) quaternionic linear operators.
				
				\item 		Using the map
				\begin{eqnarray}
					\pi_J:V\to \spc_J(V),
				\end{eqnarray}
				we can  push forward  the operator $(R_s-T)^{-1}$:
				$${\pi_J}_*((R_s-T)^{-1}):=\pi_J\circ(R_s-T)^{-1}: V\to \spc_J(V),$$
				which essentially takes the place of the classical resolvent operator $(R_s-T)^{-1}$.  This is  why we adopt the notations
				$\rho_*(T)$ and $\sigma_*(T)$.

			\end{enumerate}
			
		\end{rem}
		
		The counterpart of Theorem \ref{thm:Rs-Tsliceregular} is the following result whose proof is analogous and thus omitted here.
	\begin{thm}\label{thm:pushforw slice open}
Let $T\in\mathscr{B}_{\mathcal{RO}}(V)$. Then ${\rho_*}_J(T)$ is open in $\mathbb C_J$ for every $J\in\mathbb S$.
\end{thm}

		{The following theorem is the counterpart of Theorem \ref{thm:Rs-Tsliceregular} and establishes that a suitably defined function is right
			slice-regular on the push-forward slice-resolvent set.}
		\begin{thm}\label{thm:liftRs-Tsliceregular}
			Let $T\in \mathscr{B}_{\mathcal{RO}}(V)$. Then 	for any fixed  $v\in V$, $J\in \mathbb S$, and any octonionic linear functional  $\phi\in V^{*_{\O}}$,  define for any $s\in {\rho_*}_{J}(T)$ that  $$g(s):=\phi\Big(	\pi_J\,  (R_s-T)^{-\circledcirc}(v)\Big).$$ Then $g:{\rho_*}_{J}(T)\to \spc_J$ is holomorphic on ${\rho_*}_{J}(T)$.
		\end{thm}
		
		\begin{proof}
			Note that $\phi$ is an octonionic linear functional. It follows from \eqref{eq:refx=frex} that $g$ is $\spc_J$-valued.
			
			Let   $s=x+yJ\in {\rho_*}_J(T)$.
			By Definition \ref{def:ncirc}, we have
			\begin{eqnarray}\label{eq:reV (Rs-T)-circledcirc}
				\re_V (R_s-T)^{-\circledcirc}=\re_V (R_s-T)^{-1}.
			\end{eqnarray}
			Since  $s\in {\rho_*}_J(T)$, it follows that
			\begin{eqnarray}\label{eqpf:reVRs-T}
				\re_V (R_s-T)^{-1}R_{\overline{J}}=\re_V R_{\overline{J}}(R_s-T)^{-1}.
			\end{eqnarray}
			Thus 		we deduce
			\begin{align}
				&\quad \ 	\pi_J\,(R_s-T)^{-\circledcirc}\notag\\
				&=(\re_V +R_J	\operatorname{Re}_V R_{\overline{J}})(R_s-T)^{-\circledcirc}\notag\\
				&=\re_V (R_s-T)^{-\circledcirc}+R_J	\operatorname{Re}_V R_{\overline{J}}(R_s-T)^{-\circledcirc}\notag\\
				&=\re_V (R_s-T)^{-\circledcirc}+R_J	\operatorname{Re}_V (R_s-T)^{-\circledcirc}R_{\overline{J}}&\text{since $(R_s-T)^{-\circledcirc}$ is right para-linear}\notag\\
				&=\re_V (R_s-T)^{-1}+R_J	\operatorname{Re}_V (R_s-T)^{-1}R_{\overline{J}}&\text{using \eqref{eq:reV (Rs-T)-circledcirc} }\notag\\
				&=\re_V (R_s-T)^{-1}+R_J	\operatorname{Re}_V R_{\overline{J}}(R_s-T)^{-1}&\text{using \eqref{eqpf:reVRs-T} }\notag\\
				&=	\pi_J\,(R_s-T)^{-1}.
			\end{align}
			Hence, as we did in the proof of Theorem \ref{thm:Rs-Tsliceregular}, for any $v\in V$ we have
			\begin{align*}
				\pi_J\, \frac{\partial}{\partial x}(R_s-T)^{-\circledcirc}(v)&=	 \frac{\partial}{\partial x}\pi_J\,(R_s-T)^{-\circledcirc}(v)\\
				&=	 \frac{\partial}{\partial x}\pi_J\,(R_s-T)^{-1}(v)\\
				&=-\pi_J\,(R_s-T)^{-1}(R_s-T)^{-1}(v)
			\end{align*}
			and
			\begin{align*}
				\pi_J\, \frac{\partial}{\partial y}(R_s-T)^{-\circledcirc}(v)&=	 \frac{\partial}{\partial y}\pi_J\,(R_s-T)^{-\circledcirc}(v)\\
				&=	 \frac{\partial}{\partial y}\pi_J\,(R_s-T)^{-1}(v)\\
				&=-\pi_J\,(R_s-T)^{-1}R_J(R_s-T)^{-1}(v).
			\end{align*}
			Therefore, we get from \eqref{eq:piJ} that
			\begin{align*}
				\left (\frac{\partial}{\partial x}+ \frac{\partial}{\partial y}{J}\right)g(s)&= 	\left (\frac{\partial}{\partial x}+ \frac{\partial}{\partial y}J\right)\phi (\pi_J\, (R_s-T)^{-\circledcirc}(v))\\
				&=	\phi \circ\pi_J\,\left(\frac{\partial}{\partial x}+R_J\frac{\partial}{\partial y}\right)(R_s-T)^{-\circledcirc}(v)\\
				&=\phi	\Big(		-\pi_J\,(R_s-T)^{-1}(R_s-T)^{-1}(v)	-\pi_J\,R_J(R_s-T)^{-1}R_J(R_s-T)^{-1}(v)\Big)\\
				&=0,
			\end{align*}
			where   the last equality follows from Lemma \ref{lem:cy lpa}:
			$$\pi_J\,R_J(R_s-T)^{-1}R_J(R_s-T)^{-1}(v)=-\pi_J\,(R_s-T)^{-1}(R_s-T)^{-1}(v).$$ This completes  the proof.
		\end{proof}		
		\begin{thm}\label{thm:pushforw bound}
			Let $T\in \mathscr{B}_{\mathcal{RO}}(V)$ be a power-associative operator.	If $|s|>\fsh{T}$, then $s\in \rho_{*}(T)$.
		\end{thm}
		\begin{proof}
			The assertion immediately follows from Lemma \ref{lem:Rs-Tinv} and Lemma \ref{lem:liftRs-TinvCpow}.
		\end{proof}
		Hence we also get
		\begin{thm}\label{thm:pushfw cpt}
			Let $T\in \mathscr{B}_{\mathcal{RO}}(V)$ be a power-associative operator. Then	${\sigma_*}_J(T)$ is  compact  and nonempty in $\spc_J$ for any $J\in \mathbb S$.
		\end{thm}
		
		\section{Octonionic functional calculus}
		\subsection{Preliminaries}
		To introduce the functional calculus we first introduce some preliminary definitions and notions. To begin with, we give the following definition:
		\begin{mydef}\label{def:resolvent op}
			Let $V$ be a  Banach $\O$-bimodule, $T\in \mathscr{B}_{\mathcal{RO}}(V)$.
			\begin{enumerate}
				\item For $s\in \rho^*(T)$, we define the \textbf{left resolvent operator}  as
				\begin{eqnarray}\label{eq:lro}
					(	R_s-T)^{\circledcirc -}.
				\end{eqnarray}
				\item For $s\in \rho_*(T)$, we define the \textbf{right resolvent operator}  as
				\begin{eqnarray}\label{eq:rro}
					(	R_s-T)^{-\circledcirc }.
				\end{eqnarray}
			\end{enumerate}
		\end{mydef}
		
		\begin{rem}\label{rem:resop}
			
			\begin{enumerate}
				\item  The bijections $\operatorname{ext} $ and $\operatorname{lif}$ can also be defined in quaternionic case, and are analogous to  those in \eqref{eq:ext} and \eqref{eq:lif}. Hence in view of Definition \ref{def:ncirc},  the operators defined by \eqref{eq:lro} and \eqref{eq:rro} can also be given in quaternionic case. Specifically, if $V$ is a two-sided quaternionic Banach space  (i.e., Banach $\mathbb H$-bimodule), $T$ is a bounded right quaternionic linear operator and $s\in \rho_S(T)$ (see \cite[Definition 4.8.1]{colombo2011noncomfunctcalculus}). Then, applying the quaternionic  counterpart  of the
				Uniqueness Lemma \ref{lem:Uniqueness Lemma}, with some calculations we can prove that
				\begin{eqnarray}
					(	R_s-T)^{\circledcirc -}&=	-(T^ 2 - 2\re [s] T + \abs{s}^2\mathcal{I})^{-1}(T - L_{\overline{s}})\label{eq:lroq}\\
					(	R_s-T)^{-\circledcirc }&=	-(T - L_{\overline{s}})(T^ 2 - 2\re [s] T + \abs{s}^2\mathcal{I})^{-1}.\label{eq:rroh}
				\end{eqnarray}
				These two formulas coincide with the two S-resolvent operators in the quaternionic case \cite[Definition 4.8.3]{colombo2011noncomfunctcalculus}.
				
				We prove only \eqref{eq:lroq}, since the proof of \eqref{eq:rroh} works in a  similar way. Since both sides of \eqref{eq:lroq} are right quaternionic linear operators, by the quaternionic analog Uniqueness Lemma \ref{lem:Uniqueness Lemma}, it suffices to check the equality on the real part $\re V$.
				Since for any $s\neq 0$, there exists $p\in \mathbb H$ with $\abs{p}=1$ such that
				$\overline{s}=ps\overline{p}$, we deduce
				$	R_s-T$ is bounded invertible if and only if $	R_{\overline{s}}-T$ is bounded invertible.
				Note that, {as bounded real linear operator, we can write the factorization}
				$$T^2 - 2\re [s] T + \abs{s}^2\mathcal{I}=(T-R_s)(T-R_{\overline{s}}).$$
				Thus for any $s$ such that $	R_s-T$ is bounded invertible {as real linear operator},
				we have $$(T^ 2 - 2\re [s] T + \abs{s}^2\mathcal{I})^{-1}=(T-R_s)^{-1}(T-R_{\overline{s}})^{-1}.$$
				Combining with
				$$(T - L_{\overline{s}})|_{\re V}=(T - R_{\overline{s}})|_{\re V},$$
				we deduce
				\begin{align*}
					-(T^ 2 - 2\re [s] T + \abs{s}^2\mathcal{I})^{-1}(T - L_{\overline{s}})|_{\re V}&=-(T-R_s)^{-1}(T-R_{\overline{s}})^{-1}(T - R_{\overline{s}})|_{\re V}\\
					&=-(T-R_s)^{-1}|_{\re V}\\
					&=	(	R_s-T)^{\circledcirc -}|_{\re V},
				\end{align*}
				thus proving \eqref{eq:lroq}. We obtain \eqref{eq:rroh} with an analogous reasoning.

				Moreover, in light of Remark \ref{rem:regular inv} (which also holds in the quaternionic case) we have that $R_s-T$, {seen again as real linear operator,} satisfies:
				\begin{align*}
					(R_s-T)\circledcirc	(	R_s-T)^{\circledcirc -}&=\mathcal{I},\\
					(	R_s-T)^{-\circledcirc }	\circledcirc(R_s-T)	&=\mathcal{I}.
				\end{align*}
				This is equivalent to
				\begin{align*}
					(R_s-T)	(	R_s-T)^{\circledcirc -}|_{\re V}&={\mathcal{I}|_{\re V}},\\
					\re_V	(	R_s-T)^{-\circledcirc }	(R_s-T)	&={\re_V\mathcal{I}}.
				\end{align*}
				Using the right quaternionic linearity of $(	R_s-T)^{\circledcirc -}$ and $(	R_s-T)^{-\circledcirc }$, these two equalities become
				\begin{align*}
					\Big	(	(	R_s-T)^{\circledcirc -}R_s-T(	R_s-T)^{\circledcirc -}\Big)|_{\re V}&={\mathcal{I}|_{\re V}},\\
					\re_V\Big(R_s	(	R_s-T)^{-\circledcirc }	-	(	R_s-T)^{-\circledcirc }T\Big)	&={\re_V\mathcal{I}}.
				\end{align*}
				Noticing that $$R_s|_{\re V}=L_s|_{\re V},\qquad 	\re_VR_s=	\re_VL_s,$$
				we thus get
				\begin{align*}
					\Big	(	(	R_s-T)^{\circledcirc -}L_s-T(	R_s-T)^{\circledcirc -}\Big)|_{\re V}&={\mathcal{I}|_{\re V}},\\
					\re_V\Big(L_s	(	R_s-T)^{-\circledcirc }	-	(	R_s-T)^{-\circledcirc }T\Big)	&={\re_V\mathcal{I}}.
				\end{align*}
				{Note that $(	R_s-T)^{\circledcirc -}L_s-T(	R_s-T)^{\circledcirc -}$ and $L_s	(	R_s-T)^{-\circledcirc }	-	(	R_s-T)^{-\circledcirc }T$ are right quaternionic linear operators } hence by the quaternionic counterpart of Uniqueness Lemma we can deduce the quaternionic S-resolvent equations (\cite[Theorem 4.8.4, Definition  4.8.5]{colombo2011noncomfunctcalculus}):
				\begin{align*}
					(	R_s-T)^{\circledcirc -}L_s-T(	R_s-T)^{\circledcirc -}&=\mathcal{I},\\
					L_s	(	R_s-T)^{-\circledcirc }	-	(	R_s-T)^{-\circledcirc }T	&=\mathcal{I}.
				\end{align*}
				\item Definition \ref{def:resolvent op} unifies the notion of  resolvent operator in Banach space over division algebras.
			\end{enumerate}
		\end{rem}

	{Recall Definition \ref{def:s-dom} for s-domain.}
		\begin{mydef}\label{def:Tadmiss}
			\begin{enumerate}
				\item 	Let $U\subseteq \O$ be a \textbf{{bounded} axially symmetric s-domain} that contains the {closure of the} \textbf{pull-back spectrum $\sigma^*(T )$}, and such that $\partial(U \cap \spc_J )$ is the union of a finite number of continuously differentiable Jordan curves for every $J\in \mathbb S$. We say that $U$ is a \textbf{$T$-left-admissible} open set.
				\item 	Let $U\subseteq \O$ be a \textbf{{bounded} axially symmetric s-domain} that contains the  {closure of the} push-forward spectrum $\sigma_*(T )$, and such that $\partial(U \cap \spc_J )$ is the union of a finite number of continuously differentiable Jordan curves for every $J\in \mathbb S$. We say that $U$ is a \textbf{$T$-right-admissible} open set.
			\end{enumerate}
			
		\end{mydef}
		\begin{rem}
			The requirement that $U$ is an 	\textbf{axially symmetric} s-domain
			is imposed to ensure the validity of the slice Cauchy formula  for slice regular functions for subsequent use. This hypothesis could be eliminated if a non-axially symmetric version of the slice Cauchy formula were developed.
		\end{rem}
		\begin{mydef}\label{def:T-admiss}
			Let $W$ be a slice-open set in $\O$.
			\begin{enumerate}
				\item A function $f\in \mathcal{SR}^L(W)$ is said to be \textbf{locally left regular on $\sigma^* (T )$} if there exists a {$T$-\textbf{left}-admissible domain} $U\subseteq \O$ such that $\overline{U}\subseteq W $. We will denote by $ \mathcal{SR}^L({\sigma^*(T)})$ the set of all locally left regular functions on $\sigma^* (T )$.
				\item A function $f\in \mathcal{SR}^R(W)$ is said to be \textbf{locally right regular on $\sigma^* (T )$} if there exists a {$T$-\textbf{left}-admissible domain} $U\subseteq \O$ such that $\overline{U}\subseteq W $. We will denote by $ \mathcal{SR}^R({\sigma^*(T)})$ the set of locally right regular functions on $\sigma ^*(T )$.
				\item 	A function $f\in \mathcal{SR}_{\R}(W)$ is said to be \textbf{locally slice preserving on $\sigma^* (T )$} if there exists a {$T$-\textbf{left}-admissible domain} $U\subseteq \O$ such that $\overline{U}\subseteq W $. We will denote by $ \mathcal{SR}_{\R}({\sigma^*(T)})$ the set of all locally slice preserving functions on $\sigma^* (T )$.
				\item A function $f\in \mathcal{SR}^L(W)$ is said to be \textbf{locally left regular on $\sigma_* (T )$} if there exists a {$T$-\textbf{right}-admissible domain} $U\subseteq \O$ such that $\overline{U}\subseteq W $. We will denote by $ \mathcal{SR}^L({\sigma_*(T)})$ the set of all locally left regular functions on $\sigma_* (T )$.
				\item A function $f\in \mathcal{SR}^R(W)$ is said to be \textbf{locally right regular on $\sigma_* (T )$} if there exists a {$T$-\textbf{right}-admissible domain} $U\subseteq \O$ such that $\overline{U}\subseteq W $. We will denote by $ \mathcal{SR}^R({\sigma_*(T)})$ the set of locally right regular functions on $\sigma _*(T )$.
				\item 	A function $f\in \mathcal{SR}_{\R}(W)$ is said to be \textbf{locally slice preserving on $\sigma_* (T )$} if there exists a {$T$-\textbf{right}-admissible domain} $U\subseteq \O$ such that $\overline{U}\subseteq W $. We will denote by $ \mathcal{SR}_{\R}({\sigma_*(T)})$ the set of all locally slice preserving functions on $\sigma_* (T )$.
			\end{enumerate}
		\end{mydef}

		\begin{lemma}\label{cor:HBcor}
			Let $V$ be a Banach $\O$-bimodule and $v\in V$.	If for any octonionic linear functional  $\phi\in V^{*_{\O}}$, $\phi (v)=0$, then $v=0$.
		\end{lemma}
		\begin{proof}
			From Theorem \ref{thm:para_bimodule} we deduce that $ V^{*}$ can be endowed with an $\O$-bimodule structure. Therefore, in view of the real part decomposition  \eqref{eq:redec}, there is a decomposition of any (right) para-linear functional $f\in V^{*}$ as:
			$$f=\sum_{i=0}^{7}e_i\odot f_{(i)},$$
			where $f_{(i)}\in \re V^{*}= V^{*_{\O}}$ 
			for $i=0,\dots,7$.
			By hypothesis,
			$f_{(i)}(v)=0$ for $i=0,\dots,7$. Hence it follows from the vanishing of  the second associators related with octonionic linear maps that
			$$f(v)=\sum_{i=0}^{7}(e_i\odot f_{(i)})(v)=\sum_{i=0}^{7}e_if_{(i)}(v)=0.$$
			An important corollary of  octonionic Hahn-Banach theorem, see \cite[Corollary 1.2.]{huo2025BLMSHB}, implies that
			for any $v\neq 0$ in $V$, there must exist a (right) para-linear functional $g\in V^{*}$ such that $g(v)\neq 0.$
			This forces $v=0$.
		\end{proof}

		\subsection{Octonionic left slice regular functional calculus}
		
		In this subsection, we shall establish the octonionic left slice regular functional calculus. The key point is the use of the octonionic left resolvent operator $(R_s-T)^{\circledcirc -}$.

		\begin{thm}\label{thm:poly left fctcal}
			Let $m\in \mathbb N, $	$T\in \mathscr{B}_{\mathcal{RO}}(V)$ be a power-associative operator and 	$U\subseteq \O$  be a $T$-left-admissible domain. For any $J\in \mathbb S$,   we have
			\begin{eqnarray}\label{eq:intpolyTJ}
				\dfrac{1}{2\pi} \int_{\partial (U\cap \spc_J)}  (R_s-T)^{\circledcirc -}\odot ({d}s_Js^m )=T^m,
			\end{eqnarray}
			where   ${d}s_J = -{d}sJ$.
		\end{thm}
	
\begin{proof}
	{Fix $J\in\mathbb S$. Since $U$ is bounded, choose
	\[
	r>\max\left\{\fsh{T},\sup_{s\in U}|s|\right\},
	\]
	and let $C_r\subset\spc_J$ be the positively oriented circle
	centered at the origin with radius $r$.}
	
{Since $\sigma_J^\ast(T)\subset U\cap\spc_J$, the region between
	$\partial(U\cap\spc_J)$ and $C_r$ is contained in
	$\rho_J^\ast(T)$. Hence, by the classical complex Cauchy theorem
	on the slice $\spc_J$,\[
	\int_{\partial(U\cap\spc_J)}
	(R_s-T)^{\circledcirc-}\odot(ds_Js^m)
	=
	\int_{C_r}
	(R_s-T)^{\circledcirc-}\odot(ds_Js^m).
	\]}
	
			Consider the power series expansion \eqref {eq:powextinvRs-T} for the resolvent operator $(R_s-T)^{\circledcirc -}$, we have 	\begin{eqnarray*}
				\dfrac{1}{2\pi} \int_{\partial (U\cap \spc_J)}  (R_s-T)^{\circledcirc -}\odot ({d}s_Js^m)=	\dfrac{1}{2\pi}\sum_{n\geqslant0} T^n\odot\int_{ C_r}   s^{-1-n+m} {d}s_J =T^m.
			\end{eqnarray*}
		\end{proof}

		\begin{thm}\label{thm:f*TindependentU}
			Let $T\in \mathscr{B}_{\mathcal{RO}}(V)$ be a power-associative operator.	Let $U\subseteq \O$  be a $T$-left-admissible domain, $f\in  {\mathcal{SR}_{\R}}({\sigma^*(T)})$ be a slice preserving function and set ${d}s_J = -{d}sJ$ for $J\in \mathbb S$.  Then the integral
			\begin{eqnarray}\label{eq:intfTJ}
				f^*(T)_J:=\dfrac{1}{2\pi} \int_{\partial (U\cap \spc_J)}  (R_s-T)^{\circledcirc -}\odot ({d}s_Jf(s))
			\end{eqnarray}  does not depend on the choice of  $U$.
			Moreover, the map
			\begin{eqnarray}\label{eqdef:f(T)}
				f^*(T):\mathbb S&\to& \mathscr{B}_{\mathcal {RO}}(V)\\
				J&\mapsto& f^*(T)_J\notag
			\end{eqnarray} is continuous, and  for any $I,J\in \mathbb S$, we have
			$$\operatorname{Re}_{\mathscr{B}_{\mathcal{RO}}(V)}\, f^*(T)_J=\operatorname{Re}_{\mathscr{B}_{\mathcal{RO}}(V)}\, f^* (T)_I.$$

		\end{thm}
		\begin{proof}
		Let us prove that \eqref{eq:intfTJ} does not depend on the choice of $ U$.
			Fixing $J\in \mathbb S$ and $f\in  {\mathcal{SR}_{\R}}({\sigma^*(T)})$, we denote by $T_U$ the operator defined by \eqref{eq:intfTJ} for brevity, so as to emphasize the dependence of the operator on  $U$.
			
			Suppose $U_1,U_2$ are two $T$-left-admissible domains. Since the fixed slice spectrum $\sigma_J^*(T)$ is compact, choose a bounded planar domain $\overline{U_0}\subset (U_1\cap U_2)\cap\mathbb C_J$ with piecewise $C^1$ Jordan boundary and $\sigma_J^*(T)\subset U_0$.
			We aim to prove $T_{U_1}=T_{U_2}$.  Note that both  $T_{U_1}$ and $T_{U_2}$ are right para linear operators. Hence thanks to Lemma \ref{lem:Uniqueness Lemma}, it suffices to show  that
			$$T_{U_1}(v)=T_{U_2}(v)$$ for all $v\in \re V$.
			By Lemma \ref{cor:HBcor}, we only need to show that for any octonionic linear functional $\phi\in V^{*_{\O}}$,
			$$\phi(T_{U_1}(v))=\phi(T_{U_2}(v)).$$
			We define the function
		\begin{equation}\label{gesse-old}
						g_J(s):=\phi\big((R_s-T)^{-1}(v)\big),\qquad s\in\rho_J^*(T).
			\end{equation}
			It follows from Theorem \ref{thm:Rs-Tsliceregular} that $g_J$ is holomorphic on $\rho_J^*(T)$.
			
			By direct calculations, for $i=1,2$,
			\begin{align*}
				\phi(T_{U_i}(v))&=\dfrac{1}{2\pi} \phi\int_{\partial (U_i\cap \spc_J)} \Big( (R_s-T)^{\circledcirc -}\odot ds_Jf(s)\Big)(v)\\
				&=\dfrac{1}{2\pi} \int_{\partial (U_i\cap \spc_J)}\phi \Big( (R_s-T)^{-1}(v) ({d}s_Jf(s))\Big)\\
				&=\dfrac{1}{2\pi} \int_{\partial (U_i\cap \spc_J)}\phi \Big( (R_s-T)^{-1}(v) \Big)({d}s_Jf(s))\\
				&=\dfrac{1}{2\pi} \int_{\partial (U_i\cap \spc_J)}g_J(s){d}s_Jf(s).
			\end{align*}
	Since $f|_{\mathbb C_J}$ is holomorphic and $U_i\setminus\overline{U_0}$ contains no point of $\sigma_J^*(T)$, the ordinary complex Cauchy theorem on the fixed slice gives, componentwise if necessary (there are only finitely many components since the boundary is contained in the finite union \(\partial(U_i\cap\mathbb C_J)\cup\partial U_0\) of Jordan curves),
				\[
				\begin{aligned}
					\phi(T_{U_i}(v))
					&=\frac1{2\pi}\int_{\partial(U_i\cap\mathbb C_J)}g_J(s)\,ds_J\,f(s)\\
					&=\frac1{2\pi}\int_{\partial U_0}g_J(s)\,ds_J\,f(s),\qquad i=1,2.
				\end{aligned}
				\]
			This shows that  \eqref{eq:intfTJ} does not depend on the choice of $ U$.

			We  next show $$\operatorname{Re}_{\mathscr{B}_{\mathcal{RO}}(V)}\, f^*(T)_I=\operatorname{Re}_{\mathscr{B}_{\mathcal{RO}}(V)}\, f^*(T)_J$$ for any $I,J\in \mathbb S$.
			It suffices to show $$(\operatorname{Re}_{\mathscr{B}_{\mathcal{RO}}(V)}\, f^*(T)_I)(v)=(\operatorname{Re}_{\mathscr{B}_{\mathcal{RO}}(V)}\, f^*(T)_J)(v)$$ for all $v\in \re V$.
			Fix $I,J\in \mathbb S$. In view of \eqref{eq:re f}, this is equivalent to
			\begin{eqnarray*}
				\operatorname{Re}_V (f^*(T)_I(v))=\operatorname{Re}_V (f^*(T)_J(v)), \qquad v\in \re V.
			\end{eqnarray*}
			We shall use $\operatorname{Re}$ to represent $\operatorname{Re}_V$ for simplicity in the sequel if no confusion arises.
			Using Lemma \ref{cor:HBcor} again, we only need to prove
			\begin{eqnarray}\label{eqpf:phirefti}
				\phi\Big(	\re (f^*(T)_I(v))\Big)=\phi\Big(	\re (f^*(T)_J(v))\Big), \qquad v\in \re V,
			\end{eqnarray}
			for all $\phi\in V^{*_{\O}}$.
			By \eqref{eq:refx=frex}, \eqref{eqpf:phirefti} is equivalent to
			\begin{eqnarray}\label{eqpf:rephifti}
				\re \phi(f^*(T)_I(v))=	\re \phi(f^*(T)_J(v)), \qquad v\in \re V.
			\end{eqnarray}
		
Fix    a $T$-left-admissible domain $U\subseteq \O$.
		Since $\sigma^*(T)$ is bounded, its closure
	$\overline{\sigma^*(T)}$ is compact. Note that $U$ is a  $T$-left-admissible domain. Hence $U$ is   bounded, and 
	$\partial U$ is also compact. Moreover, it follows from 
	$	\overline{\sigma^*(T)}\subset U$ that 
	$	\overline{\sigma^*(T)}\cap\partial U=\emptyset$ and 
	\[
	\delta:=
	\operatorname{dist}
	\bigl(\overline{\sigma^*(T)},\partial U\bigr)>0.
	\]
		This implies the existence of a bounded axially symmetric open set $U'$ satisfying
		\begin{enumerate}
			\item[a.] $\overline{U'}\subseteq \rho^*(T)$;
			\item[b.] $U'\cap \R\neq \emptyset$;
			\item[c.] $ \partial U\subseteq U '$, and ${\partial (U'\cap \spc_K)}$ consists of a finite number of continuously differentiable Jordan curves  for any $K\in \mathbb S$.
		\end{enumerate}

%
			
			{ We consider again
				\begin{equation}\label{gesse}
					g(s):=\phi\big((R_s-T)^{-1}(v)\big),\qquad \phi\in V^{*_{\O}},\quad v\in\operatorname{Re}V.
			\end{equation}}
	Since $U'$ is a slice-open subset of $\rho^*(T)$, the function $g$
		is right slice regular on $U'$. Moreover, by the resolvent expansion,
		$g(s)\to0$ as $s\to\infty$ on the common exterior region
		$|s|>\|T\|$.
			By the slice Cauchy integral formula \eqref{eq:rslicecauchy}, componentwise if necessary, for any $s\in \partial (U\cap \spc_I)\subseteq U'$  we can represent $g(s) $ as
			\begin{eqnarray}\label{eq:gs}
				g(s)=\dfrac{1}{2\pi} \int_{{\partial (U'\cap \spc_J)}^-}(g(q){d}q_J)\bullet_s^R S_R^{-1}(q,s)
			\end{eqnarray}

			where the boundary ${\partial (U'\cap \spc_J)}^-$ is oriented clockwise, 	 $\bullet_s^R$ denotes the right slice product of functions with variable
			$s$.
			If $a, b \in \R, a < b$, and $q: [a, b] \to\spc_J$ is a piecewise $C^1$ parametrization of the (counterclockwise oriented) Jordan curve $\partial (U'\cap \spc_J)$ in the plane $\spc_J$, then \eqref{eq:gs} becomes
			\begin{eqnarray}\label{eqpf:g(s)0}
				g(s)=-\dfrac{1}{2\pi} \int_{a}^{b}(g(q)\overline{J}q'(t_1))\bullet_s^R S_R^{-1}(q,s){d}t_1,
			\end{eqnarray}
			where $q$ is   the abbreviation of $q(t_1)$ and $q'(t_1)\in \spc_J$ is  the derivative of $q(t_1)$.
			Note that as a function of variable $s$, $g(q)\overline{J}q'(t_1)$ is just a constant function.
			Theorem \ref{thm:slice algebra} implies that the slice preserving function $Q_{s}(q)^{-1}$  is in the nucleus  of  the alternative algebra of slice functions. Combining   \eqref{eq:SR-1} with   Proposition \ref{prop:slicepre}, we get that
			\begin{align}\label{eqpf:g(s)}
				&\quad\, \, (g(q)\overline{J}q'(t_1))\bullet_s^R S_R^{-1}(q,s)\notag\\
				&=	(g(q)\overline{J}q'(t_1))\bullet_s^R \Big((\overline{q}-s)	\bullet_s^R Q_{q}(s)^{-1}\Big)\notag\\
				&=\Big(g(q)\overline{J}q'(t_1)\bullet_s^R (\overline{q}-s)\Big)\bullet_s^R Q_{q}(s)^{-1}\notag\\
				&=\Big((g(q)\overline{J}q'(t_1)) (\overline{q}-s)\Big) Q_{q}(s)^{-1}\notag\\
				&=(g(q)\overline{J}q'(t_1))S_R^{-1}(q,s) +[g(q)\overline{J}q'(t_1), \overline{q}-s, Q_{q}(s)^{-1}].
			\end{align}
			Substituting \eqref{eqpf:g(s)} and identity \eqref{eq:SL=-SR} into \eqref{eqpf:g(s)0}, we  obtain
			\begin{eqnarray}
				g(s)=\dfrac{1}{2\pi} \int_{a}^{b}(g(q)\overline{J}q'(t_1))S_L^{-1}(s,q) -[g(q)\overline{J}q'(t_1), \overline{q}-s, Q_{q}(s)^{-1}]{d}t_1,
			\end{eqnarray}
			
			We next compute $\phi(f^*(T)_I(v))$. Suppose  $c, d \in \R, c < d$, and $s: [c, d] \to\spc_I$ is a piecewise $C^1$ parametrization of the (counterclockwise oriented) Jordan curve $\partial (U\cap \spc_I)$ in the plane $\spc_I$, and $s'(t_2)\in \spc_I$ is  the derivative of $s(t_2)$.
			\begin{align}\label{eqpf:phifTI}
				&\qquad\phi(f^*(T)_I(v))\notag\\
				&=\dfrac{1}{2\pi}\int_{\partial (U\cap \spc_I)}g(s){d}s_If(s)\notag\\
				&=\dfrac{1}{2\pi}\int_{\partial (U\cap \spc_I)}\dfrac{1}{2\pi} \int_{a}^{b}(g(q)\overline{J}q'(t_1))S_L^{-1}(s,q) -[g(q)\overline{J}q'(t_1), \overline{q}-s, Q_{q}(s)^{-1}]{d}t_1{d}s_If(s)\notag\\
				&=\dfrac{1}{4\pi^2}\int_{c}^{d} \int_{a}^{b}(g(q)\overline{J}q'(t_1))S_L^{-1}(s,q) -[g(q)\overline{J}q'(t_1), \overline{q}-s, Q_{q}(s)^{-1}]{d}t_1\big(\overline{I}s'(t_2)f(s)\big){d}t_2.
			\end{align}
			Denote by $A(q)=g(q)\overline{J}q'(t_1),B(s)=\overline{I}s'(t_2)f(s),$ where $q=q(t_1),s=s(t_2)$. Then \eqref{eqpf:phifTI} becomes
			\begin{align}\label{eqpf:phifTI2}
				&\qquad\phi(f^*(T)_I(v))\notag\\
				&=\dfrac{1}{4\pi^2}\int_{c}^{d} \int_{a}^{b}A(q)S_L^{-1}(s,q) -[A(q), \overline{q}-s, Q_{q}(s)^{-1}]{d}t_1B(s){d}t_2\notag\\
				&=\dfrac{1}{4\pi^2}\int_{c}^{d} \int_{a}^{b}A(q)(S_L^{-1}(s,q)B(s))+[A(q),S_L^{-1}(s,q),B(s)] \\
				&\qquad\qquad\qquad-[A(q), \overline{q}-s, Q_{q}(s)^{-1}]B(s){d}t_1{d}t_2.\notag
			\end{align}
			Using the Fubini theorem, the first term becomes
			\begin{align}
				(\text{I})&:=\dfrac{1}{4\pi^2}\int_{c}^{d} \int_{a}^{b}A(q)(S_L^{-1}(s,q)B(s)){d}t_1{d}t_2\notag\\
				&= \dfrac{1}{4\pi^2} \int_{a}^{b}A(q)\int_{c}^{d}(S_L^{-1}(s,q)B(s)){d}t_2{d}t_1\notag\\
				&= \dfrac{1}{2\pi} \int_{\partial (U'\cap \spc_J)}A(q)\left(\dfrac{1}{2\pi}\int_{\partial (U\cap \spc_I)}S_L^{-1}(s,q)ds_If(s)\right)dt_1\notag\\
				&= \dfrac{1}{2\pi} \int_{\partial (U'\cap \spc_J)}A(q)f(q){d}t_1,\label{eqpf:rephifTI}
			\end{align}
			where we used Corollary \ref{cor:slicecauchy} in the last line.
			Substituting $A(q)=g(q)\overline{J}q'(t_1)$ into \eqref{eqpf:rephifTI}, we have
			\begin{align}\label{eqpf:rephifTI2}
				(\text{I})	&= \dfrac{1}{2\pi} \int_{\partial (U'\cap \spc_J)}\big(g(q)\overline{J}q'(t_1)\big)f(q){d}t_1\notag\\
				&=\dfrac{1}{2\pi} \int_{\partial (U'\cap \spc_J)}g(q)dq_Jf(q).
			\end{align}
			Now observe that $\partial (U'\cap \spc_J)$ is positively oriented and surrounds  {$\sigma^*_J(T)$}. By the independence of the integral on the open set, we can substitute $\partial (U'\cap \spc_J)$ by $\partial (U\cap \spc_J)$ in \eqref{eqpf:rephifTI2} and we get
			$$		(\text{I})=	 \phi(f^*(T)_J(v)).$$
			It remains to show
			\begin{eqnarray}\label{eqpf:rephift-I}
				\re \Big(\phi(f^*(T)_I(v))-	(\text{I})\Big)=0.
			\end{eqnarray}
			Noticing $f$ is slice preserving, we obtain that $B(s)\in \spc_I$. Since  $\overline{q}-s, Q_{q}(s)^{-1}\in \mathbb H_{I,J}$, where $\mathbb H_{I,J}$ denotes the associative subalgebra of $\O$ generated by $I,J$, it follows from  identity \eqref{eq:five_term} that
			$$\re [A(q), \overline{q}-s, Q_{q}(s)^{-1}]B(s)=\re A(q)[\overline{q}-s, Q_{q}(s)^{-1},B(s)]=0.$$
			Combining this with \eqref{eqpf:phifTI2}, we   get \eqref{eqpf:rephift-I} as desired.

			Finally, we prove the map \eqref{eqdef:f(T)} is continuous.
			By \eqref{eqpf:phifTI2} and above discussion, for any $\phi\in V^{*_{\O}}, v\in \re V$, $\phi(f^*(T)_I(v))-\phi(f^*(T)_J(v))$ becomes
			\begin{eqnarray}
				\dfrac{1}{4\pi^2}\int_{c}^{d} \int_{a}^{b}[A(q),S_L^{-1}(s,q),B(s)] -[A(q), \overline{q}-s, Q_{q}(s)^{-1}]B(s){d}t_1{d}t_2.
			\end{eqnarray}
			Note that $S_L^{-1}(s,q)$ is continuous both in variable $s\in  \spc_I$ and $q\in \spc_J$. Fix $I\in \mathbb S$,  $s\in \partial (U\cap\spc_I)$ and $q_0,q_1\in \R$ such that for any  $K\in \mathbb S$ $$q_0+q_1K\in {\partial (U'\cap \spc_K)}.$$  Denote for each  $J\in \mathbb S$ that $$q_J:=q_0+q_1J.$$
		{By the bounded choice of $U'$ and the finite piecewise $C^1$ slice-boundary hypothesis, the boundary family used in this parametrization is compact.} 	This, combined with the continuity of
			$S_L^{-1}(s,q)$ in the variable  $q$ (since $s\in \partial (U\cap\spc_I)$), implies that  as 		
			$J\to I$,
			\begin{eqnarray}\label{eqpf:SL-1sq}
				S_L^{-1}(s,q_J)\to S_L^{-1}(s,q_I) \qquad \text{uniformly for  $q_J\in \partial U'$.}
			\end{eqnarray}
			Recall that $A(q)=g(q)\overline{J}q'(t_1)=\phi \Big( (R_q-T)^{-1}(v) \Big)\overline{J}q'(t_1)$ is a function of $t_1$ depending on given $\phi,v$.
			Let us set $$\mathbb S_{V}:=\{v\in V:\fsh{v}=1\},\qquad \mathbb S_{V^{*_{\O}}}:=\{\phi\in V^{*_{\O}}:\fsh{\phi}=1\}.$$
			From the compactness of   $\partial U'$
			and  the continuity of
			$(R_q-T)^{-1}$ with respect to the  variable $q$, it follows   that
			there exists a constant $M$, such that \begin{eqnarray}\label{eqpf:Aqj}
				\qquad\abs{A(q_J)}\leqslant M\abs{q_J'(t_1)} \qquad \text{ for  any $v\in \mathbb S_{V}$, $\phi\in \mathbb S_{V^{*_{\O}}}$ and $q_J\in \partial U'$.}
			\end{eqnarray}
			
			For arbitrary $\varepsilon>0$, it follows from \eqref{eqpf:SL-1sq} that there exists $\delta$ such that for each $J$ satisfying  $\abs{J-I}<\delta$, $\abs{S_L^{-1}(s,q_J)-S_L^{-1}(s,q_I)}<\varepsilon.$ 	Since  $S_L^{-1}(x,y)\in \spc_I$ when $x,y\in \spc_I$, it follows  from $B(s)\in \spc_I$ that $$[A(q_J),S_L^{-1}(s,q_I),B(s)]=0.$$
			Define	$C:=\abs{\partial (U'\cap \spc_J)}$, which is independent of $J$ since $U'$ is axially symmetric.
		Thus we have  \begin{align}
				&\quad\ \abs{\int_{a}^{b}[A(q_J),S_L^{-1}(s,q_J),B(s)]B(s){d}t_1}\notag\\
				&\leqslant	\int_{a}^{b}\abs{[A(q_J),S_L^{-1}(s,q_J),B(s)]B(s)}{d}t_1\notag\\
				&\leqslant	\int_{a}^{b}\abs{[A(q_J),S_L^{-1}(s,q_J)-S_L^{-1}(s,q_I),B(s)]B(s)}{d}t_1\notag\\
				&\qquad+	\int_{a}^{b}\abs{[A(q_J),S_L^{-1}(s,q_I),B(s)]B(s)}{d}t_1\notag\\
				&\leqslant	\int_{a}^{b}2M\abs{q_J'(t_1)}\varepsilon \abs{B(s)}^2{d}t_1\notag\\
				&\leqslant 2M\abs{B(s)}^2 C\varepsilon \label{eqpf:Bs}
			\end{align}
			uniformly for $v\in \mathbb S_{V}$, $\phi\in \mathbb S_{V^{*_{\O}}}$, where we used \eqref{eqpf:Aqj} in the third line.
				This shows that for any fixed $s\in \partial (U\cap\spc_I)$, as $J\to I$,   \begin{align*}
				\int_{a}^{b}[A(q_J),S_L^{-1}(s,q_J),B(s)]B(s){d}t_1\to 0.
			\end{align*}
			Note that $\abs{B(s)}^2$ is integrable. In view of \eqref{eqpf:Bs}, we conclude from Lebesgue dominated convergence Theorem that
			$$\int_{c}^{d}	\int_{a}^{b}[A(q_J),S_L^{-1}(s,q_J),B(s)]B(s){d}t_1{d}t_2\to 0, \qquad \text{uniformly for $v\in \mathbb S_{V}$, $\phi\in \mathbb S_{V^{*_{\O}}}$}.$$
			Similarly, we can also prove that 	as $J\to I$,
			\begin{align*}
				\int_{c}^{d}\int_{a}^{b}[A(q), \overline{q}-s, Q_{q}(s)^{-1}]B(s){d}t_1{d}t_2\to 0 \qquad \text{uniformly for $v\in \mathbb S_{V}$, $\phi\in \mathbb S_{V^{*_{\O}}}$}.
			\end{align*}
			Namely, as $J\to I$,
			$$	\phi(f^*(T)_J(v))\to\phi(f^*(T)_I(v)),\qquad \text{uniformly for $v\in \mathbb S_{V}$, $\phi\in \mathbb S_{V^{*_{\O}}}$}.$$
			By  Lemma \ref{cor:HBcor} again, this implies that as $J\to I$,
			$$f^*(T)_J(v)\to f^*(T)_I(v),\qquad\text{uniformly for }  v\in \re V\cap  \mathbb S_{V},$$
			which proves
			$	\lim_{J\to I}f^*(T)_J=f^*(T)_I$ as desired, as the asserted continuity follows.
		\end{proof}
		
	Denote by $\Gamma(\mathbb S, \mathscr{B}_{\mathcal{RO}}(V))$ the set of continuous right para-linear operator-valued sections on the six dimensional sphere $\mathbb S$. We note that the $\O$-bimodule structure of $\mathscr{B}_{\mathcal{RO}}(V)$   induces an $\O$-bimodule structure on $\Gamma(\mathbb S, \mathscr{B}_{\mathcal{RO}}(V))$.

		
		\begin{mydef}[The octonionic left slice regular functional calculus]\label{def:ofunccal}
			Let $T\in \mathscr{B}_{\mathcal{RO}}(V)$ be a power-associative operator.	Define
			\begin{eqnarray}
				(\Phi_T)_0:{\mathcal{SR}_{\R}}({\sigma^*(T)})&\to& \Gamma(\mathbb S, \mathscr{B}_{\mathcal{RO}}(V))\\
				f&\mapsto& (f^*(T))(J):=f^*(T)_J\notag\\
			 &&\qquad\quad\  \qquad =\dfrac{1}{2\pi} \int_{\partial (U\cap \spc_J)}  (R_s-T)^{\circledcirc -}\odot ({d}s_Jf(s)), \notag
			\end{eqnarray}	
			where $U\subseteq \O$  is a $T$-left-admissible domain.
			
			We then define the octonionic left slice regular functional calculus as $$\Phi_T:=\ext 	(\Phi_T)_0:\mathcal{SR}^L({\sigma^*(T)})\to \Gamma(\mathbb S, \mathscr{B}_{\mathcal{RO}}(V)).$$
		\end{mydef}
		\begin{mydef}\label{def:lsliceinv}
			Let $T\in \mathscr{B}_{\mathcal{RO}}(V)$ be a power-associative operator.	$T$ is called \textbf{(left)  \sinv} if for any $f\in \mathcal{SR}^L({\sigma^*(T)})$, $f^*(T)_J$ is independent of $J\in \mathbb S$.
		\end{mydef}
		
		\begin{rem}\label{rem:sphereinv}
			\begin{enumerate}
				\item 	Let $T\in \mathscr{B}_{\mathcal{RO}}(V)$ be a power-associative operator. 	In view of Theorem \ref{thm:poly left fctcal}, we know that for any  polynomial    $P\in \mathcal{SR}_{\R}({\sigma^*(T)})$, $P^*(T)_J$ is independent of $J\in \mathbb S$. Hence to find a non-{\sinv} operator, one must consider non-polynomial functional calculus.
				\item  The octonionic version of functional calculus of an operator $T$ is an operator-valued function on the six-dimensional sphere $\mathbb S$. We point out that in the quaternionic case, see \cite[Definition 4.10.4]{colombo2011noncomfunctcalculus},
				every quaternionic linear operator is
				``{\sinv}''.   In fact, by Theorem \ref{thm:f*TindependentU}, we have
				$$\re_{\mathscr{B}_{\mathcal{RO}}(V)} (f^*(T)_J-f^*(T)_I)=0$$for any $I,J\in \mathbb S$.
				Considering the real part operator above as a projection onto the nucleus (see Remark \ref{rem:uniform law}), it automatically degenerates into the identity map in the quaternionic case, thereby inducing sphere invariance of $T$.
			\end{enumerate}

		\end{rem}

		\begin{thm}\label{thm:f(T)}
			Let $T\in \mathscr{B}_{\mathcal{RO}}(V)$ be a power-associative operator and	$U\subseteq \O$  be a $T$-left-admissible domain.
			Suppose $f=\sum_{i=0}^7f_{(i)}\bullet^L e_i\in\mathcal{SR}^L({\sigma^*(T)}) $ with $f_{(i)}$ slice preserving for $i=0,\dots,7$. Then 
			\begin{eqnarray}\label{eq:f(T)J}
				f^*(T)_J	&=&\dfrac{1}{2\pi} \int_{\partial (U\cap \spc_J)} 	(R_s-T)^{\circledcirc -}\odot ({d}s_Jf(s)) + \\
				&&	\dfrac{1}{2\pi}\sum_{i=1}^{7} \int_{\partial (U\cap \spc_J)} \big[(R_s-T)^{\circledcirc -},{d}s_Jf_{(i)}(s),e_i\big]_{\mathscr{B}_{\mathcal{RO}}(V)}.\notag
			\end{eqnarray}
			\eqref{eq:f(T)J} still holds when considering any standard orthonormal basis $\{1,J_1,\dots, J_7\}$ instead of $\{1,e_1,\dots, e_7\}$.
			In particular, if $f\in \spc_J(\mathcal{SR}^L({\sigma^*(T)}))$ (see notation \eqref{eq:CJM}), then
			\begin{eqnarray}\label{eq:spcJf(T)J}
				f^*(T)_J=	\dfrac{1}{2\pi} \int_{\partial (U\cap \spc_J)} 	(R_s-T)^{\circledcirc -}\odot ({d}s_Jf(s)),
			\end{eqnarray}
			and for any $v\in \re V$,
			\begin{eqnarray}\label{eqpf:f*TJv}
				f^{*}(T)_J(v)=\dfrac{1}{2\pi} \int_{\partial (U\cap \spc_J)} (R_s-T)^{ -1}(v) \big({d}s_J{f}(s)\big).
			\end{eqnarray}
		\end{thm}
		\begin{proof}
			By definition \eqref{def:ofunccal},
			\begin{align*}
				f^*(T)_J&=(\ext 	(\Phi_T)_0)\left(\sum_{i=0}f_{(i)}\bullet^L e_i\right)\ (J)\\
				&=\sum_{i=0}^7\Big((\ext 	(\Phi_T)_0)(f_{(i)} )\odot e_i\Big)\ (J).
			\end{align*}
			Note that the $\O$-bimodule structure on $\Gamma(\mathbb S, \mathscr{B}_{\mathcal{RO}}(V))$ is induced by the $\O$-bimodule structure on $\mathscr{B}_{\mathcal{RO}}(V)$. Thus $f^*(T)_J$ equals
			\begin{align}
				&\, \quad\sum_{i=0}^7\Big((\ext 	(\Phi_T)_0)(f_{(i)} )(J)\Big)\odot e_i\notag\\
				&=\sum_{i=0}^7\Bigg(\dfrac{1}{2\pi} \int_{\partial (U\cap \spc_J)} (R_s-T)^{\circledcirc -}\odot ds_Jf_{(i)}(s)\Bigg) \odot e_i \notag\\
				&=\dfrac{1}{2\pi}\sum_{i=0}^7 \int_{a}^{b}  (R_s-T)^{\circledcirc -}\odot\Big( \big(\overline{J}s'(t)f_{(i)}(s)\big) e_i \Big)+\big[(R_s-T)^{\circledcirc -},\overline{J}s'(t)f_{(i)}(s),e_i\big]_{\mathscr{B}_{\mathcal{RO}}(V)} dt.\label{eqpf:ftJ}
			\end{align}
			Here in the last line, we suppose  $a,b \in \R, a<b$, and $s: [a,b] \to\spc_J$ is a piecewise $C^1$ parametrization of the (counterclockwise oriented) Jordan curve $\partial (U\cap \spc_J)$ in the plane $\spc_J$, and $s'(t)\in \spc_J$ is  the derivative of $s(t)$.
			Since $f_{(i)}$ is slice preserving for $i=0,\dots,7$, it follows that
			$f_{(i)}(s)\in \spc_J$ for each $s\in \partial (U\cap \spc_J) $. In view of Proposition  \ref{prop:slicepre}, we obtain
			\begin{eqnarray}
				\sum_{i=0}^7	\big(\overline{J}s'(t)f_{(i)}(s)\big) e_i&=&\sum_{i=0}^7 \overline{J}s'(t)f_{(i)}(s) e_i\notag\\
				&=&\sum_{i=0}^7 \overline{J}s'(t)f_{(i)}(s)\bullet^L e_i\notag\\
				&=& \overline{J}s'(t)f(s).\label{eqpf:dsf}
			\end{eqnarray}
			Combining \eqref{eqpf:ftJ} and \eqref{eqpf:dsf}, we obtain $f^*(T)_J$ equals \eqref{eq:f(T)J} as desired.
			
			Let $f\in \spc_J(\mathcal{SR}^L({\sigma^*(T)}))$. Suppose   $$ f=f_{(0)}+f_{(1)}\bullet^L  J$$ with $f_{(0)},f_{(1)}\in \mathcal{SR}_{\R}({\sigma^*(T)} )$. Note that
			$$ \big[(R_s-T)^{\circledcirc -},{d}s_Jf_{(0)}(s),1\big]_{\mathscr{B}_{\mathcal{RO}}(V)}=0,\qquad \big[(R_s-T)^{\circledcirc -},{d}s_Jf_{(1)}(s),J\big]_{\mathscr{B}_{\mathcal{RO}}(V)}=0.$$
			Thus \eqref{eq:f(T)J} implies  \eqref{eq:spcJf(T)J}.
			
			Fix $v\in \re V$ arbitrarily.	We have
			\begin{align*}
				&\, \quad {f}^{*}(T)_J(v)\\
				&=\dfrac{1}{2\pi} \int_{\partial (U\cap \spc_J)}\Big(  (R_s-T)^{\circledcirc -} \odot{d}s_J{f}(s)\Big)(v)  &\text{by \eqref{eq:spcJf(T)J}}\notag\\
				&=\dfrac{1}{2\pi} \int_{\partial (U\cap \spc_J)} (R_s-T)^{\circledcirc -}\Big( \big({d}s_J{f}(s)\big)v\Big)  &\text{using the definition in \eqref{eqdef:fp}}\notag\\
				&=\dfrac{1}{2\pi} \int_{\partial (U\cap \spc_J)} (R_s-T)^{\circledcirc -}\Big( v\big({d}s_J{f}(s)\big)\Big)  &\text{by part(1) of Proposition  \ref{prop:real_part}}\notag\\
				&=	\dfrac{1}{2\pi} \int_{\partial (U\cap \spc_J)} (R_s-T)^{\circledcirc -}(v) \big({d}s_J{f}(s)\big) &\text{since $(R_s-T)^{\circledcirc -}$ is right para-linear}\notag\\
				&=	\dfrac{1}{2\pi} \int_{\partial (U\cap \spc_J)} (R_s-T)^{ -1}(v) \big({d}s_J{f}(s)\big) &\text{by Definition \ref{def:ncirc} and $v\in \re V$}\notag.
			\end{align*}
			This proves \eqref{eqpf:f*TJv} as desired.
		\end{proof}
		
		{
			\begin{rem}
				As we shall prove in Theorem \ref{Cauchyoctonions}, when $T=L_q$, the operator of left multiplication by a nonzero $q\in\mathbb O$, the previous Theorem \ref{thm:f(T)} and Definition \ref{def:ofunccal} are consistent with the Cauchy formula in the octonionic case for slice regular functions, see \cite{Ghiloni2017caot}.
		\end{rem}}
		\subsection{Octonionic right slice regular functional calculus}
		In this subsection, we establish the right slice regular functional calculus for power-associative operators on a Banach $\O$-bimodule $V$.  Since all results herein bear a close similarity to those presented in the previous section, certain proofs are omitted for the sake of brevity. We shall only elaborate on the different parts of the proofs.
		The counterpart of Theorem \ref{thm:poly left fctcal} is the following result:
		\begin{thm}
			Let $m\in \mathbb N, $ 		$T\in \mathscr{B}_{\mathcal{RO}}(V)$ be a power-associative operator and 	$U\subseteq \O$  be a $T$-{right}-admissible domain. For any $J\in \mathbb S$, set ${d}s_J = -{d}sJ$. Then we  have
			\begin{eqnarray}\label{eq:lifpolyTJ}
				\dfrac{1}{2\pi} \int_{\partial (U\cap \spc_J)}  (s^m{d}s_J)\odot(R_s-T)^{-\circledcirc }  =T^m.
			\end{eqnarray}
		\end{thm}
		\begin{thm}\label{thm:main thm2}
			Let $T\in \mathscr{B}_{\mathcal{RO}}(V)$ be a power-associative operator.	Let $U\subseteq \O$  be a $T$-right-admissible domain and $f\in  {\mathcal{SR}_{\R}}({\sigma_*(T)})$ be a slice preserving function. For any $J\in \mathbb S$, set ${d}s_J = -{d}sJ$.  Then the integral
			\begin{eqnarray}\label{eq:lifintfTJ}
				f_*(T)_J:=\dfrac{1}{2\pi} \int_{\partial (U\cap \spc_J)}(f(s){d}s_J)   \odot(R_s-T)^{-\circledcirc }
			\end{eqnarray}  does not depend on the choice of  $U$.
			Moreover, the map
			\begin{eqnarray}\label{eqdef:liff(T)}
				f_*(T):\mathbb S&\to& \mathscr{B}_{\mathcal {RO}}(V)\\
				J&\mapsto& f_*(T)_J\notag
			\end{eqnarray} is continuous, and  for any $I,J\in \mathbb S$, we have
			$$\operatorname{Re}_{\mathscr{B}_{\mathcal{RO}}(V)}\, f_*(T)_J=\operatorname{Re}_{\mathscr{B}_{\mathcal{RO}}(V)}\, f_* (T)_I.$$

		\end{thm}

\begin{proof}
				Let us prove that \eqref{eq:lifintfTJ} does not depend on the choice of $ U$, the rest of the proof follows similarly  to that of Theorem \ref{thm:f*TindependentU} and therefore omitted here.
				
				Fixing $J\in \mathbb S$ and $f\in  {\mathcal{SR}_{\R}}({\sigma_*(T)})$, we denote by $T_U$ the operator defined by \eqref{eq:lifintfTJ} for brevity, so as to emphasize the dependence of the operator on  $U$.
				
				Suppose $U_1,U_2$ are two $T$-right-admissible domains. 
				There exists a 
				bounded  open set  $D_0\subset \spc_J$ such that $\sigma_{*J}(T)\subset D_0 $ and
					$\overline{D_0}\subset (U_1\cap U_2)\cap\spc_J$ for $i=1,2$.
											We aim to prove $T_{U_1}=T_{U_2}$.  Note that both  $T_{U_1}$ and $T_{U_2}$ are right para linear operators. Hence thanks to Lemma \ref{lem:Uniqueness Lemma}, it suffices to show  that
				$$\re_V T_{U_1}(v)=\re_V T_{U_2}(v)$$ for all $v\in V$.
				By Lemma \ref{cor:HBcor}, we only need to show that for any octonionic linear functional $\phi\in V^{*_{\O}}$,
				$$\phi(\re_V T_{U_1}(v))=\phi(\re_V T_{U_2}(v)). $$
				Due to \eqref {eq:refx=frex}, this  is equivalent to show
				\begin{eqnarray}\label{eqpf:reVphiTui}
					\re	\phi( T_{U_1}(v))=\re\phi( T_{U_2}(v))
				\end{eqnarray}
				for any octonionic linear functional $\phi\in V^{*_{\O}}$ and any $v\in V$.

				Fix $J\in \mathbb S$, $\phi\in V^{*_{\O}}$ and  $v\in V$ arbitrarily and define the function $$g(s):=\phi \Big(\pi_J (R_s-T)^{-1}(v) \Big),$$ where $s\in {\rho_*}_J(T)$. {It follows from Theorem \ref{thm:liftRs-Tsliceregular} that $g:{\rho_*}_J(T)\subseteq \spc_J\to \spc_J$ is holomorphic on ${\rho_*}_J(T)$. {Thus}  $g$ is holomorphic on $(U_i\cap \spc_J)\setminus \overline{D_0}$ for $i=1,2$.}

				By direct calculations, for $i=1,2$,
				\begin{align}\label{eqpf:ref-*TUi}
					&\, \quad \pi_J	\phi(T_{U_i}(v))\\
					&=	\phi\circ\pi_J\dfrac{1}{2\pi} \int_{\partial (U_i\cap \spc_J)}\Big((f(s){d}s_J)   \odot(R_s-T)^{-\circledcirc }\Big)(v)  \notag\\
					&=\phi\circ\pi_J\dfrac{1}{2\pi} \int_{\partial (U_i\cap \spc_J)}(f(s){d}s_J)   (R_s-T)^{-\circledcirc }(v) &\text{by definition \eqref{eqdef:pf} and \eqref{eqpf:piJB=0}}\notag\\
					&=\phi\circ\pi_J\dfrac{1}{2\pi} \int_{\partial (U_i\cap \spc_J)}   (R_s-T)^{-\circledcirc }(v)(f(s){d}s_J) &\text{by \eqref{eq:piJ} and $f(s)\in \spc_J$}\notag\\
					&=\phi\circ\pi_J\dfrac{1}{2\pi} \int_{\partial (U_i\cap \spc_J)}   (R_s-T)^{-\circledcirc }(v(f(s){d}s_J)) &\text{by \eqref{eqpf:piJB=0}}\notag\\
					&=\phi\circ\pi_J\dfrac{1}{2\pi} \int_{\partial (U_i\cap \spc_J)}   (R_s-T)^{-1 }(v(f(s){d}s_J)) &\text{by Definition \ref{def:ncirc}}\notag\\
					&=\phi\circ\pi_J\dfrac{1}{2\pi} \int_{\partial (U_i\cap \spc_J)}   (R_s-T)^{-1 }(v) (f(s){d}s_J)&\text{since $s\in {\rho_*}_J(T)$ and $f(s)\in \spc_J$}\notag\\
					&=\dfrac{1}{2\pi} \int_{\partial (U_i\cap \spc_J)} \phi\Big(
					\left(\pi_J  (R_s-T)^{-1 }(v) \right)(f(s){d}s_J)\Big)&\text{by \eqref{eq:piJ} and $f(s)\in \spc_J$}\notag\\
					&=\dfrac{1}{2\pi} \int_{\partial (U_i\cap \spc_J)}   g(s) (f(s){d}s_J)&\text{since  $\phi$ is octonionic linear}.\notag
				\end{align}
				
				Define an open subset $W_i:=(U_i\cap \spc_J)\setminus \overline{D_0}.$ {By definition, we have $f$ {and $g$ are  holomorphic on $W_i$ for $i=1,2$.}   This gives}
				\begin{align*}
					&\qquad\pi_J\phi(T_{U_i}(v))\\
					&=\dfrac{1}{2\pi} \int_{\partial (U_i\cap \spc_J)}g(s)f(s){d}s_J-\dfrac{1}{2\pi} \int_{\partial D_0}g(s)f(s){d}s_J+\dfrac{1}{2\pi} \int_{\partial D_0}g(s)f(s){d}s_J\\
					&=\dfrac{1}{2\pi}\int_{\partial W_i}g(s)f(s){d}s_J+\dfrac{1}{2\pi} \int_{\partial D_0}g(s)f(s){d}s_J\\
					&=\dfrac{1}{2\pi} \int_{\partial D_0}g(s)f(s){d}s_J.
				\end{align*}
				In view of \eqref{eqpf:reVphiTui}, this deduce that  \eqref{eq:lifintfTJ} does not depend on the choice of $ U$.			
			\end{proof}

		Recall the definition of left extension map in Definition \ref{def:lif and ext}.
		\begin{mydef}[The octonionic right slice regular functional calculus]\label{def:lifofunccal}
			Let $T\in \mathscr{B}_{\mathcal{RO}}(V)$ be a power-associative operator.	Define
			\begin{eqnarray}
				\ 	(\Psi_T)_0:{\mathcal{SR}_{\R}}({\sigma_*(T)})&\to& \Gamma(\mathbb S, \mathscr{B}_{\mathcal{RO}}(V))\\
		f&\mapsto& 	{(f_*(T))(J):=f_*(T)_J=\dfrac{1}{2\pi}\int_{\partial(U\cap\spc_J)}(f(s)ds_J)\odot(R_s-T)^{-\circledcirc}}\notag
			\end{eqnarray}	
			where $U\subseteq \O$  is a right $T$-admissible domain.
			
			We then define the octonionic right slice regular functional calculus as the \textbf{left extension} of $	(\Psi_T)_0$: $$\Psi_T:= \operatorname{ext}_L	(\Psi_T)_0:\mathcal{SR}^R({\sigma_*(T)})\to \Gamma(\mathbb S, \mathscr{B}_{\mathcal{RO}}(V)).$$
		\end{mydef}

		\begin{mydef}\label{def:rsliceinv}
			Let $T\in \mathscr{B}_{\mathcal{RO}}(V)$ be a power-associative operator.		$T$ is called \textbf{right \sinv} if for any  $f\in \mathcal{SR}^R({\sigma_*(T)})$, $ f_{*}(T)_J$ is independent of $J\in \mathbb S$.
		\end{mydef}
		\begin{thm}\label{thm:rslice fc}
			Let $T\in \mathscr{B}_{\mathcal{RO}}(V)$ be a power-associative operator and 	$U\subseteq \O$  be a  $T$-right-admissible domain.
			Suppose $f=\sum_{i=0}^7e_i\bullet^R f_{(i)}\in\mathcal{SR}^R({\sigma_*(T)}) $ with $f_{(i)}$ slice preserving for $i=0,\dots,7$. Then
			\begin{eqnarray}\label{lifeq:f(T)J}
				f_*(T)_J	&=&\dfrac{1}{2\pi} \int_{\partial (U\cap \spc_J)} (f(s){d}s_J)\odot {(R_s-T)^{-\circledcirc }} - \\
				&&	{\dfrac{1}{2\pi}\sum_{i=1}^{7} \int_{\partial (U\cap \spc_J)} \big[e_i,f_{(i)}(s){d}s_J,(R_s-T)^{-\circledcirc }\big]_{\mathscr{B}_{\mathcal{RO}}(V)}.}\notag
			\end{eqnarray}
			\eqref{lifeq:f(T)J} still holds when substituting any standard orthonormal basis $\{1,J_1,\dots, J_7\}$ for $\{1,e_1,\dots, e_7\}$.
			In particular, if $f\in \spc_J(\mathcal{SR}^R({\sigma_*(T)}))$ (see \eqref{eq:CJM} and Theorem \ref{thm:slice algebra}), then
			\begin{eqnarray}\label{eq:lifspcJf(T)J}
				f_*(T)_J=	\dfrac{1}{2\pi} \int_{\partial (U\cap \spc_J)} (f(s){d}s_J)\odot 	{(R_s-T)^{-\circledcirc }}
			\end{eqnarray}
			and for any $v\in V$,
			\begin{eqnarray}\label{eqpf:ref-*T}
				\operatorname{Re}_{{V}}f_{*}(T)_J(v)=\operatorname{Re}_{{V}}\dfrac{1}{2\pi} \int_{\partial (U\cap \spc_J)}   (R_s-T)^{-1 }(v)(f(s){d}s_J).
			\end{eqnarray}
		\end{thm}
		\begin{proof}
			Formulas \eqref{lifeq:f(T)J} and \eqref{eq:lifspcJf(T)J} can be proved using the same method as in Theorem \ref{thm:f(T)}, only noting that the octonionic right slice regular functional calculus $\Psi_T$ is \textbf{left para-linear} by definition.
			
			We next establish  \eqref{eqpf:ref-*T}.	By direct calculations, for any $v\in V$, we have
			\begin{align}\label{eqpf:ref-*T2}
				&\, \quad \operatorname{Re}_{{V}}f_{*}(T)_J(v)\\
				&=	\operatorname{Re}_{{V}}\dfrac{1}{2\pi} \int_{\partial (U\cap \spc_J)}\Big((f(s){d}s_J)   \odot(R_s-T)^{-\circledcirc }\Big)(v)  &\text{by \eqref{eq:lifspcJf(T)J}}\notag\\
				&=\operatorname{Re}_{{V}}\dfrac{1}{2\pi} \int_{\partial (U\cap \spc_J)}(f(s){d}s_J)   (R_s-T)^{-\circledcirc }(v) &\text{using  \eqref{eqdef:pf}}\notag\\
				&=\operatorname{Re}_{{V}}\dfrac{1}{2\pi} \int_{\partial (U\cap \spc_J)}   (R_s-T)^{-\circledcirc }(v)(f(s){d}s_J) &\text{by part(2) of Proposition  \ref{prop:real_part}}\notag\\
				&=\operatorname{Re}_{{V}}\dfrac{1}{2\pi} \int_{\partial (U\cap \spc_J)}   (R_s-T)^{-\circledcirc }(v(f(s){d}s_J)) &\text{since $(R_s-T)^{-\circledcirc }$ is right para-linear}\notag\\
				&=\operatorname{Re}_{{V}}\dfrac{1}{2\pi} \int_{\partial (U\cap \spc_J)}   (R_s-T)^{-1 }(v(f(s){d}s_J)) &\text{by Definition \ref{def:ncirc}}.\notag
			\end{align}
			It follows from $f\in \spc_J(\mathcal{SR}^R({\sigma_*(T)})$ that $f(s)\in \spc_J$ for any $s\in \partial (U\cap \spc_J)$. Note that $\partial (U\cap \spc_J)\subseteq \rho_{*J}(T)$.
			By definition \ref{def:pushforward sepc}, we have $$\operatorname{Re}_{{V}}\int_{\partial (U\cap \spc_J)} (R_s-T)^{-1 }(v(f(s){d}s_J))=\operatorname{Re}_{{V}}\int_{\partial (U\cap \spc_J)} \big((R_s-T)^{-1 }(v)\big)(f(s){d}s_J).$$
			Combining this with \eqref{eqpf:ref-*T2}, we get \eqref{eqpf:ref-*T} as desired.
		\end{proof}
		
		\section{Examples}\label{sec:eg}
		We only consider examples of  the pull-back spectrum for octonionic right para-linear operators. The examples of  the push-forward  spectrum can be formulated similarly.
		\subsection{Power-associative operators}
		
		$\newline$
		Recall the definition \eqref{eq:CJM} of  ${\mathbb C_J} $-slice.
		\begin{lemma}\label{lem:TS=TrS}
			Let  $T,S\in \spc_J(\mathscr{B}_{\mathcal{RO}}(V))$. Then  $$TS=T\circledcirc S.$$
		\end{lemma}
		\begin{proof}
			Let $$T=T_{0}+J\odot T_{1},\qquad S=S_{0}+J\odot S_{1},$$
			where $T_{0},T_{1},S_{0},S_{1}\in \re \mathscr{B}_{\mathcal{RO}}(V)=\mathscr{B}_{\mathbb{O}}(V)$ are octonionic linear operators.
			Thus for any $x\in V$,  it follows from identity \eqref{eqpf:tensor BROV} that
			\begin{align*}
				(	T\circledcirc S)(x)&=(T_0S_0+J\odot(T_0S_1+T_1S_0)-T_1S_1)(x)\\
				&=(T_0S_0)(x)+(J\odot(T_0S_1+T_1S_0))(x)-(T_1S_1)(x)\\
				&=T_0(S_0(x))+J\big(T_0(S_1(x))+T_1(S_0(x))\big)-T_1(S_1(x))\\
				&=T(S(x)).
			\end{align*}This completes the proof.
		\end{proof}

		\begin{prop}\label{prop:CJispass}
			Any $T\in \spc_J(\mathscr{B}_{\mathcal{RO}}(V))$ is power-associative.
		\end{prop}
		\begin{proof}
			We prove $T^n=T^{\circledcirc n}$ for any integer $n$ by induction.
			It clearly holds for $n=1$. Suppose $T^n=T^{\circledcirc n}$ for $n$. Since $T\in \spc_J(\mathscr{B}_{\mathcal{RO}}(V))$, it follows from identity \eqref{eqpf:tensor BROV} that   $T^{\circledcirc n}\in \spc_J(\mathscr{B}_{\mathcal{RO}}(V))$. Thus
			$$T^{n+1}=T\circ T^n=T\circ T^{\circledcirc n}.$$
			In view of Lemma \ref{lem:TS=TrS} and identity \eqref{eq:Tnext}, we deduce
			\begin{align*}
				T^{n+1}=T\circledcirc  T^{\circledcirc n}= T^{\circledcirc (n+1)}
			\end{align*}
			as desired.
			This shows that $T^n$ is right para-linear for any $n$, i.e., $T$ is power-associative.		
		\end{proof}
		
		\begin{prop}\label{prop:commutativepowass}
			Let $T=\sum_{i=0}^7e_i\odot T_i\in\mathscr{B}_{\mathcal{RO}}(V)$, where $T_i\in \mathscr{B}_{\mathbb O}(V)$ for $i=0,\dots, 7$. If $T_iT_j=T_jT_i$ for all $i,j=0,\dots, 7$, then $T$ is power-associative.		
		\end{prop}
		\begin{proof}
			Denote by $(e_{j})_{\otimes_r}:=e_j$ for any $j=0,\dots,7$, and for $n>1$, $$(e_{j_1}\cdots e_{j_n})_{\otimes_r}:=e_{j_1}(e_{j_2}\cdots e_{j_n})_{\otimes_r},$$
			where $j_1,\dots,j_n\in \{0,1,2,\dots,7\}$.
			
			We first show that \begin{eqnarray}\label{eq:TOn=ej1ejnTj}
				T^{\circledcirc n}=\sum_{j_1,\dots,j_n}(e_{j_1}\cdots e_{j_n})_{\otimes_r}\odot (T_{j_1}T_{j_2}\cdots T_{j_n}).
			\end{eqnarray}
			Equality \eqref {eq:TOn=ej1ejnTj} clearly holds  for $n=1$. Suppose 	\eqref {eq:TOn=ej1ejnTj} also holds for case $n$. It follows from identities \eqref{eq:Tnext}  and  \eqref{eqpf:tensor BROV} that
			\begin{eqnarray*}
				T^{\circledcirc (n+1)}&=&T\circledcirc 	T^{\circledcirc (n)}\\
				&=&\sum_{j_1=0}^7(e_{j_1}\odot T_{j_1})\circledcirc\sum_{j_2,\dots,j_{n+1}}(e_{j_2}\cdots e_{j_{n+1}})_{\otimes_r}\odot (T_{j_2}\cdots T_{j_{n+1}})\\
				&=&\sum_{j_1,j_2,\dots,j_{n+1}}e_{j_1}(e_{j_2}\cdots e_{j_{n+1}})_{\otimes_r}\odot T_{j_1}(T_{j_2}\cdots T_{j_{n+1}})\\
				&=&\sum_{j_1,\dots,j_{n+1}}(e_{j_1}\cdots e_{j_{n+1}})_{\otimes_r}\odot (T_{j_1}T_{j_2}\cdots T_{j_{n+1}}).
			\end{eqnarray*}
			
			We next show that \begin{eqnarray}\label{eqpf:[t,TOn,x]=0}
				[T,T^{\circledcirc n},x]_{\circledcirc}=0
			\end{eqnarray}
			for all $x\in V$, for all integer $n$.
			If $n=0$, then $T^{\circledcirc n}=\mathcal I$ and hence \eqref{eqpf:[t,TOn,x]=0} holds.
			If $n=1$, then \begin{align*}
				[T,T^{\circledcirc n},x]_{\circledcirc}&=(T\circledcirc T)(x)-T^2(x)\\
				&=\sum_{j_1,j_2}(e_{j_1}e_{j_2})(T_{j_1}T_{j_2}(x))-\sum_{j_1=0}^7e_{j_1}\odot T_{j_1}\left(\sum_{j_2=0}^7e_{j_2}T_{j_2}(x)\right)\\
&={\sum_{j_1,j_2=0}^{7}\Big((e_{j_1}e_{j_2})(T_{j_1}T_{j_2}(x))-e_{j_1}\big(e_{j_2}T_{j_1}T_{j_2}(x)\big)\Big)}\\
&={\sum_{j_1,j_2=1}^{7}[e_{j_1},e_{j_2},T_{j_1}T_{j_2}(x)]},
			\end{align*}
{Here the component expansions run over $0,\dots,7$, while the terms with $j_1=0$ or $j_2=0$ vanish in the displayed associator because $e_0=1$.}
			Exchanging $j_1$ and $j_2$ and using $T_{j_1}T_{j_2}=T_{j_2}T_{j_1}$, we deduce \eqref{eqpf:[t,TOn,x]=0} holds for $n=1$.
			Suppose \eqref{eqpf:[t,TOn,x]=0} also holds for all  integers  $ 1,\dots,n-1$ ($n\geqslant 2$). In view of \eqref{eq:TOn=ej1ejnTj}, we have
			\begin{align*}
				&\qquad[T,T^{\circledcirc n},x]_{\circledcirc}\\
				&=T^{\circledcirc (n+1)}(x)-T(T^{\circledcirc n}(x))\\
				&=\sum_{j_1,\dots,j_{n+1}}\Big((e_{j_1}\cdots e_{j_{n+1}})_{\otimes_r}\odot (T_{j_1}\cdots T_{j_{n+1}})\Big)(x)-\\
				&\qquad\sum_{j_1=0}^7(e_{j_1}\odot T_{j_1})\left(\sum_{j_2,\dots,j_{n+1}}\big((e_{j_2}\cdots e_{j_{n+1}})_{\otimes_r}\odot (T_{j_2}\cdots T_{j_{n+1}})\big)(x)\right)\\
				&=\sum_{j_1,\dots,j_{n+1}}\Big(e_{j_1}(e_{j_2}\cdots e_{j_{n+1}})_{\otimes_r}\Big) (T_{j_1}\cdots T_{j_{n+1}})(x)-\\
				&\qquad\sum_{j_1,\dots,j_{n+1}}e_{j_1}\big((e_{j_2}\cdots e_{j_{n+1}})_{\otimes_r}  T_{j_1}(T_{j_2}\cdots T_{j_{n+1}})(x)\big)\\
				&=-\sum_{j_1,\dots,j_{n-1}}[e_{j_1},(e_{j_2}\cdots e_{j_{n-1}})_{\otimes_r}, (T_{j_1}\cdots T_{j_{n-1}})(T_{j_n}T_{j_{n+1}}x)]+\\
&\qquad{\sum_{\substack{j_1,\dots,j_{n-1}=0,\dots,7\\ j_n,j_{n+1}=1,\dots,7,\ j_n\neq j_{n+1}}}[e_{j_1},(e_{j_2}\cdots e_{j_{n+1}})_{\otimes_r}, (T_{j_1}\cdots T_{j_{n+1}})(x)]}.
			\end{align*}
			By induction hypothesis, the first term in the last equality vanishes. {The contributions in which one of the last scalar indices is $0$ reduce, because $e_0=1$ and $T_0$ commutes with all components, to lower-order associator sums covered by the induction hypothesis. Thus the remaining nontrivial paired terms have $j_n,j_{n+1}\in\{1,\dots,7\}$.} It is easy to verify that if $j_n\neq j_{n+1}$, then $$(e_{j_2}\cdots e_{j_n}e_{j_{n+1}})_{\otimes_r}=-(e_{j_2}\cdots e_{j_{n-1}}e_{j_{n+1}}e_{j_n})_{\otimes_r}.$$
			Exchanging $j_n$ and $j_{n+1}$ and using $T_{j_n}T_{j_{n+1}}=T_{j_{n+1}}T_{j_n}$, we deduce \eqref{eqpf:[t,TOn,x]=0} holds for case $n$.
			
			Finally, we conclude from Lemma \ref{lem:powass} that $T$ is power-associative.
		\end{proof}
		\subsection{1-dimensional  case}
		We  consider the one-dimensional case $V=\O$. We compute the functional calculus  of the operator $$L_q:\O\to\O$$ of left multiplication by a fixed nonzero $q\in \O$. Clearly, $L_q\in \spc_J(\mathscr{B}_{\mathcal{RO}}(\O))$ and by Proposition \ref{prop:CJispass} we have $L_q$ is power-associative. Moreover, one can check that for all $p,q\in \O$,
		\begin{eqnarray}\label{eq:Lqp}
			L_q\odot p=L_{qp}.
		\end{eqnarray}
		
		For $s,q\in \O$, recall the notation \begin{eqnarray}
			Q_{s}(q):=q^2-2(\re s)q+\abs{s}^2
		\end{eqnarray}
		and  $$[q]:=\{a+bJ:a=\re q, b=\abs{\text{Im}q},J\in \mathbb S\}.$$
		Let us denote by $\mathbb H_{s,q}$ the subalgebra of $\mathbb O$ generated by $s$ and $q$.
		
		\begin{lemma}For any $s,q\in \O$, we have
			\begin{eqnarray}\label{eq:1dimLdeltasq}
				L_{	Q_{s}(q)}=(R_s-L_q)(R_{\overline{s}}-L_q)+[L_q,R_{\overline{s}}],
			\end{eqnarray}	
			where the commutator $[L_q,R_{\overline{s}}]$ is defined by
			$$[L_q,R_{\overline{s}}]:=L_qR_{\overline{s}}-R_{\overline{s}}L_q=-[q,\cdot,\overline{s}].$$
			In particular, \begin{eqnarray}\label{eq:1dimLdeltasqHsq}
				L_{	Q_{s}(q)}|_{\mathbb H_{s,q}}=(R_s-L_q)(R_{\overline{s}}-L_q)|_{\mathbb H_{s,q}}.
			\end{eqnarray}	
		\end{lemma}	
		\begin{proof}
			The claim follows by direct calculations.
			For any $x\in \O$,
			\begin{align*}
				&\qquad (R_s-L_q)(R_{\overline{s}}-L_q)(x)+[L_q,R_{\overline{s}}](x)\\
				&= (R_s-L_q)(x{\overline{s}-qx})+q(x{\overline{s}})-(qx)\overline{s}\\
				&=x\abs{s}^2-(qx)s-q(x\overline{s})+q^2x+q(x{\overline{s}})-(qx)\overline{s}\\
				&=	Q_{s}(q)\ x.
			\end{align*}
			This completes the proof.
		\end{proof}
		
		\begin{lemma}\label{lem:1dimRs-Lqinv}
			$(R_s-L_q)$ is invertible {as real linear operator} if and only  if $	s\notin [q]$.	
		\end{lemma}
		\begin{proof}
			The operator $R_s-L_q$ can be represented by a real matrix from $\R^8$ to itself. From \cite[Theorem 2.14]{tian2000matrixO} the determinant of the real matrix $R_s-L_q$ is
			$${\det} _{\R} (R_s-L_q)=\abs{q-\overline{s}}^4\Big(\big(\re (q-s)\big)^2+\abs{\text{Im} q}^2+\abs{\text{Im} s}^2\Big)\Big(\big(\re (q-s)\big)^2+(\abs{\text{Im} q}-\abs{\text{Im} s})^2\Big).$$
			This implies that $\det (R_s-L_q)=0$ if and only if $\re q=\re s$ and $\abs{\text{Im} q}=\abs{\text{Im} s}$, i.e., {$	s\in [q]$}.	
		\end{proof}

		\begin{lemma}
			If $	s\notin [q]$, then
			\begin{eqnarray}\label{eq:Rs-LqinvHsq}
				(R_s-L_q)^{-1}|_{\mathbb H_{s,q}}=L_{{Q_{s}(q)}^{-1}}(R_{\overline{s}}-L_q)|_{\mathbb H_{s,q}}.
			\end{eqnarray}
			In particular, suppose $s\in \mathbb C_{J} $ for some $J\in \mathbb S$ and  $	Q_{s}(q)\neq 0$, we have
			\begin{eqnarray}\label{eq:Rs-LqinvCs}
				(R_s-L_q)^{-1}|_{\mathbb C_{J}}=L_{{Q_{s}(q)}^{-1}({\overline{s}}-q)}|_{\mathbb C_{J}}
			\end{eqnarray}
			and hence
			\begin{eqnarray}\label{eq:(Rs-Lq)circledcirc-}
				(R_s-L_q)^{\circledcirc -}=L_{{Q_{s}(q)}^{-1}({\overline{s}}-q)}.
			\end{eqnarray}
			
		\end{lemma}
		\begin{proof}
			By Lemma \ref{lem:1dimRs-Lqinv}, 	$R_s-L_q$ is invertible since 	$	Q_{s}(q)\neq 0$.
			Note that $(R_{{s}}-L_q)|_{\mathbb H_{s,q}},(R_{\overline{s}}-L_q)|_{\mathbb H_{s,q}},{L_{{Q_{s}(q)}}}|_{\mathbb H_{s,q}}$ are all maps from ${\mathbb H_{s,q}}$ to itself and hence so are their inverses.
			Thus it follows from identity \eqref{eq:1dimLdeltasqHsq} that
			$${L_{{Q_{s}(q)}}}^{-1}|_{\mathbb H_{s,q}}=(R_{{s}}-L_q)^{-1}|_{\mathbb H_{s,q}}(R_{\overline{s}}-L_q)^{-1}|_{\mathbb H_{s,q}}.$$ Therefore for any $x\in \mathbb H_{s,q}$, we have
			\begin{align*}
				L_{{Q_{s}(q)}^{-1}}(R_{\overline{s}}-L_q)(x)&=(R_{{s}}-L_q)^{-1}|_{\mathbb H_{s,q}}(R_{\overline{s}}-L_q)^{-1}|_{\mathbb H_{s,q}}(R_{\overline{s}}-L_q)(x)\\
				&=(R_{{s}}-L_q)^{-1}|_{\mathbb H_{s,q}}(x).
			\end{align*} This proves \eqref{eq:Rs-LqinvHsq}.
			
			Equality \eqref{eq:Rs-LqinvCs} follows from \eqref{eq:Rs-LqinvHsq} immediately since $s\in \mathbb C_{J} $, for any $x\in \mathbb C_{J}$ we have  $$R_{\overline{s}}x={\overline{s}}x.$$
			\eqref{eq:(Rs-Lq)circledcirc-} follows from the fact that $L_{{Q_{s}(q)}^{-1}({\overline{s}}-q)}$ is right para-linear and both sides of \eqref{eq:(Rs-Lq)circledcirc-} coincide on the real numbers.
		\end{proof}
		
		\begin{thm}\label{thm:sigLq=Sq}
			Let $L_q:\O\to \O$ be left multiplication by a fixed nonzero
			octonion $q$. Then
			\[
			\sigma^*(L_q)=[q].
			\]
		\end{thm}
		
		\begin{proof}
					We first determine the pull-back spectrum on each slice.
				We claim that
				\[
				\rho_J^*(L_q)
				=
				\{s\in \spc_J:s\notin [q]\}
				\]
				for every $J\in\mathbb S$.

			Fix $J\in \mathbb S$ arbitrarily and consider any
			$s\in \spc_J$ such that $s\notin [q]$.
			It follows from Lemma \ref{lem:1dimRs-Lqinv} that
			$R_s-L_q$ is invertible.
			
			By associativity, the operator
			\[
			L_{Q_s(q)}|_{\mathbb H_{s,q}}:
			\mathbb H_{s,q}\to\mathbb H_{s,q}
			\]
			is right $\spc_J$-linear, i.e., for all
			$x\in\mathbb H_{s,q}$ and all $\lambda\in\spc_J$,
			\[
			L_{Q_s(q)}(x\lambda)
			=
			L_{Q_s(q)}(x)\lambda.
			\]
			It is easy to check that
			\[
			(R_{\overline{s}}-L_q)|_{\mathbb H_{s,q}}:
			\mathbb H_{s,q}\to\mathbb H_{s,q}
			\]
			is also right $\spc_J$-linear.
			Thus by identity \eqref{eq:Rs-LqinvHsq} we deduce that
			\[
			(R_s-L_q)^{-1}|_{\mathbb H_{s,q}}:
			\mathbb H_{s,q}\to\mathbb H_{s,q}
			\]
			is right $\spc_J$-linear.

				This implies that
				$(R_s-L_q)^{-n}|_{\mathbb H_{s,q}}$
				is also right $\spc_J$-linear for every integer $n\geq 1$.

			Therefore for all $x\in \spc_J(\O)=\spc_J$,
			$\lambda\in\spc_J$, and $n\geq1$, we have
			\[
			\begin{aligned}
				(R_s-L_q)^{-n}(x\lambda)
				&=(R_s-L_q)^{-n}|_{\mathbb H_{s,q}}(x\lambda)\\
				&=(R_s-L_q)^{-n}|_{\mathbb H_{s,q}}(x)\lambda\\
				&=(R_s-L_q)^{-n}(x)\lambda.
			\end{aligned}
			\]
			This proves that $(R_s-L_q)^{-1}$ is
			$\mathbb C_J$-extendable power associative.
				This means that $s\in \rho_J^*(L_q)$.

			Conversely, if $s\in [q]\cap \spc_J$, then by
			Lemma \ref{lem:1dimRs-Lqinv} we have that
			$R_s-L_q$ is not invertible and hence

				$s\notin \rho_J^*(L_q)$.
				Hence, for every $J\in\mathbb S$,
				\[
				\rho_J^*(L_q)
				=
				\{s\in \spc_J:s\notin [q]\},
				\qquad
				\sigma_J^*(L_q)
				=
				[q]\cap\spc_J.
				\]
				Therefore, by Definition \ref{def:pullback spec},
				\[
				\begin{aligned}
					\sigma^*(L_q)
					&=
					\bigcup_{J\in\mathbb S}\sigma_J^*(L_q)\\
					&=
					\bigcup_{J\in\mathbb S}
					\bigl([q]\cap\spc_J\bigr)\\
					&=[q].
				\end{aligned}
				\]
					\end{proof}

		\begin{thm}\label{Cauchyoctonions}
			For all $f\in\mathcal{SR}^L({\sigma^*(L_q))}$, we have $f^*(L_q)_J=L_{f(q)}$ for all $J\in \mathbb S$.
		\end{thm}
		\begin{proof}
			Let	$U\subseteq \O$  be a $L_q$-left-admissible domain.  Fix $J\in \mathbb S$ arbitrarily.
			
			We begin with proving $g^*(L_q)_J=L_{g(q)}$ for all $J\in \mathbb S$ whenever $g\in\mathcal{SR}_{\R}({\sigma^*(L_q)}$.			
			By Definition \ref{def:ofunccal} and identity \eqref{eq:(Rs-Lq)circledcirc-}, we have
			\begin{align*}
				g^*(L_q)_J&=\dfrac{1}{2\pi} \int_{\partial (U\cap \spc_J)}  L_{{Q_{s}(q)}^{-1}({\overline{s}}-q)}\odot {d}s_Jg(s)\\
				&=\dfrac{1}{2\pi} L_{\int_{\partial (U\cap \spc_J)}  {Q_{s}(q)}^{-1}({\overline{s}}-q) ({d}s_Jg(s))}\\
				&=L_{g(q)}.
			\end{align*}
			Here the second equality follows from 		 \eqref	{eq:Lqp}. The last equality follows from the Cauchy integral formula \ref{eq:slicecauchy}.

			Let $f=\sum_{i=0}f_{(i)}\bullet^L e_i\in\mathcal{SR}^L({\sigma^*(T)}) $ with $f_{(i)}$ slice preserving for $i=0,\dots,7$.
			By  definition and direct calculations, we get
			\begin{align*}
				f^*(L_q)_J&=\left(\sum_{i=0}^7f_{(i)}\bullet^L e_i\right)(L_q)_J\\
				&=\sum_{i=0}^7f_{(i)}(L_q)_J\odot e_i\\
				&=\sum_{i=0}^7L_{f_{(i)}(q)}\odot e_i\\
				&=\sum_{i=0}^7L_{f_{(i)}(q) e_i}\\
				&=L_{f(q)}.
			\end{align*}
			The last equality follows from  \eqref{eq:sliceprod}:
			$$f(q)=\sum_{i=0}^7(f_{(i)}\bullet^L e_i)(q)=\sum_{i=0}^7f_{(i)}(q) e_i.$$
		\end{proof}
		
		\subsection{Matrix cases}
		One can verify that any right para-linear operator on $\O^n$  can be represented as an octonionic matrix acting on the octonionic column vectors. The regular composition is just the usual matrix product. Below we illustrate some examples.
		
		\begin{eg}
			Let $V=\O^3$. Consider the matrix
			\[
			T=
			\begin{bmatrix}
				e_1&0&0\\
				0&2e_2&0\\
				0&0&3e_4
			\end{bmatrix}.
			\]
			Since
			\[
			T=
			e_1\odot
			\begin{bmatrix}
				1&0&0\\
				0&0&0\\
				0&0&0
			\end{bmatrix}
			+
			e_2\odot
			\begin{bmatrix}
				0&0&0\\
				0&2&0\\
				0&0&0
			\end{bmatrix}
			+
			e_4\odot
			\begin{bmatrix}
				0&0&0\\
				0&0&0\\
				0&0&3
			\end{bmatrix},
			\]
			it follows from Proposition \ref{prop:commutativepowass}
			that $T$ is power associative.
			
			Note that, in general, $R_s-T$ may be not right para-linear
			and hence is not necessarily expressed as an octonionic matrix.
			Considering $R_s-T$ as a real linear operator, we can express
			$R_s-T$ as the real block matrix:
			\[
			R_s-T=
			\begin{bmatrix}
				R_s-L_{e_1}&0&0\\
				0&R_s-2L_{e_2}&0\\
				0&0&R_s-3L_{e_4}
			\end{bmatrix}.
			\]
			
			By Lemma \ref{lem:1dimRs-Lqinv}, $R_s-T$ is invertible
			if and only if
			\[
			s\notin \mathbb S\cup 2\mathbb S\cup 3\mathbb S.
			\]
			
			For any
			$s\notin \mathbb S\cup 2\mathbb S\cup 3\mathbb S$,
			suppose $s\in\spc_J$ for some $J\in\mathbb S$.
			It follows from the proof of Theorem \ref{thm:sigLq=Sq}
			that
			\[
			(R_s-L_{e_1})^{-1},\qquad
			(R_s-2L_{e_2})^{-1},\qquad
			(R_s-3L_{e_4})^{-1}
			\]
			are all $\spc_J$-extendable power associative.

				Hence, for every $J\in\mathbb S$,
				\[
				\rho_J^*(T)
				=
				\spc_J\setminus
				\bigl(
				\mathbb S\cup2\mathbb S\cup3\mathbb S
				\bigr),
				\]
				and therefore
				\[
				\sigma_J^*(T)
				=
				\spc_J\cap
				\bigl(
				\mathbb S\cup2\mathbb S\cup3\mathbb S
				\bigr).
				\]
				Consequently, by Definition \ref{def:pullback spec},
				\[
				\begin{aligned}
					\sigma^*(T)
					&=
					\bigcup_{J\in\mathbb S}\sigma_J^*(T)\\
					&=
					\bigcup_{J\in\mathbb S}
					\left[
					\spc_J\cap
					\bigl(
					\mathbb S\cup2\mathbb S\cup3\mathbb S
					\bigr)
					\right]\\
					&=
					\mathbb S\cup2\mathbb S\cup3\mathbb S.
				\end{aligned}
				\]
						\end{eg}

		\begin{eg}\label{eg:non sphere}
			Let $V=\O^2$. Consider the matrix
			\[
			T=
			\begin{bmatrix}
				0&-e_1\\
				e_1&0
			\end{bmatrix}.
			\]
			Obviously
			$T\in \spc_{e_1}(\mathscr{B}_{\mathcal{RO}}(V))$ and,
			by Proposition \ref{prop:CJispass}, $T$ is power-associative.
			
			We next compute $\sigma^*(T)$.
			We first express $R_s-T$ as the real block matrix:
			\[
			R_s-T=
			\begin{bmatrix}
				R_s&L_{e_1}\\
				L_{\overline{e_1}}&R_s
			\end{bmatrix}.
			\]
			
			\textbf{Claim 1:}
			Let $s\in\spc_J$ for some $J\in\mathbb S$.
			If $R_s-T$ is invertible, then $(R_s-T)^{-1}$
			is $\spc_J$-extendable power-associative.
			
			\begin{proof}[\textbf{Proof of Claim 1:}]
				We note that
				\[
				(R_s-T)|_{\mathbb H_{J,e_1}(V)}
				=
				\begin{bmatrix}
					R_s&L_{e_1}\\
					L_{\overline{e_1}}&R_s
				\end{bmatrix}_{\mathbb H_{J,e_1}(V)}
				:
				\mathbb H_{J,e_1}(V)
				\to
				\mathbb H_{J,e_1}(V)
				\]
				is a right $\spc_J$-linear operator. Hence
				\[
				(R_s-T)^{-1}|_{\mathbb H_{J,e_1}(V)}
				:
				\mathbb H_{J,e_1}(V)
				\to
				\mathbb H_{J,e_1}(V)
				\]
				is also right $\spc_J$-linear.
					Hence so is
					$(R_s-T)^{-n}|_{\mathbb H_{J,e_1}(V)}$
					for every integer $n\geq1$.
			This proves Claim 1.
			\end{proof}

				Since $T$ is invertible, $R_0-T=-T$ is invertible.
				Moreover, Claim 1 applies to $0\in\spc_J$ for every
				$J\in\mathbb S$, and therefore
				\[
				(R_0-T)^{-1}=(-T)^{-1}
				\]
				is $\spc_J$-extendable power-associative for every
				$J\in\mathbb S$. Hence
				\[
				0\in\bigcap_{J\in\mathbb S}\rho_J^*(T),
				\]
				and consequently, by Definition
				\ref{def:pullback spec},
				\[
				0\in\rho^*(T).
				\]

			Let us fix a nonzero octonion $s$ and note that
			\begin{eqnarray}
				\begin{bmatrix}
					R_s&L_{e_1}\\
					L_{\overline{e_1}}&R_s
				\end{bmatrix}
				\begin{bmatrix}
					\mathcal I&-R_{s^{-1}}L_{e_1}\\
					0&\mathcal I
				\end{bmatrix}
				=
				\begin{bmatrix}
					R_s&0\\
					L_{\overline{e_1}}
					&R_s-L_{\overline{e_1}}R_{s^{-1}}L_{e_1}
				\end{bmatrix}.
			\end{eqnarray}
			This implies that $R_s-T$ is invertible if and only if
			\[
			T_s
			:=
			R_s-L_{\overline{e_1}}R_{s^{-1}}L_{e_1}
			:
			\O\to\O
			\]
			is invertible.

				\textbf{Claim 2:}
				For $s\neq0$, $T_s$ is invertible if and only if
				\[
				s\notin
				\bigl\{
				p\in\O:\abs{p}=1,\
				\fx{p}{e_1}_{\R}=0
				\bigr\}.
				\]

			\begin{proof}[\textbf{Proof of Claim 2:}]
				Write
				\[
				s=s_0+s_1e_1+\cdots+s_7e_7,
				\qquad
				\rho:=\abs{s}^2
				=\sum_{j=0}^{7}s_j^2,
				\]
				and set
				\[
				\alpha:=\rho-1,
				\qquad
				\beta:=\rho+1.
				\]
				Since $s\neq0$, we have
				$s^{-1}=\overline{s}/\rho$. Recall that
				\[
				T_s
				=
				R_s-L_{\overline{e_1}}
				R_{s^{-1}}L_{e_1}.
				\]
				To remove the denominator, define the real linear
				operator
				\[
				M_s:=\rho T_s
				=
				\rho R_s-
				L_{\overline{e_1}}R_{\overline{s}}L_{e_1}.
				\]
				Thus, for every $x\in\O$,
				\[
				M_s(x)
				=
				\rho(xs)
				-\overline{e_1}
				\bigl((e_1x)\overline{s}\bigr).
				\]
				
				We represent $M_s$ with respect to the ordered real
				basis
				\[
				\mathcal E=(1,e_1,e_2,e_3,e_4,e_5,e_6,e_7)
				\]
				as
				\[
				\begin{pmatrix}
					s_0\alpha&-s_1\beta&-s_2\beta&-s_3\beta&-s_4\beta&-s_5\beta&-s_6\beta&-s_7\beta\\
					s_1\beta&s_0\alpha&s_3\beta&-s_2\beta&s_5\beta&-s_4\beta&-s_7\beta&s_6\beta\\
					s_2\beta&-s_3\beta&s_0\alpha&s_1\beta&s_6\alpha&s_7\alpha&-s_4\alpha&-s_5\alpha\\
					s_3\beta&s_2\beta&-s_1\beta&s_0\alpha&s_7\alpha&-s_6\alpha&s_5\alpha&-s_4\alpha\\
					s_4\beta&-s_5\beta&-s_6\alpha&-s_7\alpha&s_0\alpha&s_1\beta&s_2\alpha&s_3\alpha\\
					s_5\beta&s_4\beta&-s_7\alpha&s_6\alpha&-s_1\beta&s_0\alpha&-s_3\alpha&s_2\alpha\\
					s_6\beta&s_7\beta&s_4\alpha&-s_5\alpha&-s_2\alpha&s_3\alpha&s_0\alpha&-s_1\beta\\
					s_7\beta&-s_6\beta&s_5\alpha&s_4\alpha&-s_3\alpha&-s_2\alpha&s_1\beta&s_0\alpha
				\end{pmatrix}.
				\]
				
				Computing the determinant of this real
				$8\times8$ matrix and factoring it, we obtain
				\begin{align}
					\det\nolimits_{\R}M_s
					&=
					\rho^4
					\left(
					(s_0-1)^2+\sum_{j=1}^{7}s_j^2
					\right)^2
					\left(
					(s_0+1)^2+\sum_{j=1}^{7}s_j^2
					\right)^2
					\nonumber\\
					&\qquad\qquad\times
					\left(
					(\rho-1)^2+4s_1^2
					\right)^2
					\label{eq:detMs}\\
					&=
					\rho^4
					\abs{s-1}^4\abs{s+1}^4
					\left(
					(\rho-1)^2+4s_1^2
					\right)^2.
					\nonumber
				\end{align}
				
				Since $M_s=\rho T_s$ and $\O$ has real dimension
				eight,
				\[
				\det\nolimits_{\R}M_s
				=
				\rho^8\det\nolimits_{\R}T_s.
				\]
				Consequently,
				\begin{equation}\label{eq:detTs}
					\det\nolimits_{\R}T_s
					=
					\frac{
						\abs{s-1}^4\abs{s+1}^4
						\left(
						(\abs{s}^2-1)^2+4s_1^2
						\right)^2
					}{
						\abs{s}^8
					}.
				\end{equation}
				
				All factors in \eqref{eq:detTs} are nonnegative.
				The two exceptional factors
				$\abs{s-1}$ and $\abs{s+1}$ vanish only at
				$s=\pm1$, and these two points already satisfy
				$\abs{s}=1$ and $s_1=0$. Hence
				\[
				\det\nolimits_{\R}T_s=0
				\quad\Longleftrightarrow\quad
				\abs{s}=1
				\quad\text{and}\quad
				s_1=0.
				\]
				Since
				\[
				s_1=\fx{s}{e_1}_{\R},
				\]
				we conclude that
				\[
				T_s\ \text{is noninvertible}
				\quad\Longleftrightarrow\quad
				s\in
				\bigl\{
				p\in\O:\abs{p}=1,\
				\fx{p}{e_1}_{\R}=0
				\bigr\}.
				\]
				This proves Claim 2.
			\end{proof}

				Combining Claim 1 and Claim 2, for every
				$J\in\mathbb S$ we obtain
				\[
				\sigma_J^*(T)
				=
				\spc_J\cap
				\bigl\{
				s\in\O:\abs{s}=1,\
				\fx{s}{e_1}_{\R}=0
				\bigr\}.
				\]
				In particular,
				\[
				\sigma_J^*(T)\cap\mathbb R=\{-1,1\},
				\qquad J\in\mathbb S.
				\]
							Consequently,
				\[
				\begin{aligned}
					\sigma^*(T)
					&=
					\bigcup_{J\in\mathbb S}\sigma_J^*(T)\\
					&=
					\bigl\{
					s\in\O:\abs{s}=1,\
					\fx{s}{e_1}_{\R}=0
					\bigr\}.
				\end{aligned}
				\]

			Thus $\sigma^*(T)$ is the unit $6$-sphere in the
			$7$-dimensional real hyperplane
			$e_1^{\perp_{\R}}$.   In particular, it is not axially
			symmetric. For example,
			\[
			\frac{1+e_2}{\sqrt{2}}\in\sigma^*(T),
			\qquad
			\frac{1+e_1}{\sqrt{2}}\notin\sigma^*(T),
			\]
			although these two points have the same real part and the
			same modulus of their imaginary parts.
		\end{eg}

		\begin{rem}
			Example \ref{eg:non sphere} also shows that the classical form of the spectral
			mapping theorem does not hold, in general, for the global
			octonionic pull-back spectrum.
			
			Indeed, for
			\[
			T=
			\begin{pmatrix}
				0&-e_1\\
				e_1&0
			\end{pmatrix},
			\]
			alternativity gives
			\[
			T^2=I.
			\]
			Hence
			\[
			\sigma^*(T^2)=\sigma^*(I)=\{1\}.
			\]
			On the other hand, by the preceding example,
			\[
			\sigma^*(T)
			=
			\bigl\{
			s\in\O:\abs{s}=1,\ 
			\fx{s}{e_1}_{\R}=0
			\bigr\}.
			\]
			We claim that this set is invariant under the squaring map.
			Indeed, write
			\[
			s=a+u,
			\qquad
			u\in\operatorname{Im}\O,\quad
			u\perp_{\R} e_1,\quad
			a^2+\abs{u}^2=1.
			\]
			Then
$		s^2
			=
			a^2-\abs{u}^2+2au,$
			so that
			\[
			\abs{s^2}=1,
			\qquad
			\fx{s^2}{e_1}_{\R}=0.
			\]
			Conversely, every element of $\sigma^*(T)$ has a square root
			which still belongs to $\sigma^*(T)$. Consequently,
			\[
			\bigl\{
			s^2:s\in\sigma^*(T)
			\bigr\}
			=
			\sigma^*(T).
			\]
			Therefore,
			\[
			\sigma^*(T^2)
			=
			\{1\}
			\neq
			\bigl\{
			s^2:s\in\sigma^*(T)
			\bigr\}.
			\]
			Thus, even for the polynomial $p(s)=s^2$, one cannot expect, in
			general, the classical identity
			\[
			\sigma^*(p(T))
			=
			p\bigl(\sigma^*(T)\bigr)
			\]
			for the global pull-back spectrum. 
		\end{rem}

		\section{Algebraic properties of the octonionic functional calculus}
		In this section $V$ is a Banach  octonionic bimodule and $T\in \mathscr{B}_{\mathcal {RO}}(V)$ is power-associative. We aim to characterize the algebraic properties of the octonionic functional calculus $\Phi_T$ and $\Psi_T$.

		Note that by Theorem \ref{thm:BROV is Oalg}, $\mathscr{B}_{\mathcal{RO}}(V)$ is not merely an $\mathbb{O}$-bimodule, but also an $\mathbb{O}$-algebra (see Definition \ref{def:O alg}). Similarly, by Theorem \ref{thm:slice algebra}, $\mathcal{SR}^L(\Omega)$ is not only an $\mathbb{O}$-bimodule, but also an alternative algebra.
		With respect to the $\mathbb{O}$-bimodule structure, Definitions \ref{def:ofunccal} and \ref{def:lifofunccal} imply that $\Phi_T$ (resp. $\Psi_T$) are right (resp. left) para-linear maps. We now focus our attention on the $\O$-algebraic structure.

		Recall the definition of the slice product (see Definition \ref{def:slice product}) and $\mathbb{O}$-algebras (see Definition \ref{def:O alg}). The following result can be verified directly.
		\begin{prop}\label{prop:SR O-alg}
			Let $\Omega$ be an axially symmetric s-domain in $\mathbb{O}$. Endowed with the slice product, both the algebras $(\mathcal{SR}^L(\Omega),\bullet^L)$ and $(\mathcal{SR}^R(\Omega),\bullet^R)$ are $\mathbb{O}$-algebras.
		\end{prop}

		For any $f=\sum_{0}^7e_i\bullet^R f_{(i)}\in  \mathcal{SR}^R(\Omega)$ where $f_{(i)}\in \mathcal{SR}_{\R}(\Omega)$,  define:
		\begin{eqnarray}
			\tilde{f} :&=\sum_{i=0}^7e_i\bullet^L f_{(i)}\in \mathcal{SR}^L(\Omega).
		\end{eqnarray}
		Similarly, we can also define 	\begin{eqnarray}
			\tilde{g} :&=\sum_{i=0}^7e_i\bullet^R g_{(i)}\in \mathcal{SR}^R(\Omega).
		\end{eqnarray}
		for $g=\sum_{0}^7e_i\bullet^L g_{(i)}\in \mathcal{SR}^L(\Omega)$.
		\begin{lemma}\label{lem:ftiled}
			Let  $f\in \spc_J(\mathcal{SR}^R({\Omega}))$ for some $J\in \mathbb S$. Then
			\begin{enumerate}
				\item $\tilde{f}\in \spc_J(\mathcal{SR}^L({\Omega}))$.
				\item $f(s)=\tilde{f}(s)$ for all $s\in \spc_J\cap \Omega.$
				\item If $g\in \spc_J(\mathcal{SR}^L(\Omega))$, then $f\bullet^R \tilde{g}\in \spc_J(\mathcal{SR}^R({\Omega})),\tilde{f}\bullet^L g\in \spc_J(\mathcal{SR}^L({\Omega}))$ and $$\widetilde{{f\bullet^R \tilde{g}}}=\tilde{f}\bullet^L g.$$
				Moreover, for all $s\in \spc_J\cap \Omega$,
				$$f(s)g(s)={{f\bullet^R \tilde{g}}}(s)=\tilde{f}\bullet^L g(s).$$
			\end{enumerate}
		\end{lemma}
		\begin{proof}
			Suppose $$f=f_{(0)}+J\bullet^Rf_{(1)},$$ where $f_{(0)},f_{(1)}\in \mathcal{SR}_{\R}({\Omega})$.
			\begin{enumerate}
				\item By definition, we have $$\tilde{f}= f_{(0)}+J\bullet^Lf_{(1)}\in \spc_J(\mathcal{SR}^L({\Omega})).$$
				\item 	For any $s\in \spc_J\cap \Omega$, it follows from Proposition \ref{prop:slicepre} that
				$$\tilde{f}(s)=f_{(0)}(s)+(J\bullet^Lf_{(1)})(s)=f_{(0)}(s)+f_{(1)}(s)J=f_{(0)}(s)+Jf_{(1)}(s)=f(s).$$
				\item This follows from direct calculations by writing $g=g_{(0)}+J\bullet^Rg_{(1)}$ with $g_{(0)},g_{(1)}\in \mathcal{SR}_{\R}({\Omega})$.
			\end{enumerate}
		\end{proof}
		
		\begin{mydef}
			Let $V$ be a Banach  octonionic bimodule and $T\in \mathscr{B}_{\mathcal {RO}}(V)$ be power-associative. The \textbf{spectrum} of $T$ is defined as
			\begin{eqnarray}
				\sigma(T):=\sigma_{*}(T)\cup \sigma^*(T).
			\end{eqnarray}
			
			Let $U\subseteq \O$ be a bounded \textbf{axially symmetric s-domain} that contains the closure of the  spectrum $\sigma(T )$ be such that $\partial(U \cap \spc_J )$ is the union of a finite number of continuously differentiable Jordan curves for every $J\in \mathbb S$. We say that $U$ is a \textbf{$T$-admissible} open set.
			
			Similar definitions for $ \mathcal{SR}^L({\sigma(T)})$ and $ \mathcal{SR}^R({\sigma(T)})$.
		\end{mydef}
		\begin{thm}\label{thm:f-*T=f*T}
			Let $V$ be a Banach  octonionic bimodule and $T\in \mathscr{B}_{\mathcal {RO}}(V)$ be power-associative.
			If $f\in \spc_J(\mathcal{SR}^R({\sigma(T)})$  for some $J\in \mathbb S$, then
			\begin{eqnarray}\label{eq:ref*=ref-*}
				\operatorname{Re}_{{\mathscr{B}_{\mathcal {RO}}(V)}}f_{*}(T)_J=\operatorname{Re}_{\mathscr{B}_{\mathcal {RO}}(V)} \tilde{f}^{*}(T)_J.
			\end{eqnarray}
		\end{thm}
		\begin{proof}
			To show \eqref{eq:ref*=ref-*}, by Uniqueness Lemma \ref {lem:Uniqueness Lemma}, it suffices to verify
			\begin{eqnarray}\label{eqpf:ref*=ref-*}
				[\operatorname{Re}_{{\mathscr{B}_{\mathcal {RO}}(V)}}f_{*}(T)_J](v)=[\operatorname{Re}_{\mathscr{B}_{\mathcal {RO}}(V)} \tilde{f}^{*}(T)_J](v)
			\end{eqnarray} for all $v\in \re V$.
			In view of \eqref{eq:re f}, \eqref{eqpf:ref*=ref-*} is equivalent to
			\begin{eqnarray}\label{eqpf:ref-*=ref*}
				\operatorname{Re}_{{V}}[f_{*}(T)_J(v)]=\operatorname{Re}_{V}[ \tilde{f}^{*}(T)_J(v)], \qquad \text{for all $v\in \re V$}.
			\end{eqnarray}
			
			Let $U\subseteq \O$  be a $T$-admissible domain on which $f, \tilde{f}$ are both slice regular.  		
			By Lemma \ref{lem:ftiled} (1), we obtain  from $f\in \spc_J(\mathcal{SR}^R({\sigma(T)})$ that $\tilde{f}\in\spc_J(\mathcal{SR}^L({\sigma(T)})$. In view of  \eqref{eqpf:f*TJv}, we have
			$$\tilde{f}^{*}(T)_J(v)=	\dfrac{1}{2\pi} \int_{\partial (U\cap \spc_J)} (R_s-T)^{ -1}(v) \big({d}s_J\tilde{f}(s)\big).$$
			Comparing this with \eqref{eqpf:ref-*T}, \eqref{eqpf:ref-*=ref*} follows from 	 Lemma \ref{lem:ftiled} (2) immediately.
		\end{proof}

		\begin{thm}
			Let $V$ be a Banach  octonionic bimodule and $T\in \mathscr{B}_{\mathcal {RO}}(V)$  be power-associative.  
			If $f\in \spc_J(\mathcal{SR}^R({\sigma(T)})$, $g\in \spc_J(\mathcal{SR}^L({\sigma(T)})$, then
			\begin{eqnarray}\label{eq:ref-*Tg*T}
				\operatorname{Re}f_{*}(T)_J\circledcirc g^*(T)_J=\operatorname{Re}(f\bullet^R \tilde{g})_{*}(T)_J=\operatorname{Re}(\tilde{f}\bullet^L g)^*(T)_J.
			\end{eqnarray}
			Here $\re$ is  the real part operator on ${\mathscr{B}_{\mathcal {RO}}(V)}$.
		\end{thm}
		\begin{proof}
			Fix $v\in \re V$ arbitrarily. Let $U\subseteq \O$  be a $T$-admissible domain on which $f, g$ are both slice regular.	Let $U_1,U_2\subseteq \O$  be two $T$-admissible domains satisfying:
			\begin{enumerate}
				\item $U_1\cup \partial U_1\subsetneq U_2$;
				\item $U_2\cup \partial U_2\subseteq U$.
			\end{enumerate}	
			In view of  \eqref{eqpf:ref-*T} and \eqref{eqpf:f*TJv}, we get
			\begin{eqnarray}\label{eqpf:ref-*g*}
				&\quad\,&	\operatorname{Re}_V\Big(f_{*}(T)_J\circledcirc g^*(T)_J\Big)(v)\\
				&=&\operatorname{Re}_Vf_{*}(T)_J\big( g^*(T)_J(v)\big)\notag\\
				&=&\operatorname{Re}_V\dfrac{1}{2\pi} \int_{\partial (U_1\cap \spc_J)}\Big((f(s){d}s_J)   \odot(R_s-T)^{-\circledcirc }\Big)\left( \dfrac{1}{2\pi} \int_{\partial (U_2\cap \spc_J)}\Big(  (R_q-T)^{\circledcirc -} \odot{d}q_Jg(q)\Big)(v)\right)\notag\\
				&=&\dfrac{1}{4\pi^2}\operatorname{Re}_V \int_{\partial (U_1\cap \spc_J)}\Big((f(s){d}s_J)   \odot(R_s-T)^{-\circledcirc }\Big)\left(  \int_{\partial (U_2\cap \spc_J)}  (R_q-T)^{-1}(v) {d}q_Jg(q)\right)\notag\\
				&=&\dfrac{1}{4\pi^2}\operatorname{Re}_V \int_{\partial (U_1\cap \spc_J)}   (R_s-T)^{-1 }\left(\left(  \int_{\partial (U_2\cap \spc_J)}  (R_q-T)^{-1}(v) {d}q_Jg(q)\right)(f(s){d}s_J)\right)\notag\\
				&=&\dfrac{1}{4\pi^2}\operatorname{Re}_V \int_{\partial (U_1\cap \spc_J)}  \int_{\partial (U_2\cap \spc_J)}  (R_s-T)^{-1 }\left(   (R_q-T)^{-1}(v) \right)\, g(q)f(s){d}q_J{d}s_J.\notag
			\end{eqnarray}
			The last equality follows from  $f(s),g(q)\in \spc_J$ for any  $s\in \partial (U_1\cap \spc_J), q\in \partial (U_2\cap \spc_J)$ and $(R_s-T)^{-1}$ (resp. $(R_q-T)^{-1}$) is $\mathbb C_J$-liftable (resp. extendable) power associative.
			
			Fix $s\in \partial (U_1\cap \spc_J), q\in \partial (U_2\cap \spc_J)$ arbitrarily. Since $(R_s-T)^{-1}$ is $\mathbb C_J$-extendable power associative, it follows from \eqref{eqpf:extasp} that for any $x\in \spc_J(V)$, we have 
			\begin{align*}
				((R_s-T)^{-1}-	(R_q-T)^{-1})(x)&=	\Big((R_s-T)^{-1}R_{q-s}	(R_q-T)^{-1}\Big)(x)\\
				&=\Big((R_s-T)^{-1}	(R_q-T)^{-1}\Big)(x(q-s)).
			\end{align*}
			{We use here the ordinary resolvent identity $(R_s-T)^{-1}-(R_q-T)^{-1}=(R_s-T)^{-1}R_{q-s}(R_q-T)^{-1}$; the $\mathbb C_J$-extendability permits the scalar factor $q-s$ to be moved as in the next line.}
			Noticing $q-s\neq 0$ by $U_1\cup \partial U_1\subsetneq U_2$, this implies that for all $x\in \spc_J(V)$,
			\begin{eqnarray}
				((R_s-T)^{-1}-	(R_q-T)^{-1})(x(q-s)^{-1})=(R_s-T)^{-1}	(R_q-T)^{-1}(x).		
			\end{eqnarray}
			Combining this with \eqref{eqpf:ref-*g*}, we obtain from $(R_s-T)^{-1}$ (resp. $(R_q-T)^{-1}$) is $\mathbb C_J$-liftable (resp. extendable) power associative that
			\begin{align}\label{eqpf:ref-*g*T2}
				&\quad\,	\operatorname{Re}_V\Big(f_{*}(T)_J\circledcirc g^*(T)_J\Big)(v)\\
				&=\dfrac{1}{4\pi^2}\operatorname{Re}_V \int_{\partial (U_1\cap \spc_J)}  \int_{\partial (U_2\cap \spc_J)} 	((R_s-T)^{-1}-	(R_q-T)^{-1})(v(q-s)^{-1}) g(q)f(s){d}q_J{d}s_J\notag\\
				&=\dfrac{1}{4\pi^2}\operatorname{Re}_V \int_{\partial (U_1\cap \spc_J)}  \int_{\partial (U_2\cap \spc_J)} 	\Big((R_s-T)^{-1}(v(q-s)^{-1})-	(R_q-T)^{-1}(v(q-s)^{-1}) \Big)g(q)f(s){d}q_J{d}s_J\notag\\
				&=\dfrac{1}{4\pi^2}\operatorname{Re}_V \int_{\partial (U_1\cap \spc_J)}  \int_{\partial (U_2\cap \spc_J)} 	\Big((R_s-T)^{-1}(v)(q-s)^{-1}-	(R_q-T)^{-1}(v)(q-s)^{-1} \Big)g(q)f(s){d}q_J{d}s_J\notag\\
				&=\dfrac{1}{4\pi^2}\operatorname{Re}_V \int_{\partial (U_1\cap \spc_J)}   	(R_s-T)^{-1}(v)\int_{\partial (U_2\cap \spc_J)}g(q)(q-s)^{-1}{d}q_Jf(s){d}s_J\notag\\
				&\quad-\dfrac{1}{4\pi^2}\operatorname{Re}_V	\int_{\partial (U_2\cap \spc_J)}(R_q-T)^{-1}(v) g(q)\int_{\partial (U_1\cap \spc_J)}f(s)(q-s)^{-1}{d}s_J{d}q_J\notag\\
				&=\dfrac{1}{2\pi}\operatorname{Re}_V \int_{\partial (U_1\cap \spc_J)}   	(R_s-T)^{-1}(v)g(s)f(s){d}s_J\notag	\end{align}
			We used the Fubini theorem in the penultimate line. In view of  formula \eqref{eqpf:ref-*T} and  Lemma \ref {lem:ftiled} (3), \eqref{eqpf:ref-*g*T2} implies that
			\begin{align*}
				\operatorname{Re}_V\Big(f_{*}(T)_J\circledcirc g^*(T)_J\Big)(v)=	\operatorname{Re}_V\Big((f\bullet^R \tilde{g})_{*}(T)_J\Big)(v).
			\end{align*}
			Then the  Uniqueness Lemma \ref {lem:Uniqueness Lemma} and  \eqref{eq:re f} implies that
			$$	\operatorname{Re}f_{*}(T)_J\circledcirc g^*(T)_J=\operatorname{Re}(f\bullet^R \tilde{g})_{*}(T)_J.$$
			Here  $\re$ is  the real part operator on ${\mathscr{B}_{\mathcal {RO}}(V)}$. By  Lemma \ref {lem:ftiled} (3) and Theorem \ref{thm:f-*T=f*T}, we obtain
			$$\operatorname{Re}(f\bullet^R \tilde{g})_{*}(T)_J=\operatorname{Re}(\tilde{f}\bullet^L g)^*(T)_J.$$
		\end{proof}
		Recall the notion of sphere invariance of $T$ in Definitions \ref{def:lsliceinv} and \ref{def:rsliceinv}.
		
\begin{thm}\label{thm:algebra prop}
Let $V$ be a Banach octonionic bimodule and $T\in\mathscr{B}_{\mathcal{RO}}(V)$ be power-associative and both left and right \sinv. Then:
\begin{enumerate}
\item If $f\in\mathcal{SR}_{\R}(\sigma(T))$ and $g\in\mathcal{SR}^L(\sigma(T))$, then
\begin{eqnarray}\label{eq:reslicef-*Tg*T}
\operatorname{Re}f_{*}(T)\circledcirc g^*(T)=\operatorname{Re}(fg)^*(T).
\end{eqnarray}
\item If $f\in\mathcal{SR}^R(\sigma(T))$ and $g\in\mathcal{SR}_{\R}(\sigma(T))$, then
\begin{eqnarray}\label{eq:ref-*Tsliceg*T}
\operatorname{Re}f_{*}(T)\circledcirc g^*(T)=\operatorname{Re}(fg)_{*}(T).
\end{eqnarray}
\end{enumerate}
Here $\operatorname{Re}$ is the real-part operator on $\Gamma(\mathbb S,\mathscr{B}_{\mathcal{RO}}(V))$.
\end{thm}

\begin{proof}
			We prove the  part (1),   part (2) is analogous.
			
			Let   $f\in \mathcal{SR}_{\R}({\sigma(T)})$ and $T$ be left and right \sinv.	Suppose $g=\sum_{0}^7e_i\bullet^L g_{(i)}$. For any $i=1,\dots,7$, we have $e_i\bullet^L g_{(i)}\in \spc_{e_i}(\mathcal{SR}^L({\sigma(T)})$. In view of \eqref{eq:ref-*Tg*T}, we conclude from  $\tilde{f}=f$ and Proposition \ref{prop:slicepre} that for any 
		$i=1,\dots,7$,			\begin{align*}
				&\quad\,	\operatorname{Re}_{\mathscr{B}_{\mathcal{RO}}(V)}f_{*}(T)_{e_i}\circledcirc (e_i\bullet^L g_{(i)})^*(T)_{e_i}\\
				&=\operatorname{Re}_{\mathscr{B}_{\mathcal{RO}}(V)}(\tilde{f}\bullet^L (e_i\bullet^L g_{(i)}))^*(T)_{e_i}\\
				&=\operatorname{Re}_{\mathscr{B}_{\mathcal{RO}}(V)}(f(e_i\bullet^L g_{(i)}))^*(T)_{e_i}.
			\end{align*}
			Since $T$ is left and right \sinv, it follows that for any $i=1,\dots,7$,	$$\operatorname{Re}_{\Gamma(\mathbb S, \mathscr{B}_{\mathcal{RO}}(V))}f_{*}(T)\circledcirc (e_i\bullet^L g_{(i)})^*(T)=\operatorname{Re}_{\Gamma(\mathbb S, \mathscr{B}_{\mathcal{RO}}(V))}(f(e_i\bullet^L g_{(i)}))^*(T).$$
	 Similarly, we can also prove the case $i=0$:
			$$\operatorname{Re}_{\Gamma(\mathbb S, \mathscr{B}_{\mathcal{RO}}(V))}f_{*}(T)\circledcirc ( g_{(0)})^*(T)=\operatorname{Re}_{\Gamma(\mathbb S, \mathscr{B}_{\mathcal{RO}}(V))}(f g_{(0)})^*(T).$$
			Summing over $i=0,\dots,7$ gives \eqref{eq:reslicef-*Tg*T}. The proof of part (2) is obtained by interchanging the left and right calculi.
		\end{proof}

	\begin{rem}
			Similar results also hold in the quaternionic case \cite[Theorem 4.11.6]{colombo2011noncomfunctcalculus}. 	Let $V$ be a two-sided quaternionic Banach space and $T$ be a right quaternionic-linear operator. Denote by $\sigma_S(T)$ the quaternionic S-spectrum of $T$. Using the same symbol as in Theorem \ref{thm:algebra prop}, we rewrite the quaternionic result as follows:
			\begin{enumerate}
				\item If $f\in \mathcal{SR}_{\R}({\sigma_S(T)})$, $g\in \mathcal{SR}^L({\sigma_S(T)})$, then \begin{eqnarray}\label{eq:leftH}
					(fg)(T)=f(T)g(T).
				\end{eqnarray}
				\item If $f\in \mathcal{SR}^R({\sigma(T)})$, $g\in \mathcal{SR}_{\R}({\sigma(T)})$, then \begin{eqnarray}\label{eq:rightH}
					(fg)(T)=f(T)g(T).
				\end{eqnarray}
			\end{enumerate}
		\end{rem}
		\section{Concluding Remarks}
		\label{sec:conclusion}
		
		
		\subsection{Unified framework for   Riesz-Dunford theory over  division algebras}
		In this article, we establish the octonionic version of the Riesz-Dunford functional calculus theory for \textbf{bounded, para-linear, power-associative} operators acting on Banach \(\mathbb{O}\)-bimodules. {As such, it extends the Riesz-Dunford functional calculus theory in complex and quaternionic case \cite{Dunford1958,colombo2011noncomfunctcalculus}.
			We list below the main differences among  these Riesz-Dunford functional calculi:}
		\begin{enumerate}
			\item
			The operators considered in these three cases are \textbf{complex linear},\textbf{ right quaternionic linear} and \textbf{right octonionic para-linear} (Definition \ref{def:para_linear}), respectively.
			\item For right quaternionic linear operators, there is only one type of spectrum {associated with the functional calculus}, i.e., the S-spectrum \cite{colombo2011noncomfunctcalculus}; while for right octonionic  para-linear  operators, there are two types of spectra, i.e., the \textbf{pull-back spectrum} and the \textbf{push-forward spectrum} (Definitions  \ref{def:pullback spec} and \ref{def:pushforward sepc})
			
			\item	For complex linear operators, there exists one type of functional calculus. In contrast, for the quaternionic and the octonionic cases, there are two types of functional calculi {corresponding to \textbf{left (resp. right) slice regular} functions which, in view of the non-commutativity, form different classes of functions although somewhat equivalent.}
			
		\end{enumerate}

		{When studying operators over the algebras $\mathbb H$ and $\O$, we can consider these operators also on the real or complex subfields. Of course, this approach cannot be taken to deal with some specific questions, for example the study of subspaces associated with linear operators, and below we discuss some further issues.}
		
		\begin{enumerate}
			\item   A right \textbf{quaternionic linear} or \textbf{octonionic linear} operator $T$ can be regarded as a $\spc_J$-linear operator acting on the space $V_J$ (see Remark  \ref{rem:pull-back spec}). Here $V_J$, as a set, is identical to the original space, but it is endowed with a complex vector structure induced by $R_J$.
			Thus a quaternionic or octonionic linear operator can be thought of as a special complex linear operator. In contrast, a right \textbf{octonionic  para-linear} operator can  not  be regarded as a $\spc_J$-linear  or quaternionic linear  operator. The method developed in this article is the first theoretical framework to deal with the functional calculus of octonionic para-linear operators.
			
			\item {As explained above, an octonionic para-linear operator cannot be treated as a $\spc_J$-linear operator.} Consequently, the S-spectral functional calculus, which is presented in quaternionic form rather than in complex form, holds significant importance from a unified perspective of division algebras.
		\end{enumerate}
		
		The results established in this article further provide a unified framework for investigating the Riesz-Dunford functional calculus theory over the division algebras \(\mathbb{C}\), \(\mathbb{H}\), and \(\mathbb{O}\). 	We highlight the following key facts, which guarantee the theory we developed is consistent with these three Riesz-Dunford functional calculus theories:

		\begin{enumerate}
			\item The notion of \textbf{para-linearity}   can be defined over other division algebras and degenerates into classical linearity when the algebra under consideration is associative \cite{huoqinghai2022Riesz}.
			\item The requirement of \textbf{power-associativity}  (Definition \ref{def:powerass}) is automatically satisfied    when the algebra under consideration is associative.
			\item The \textbf{resolvent operators} $(R_s-T)^{\circledcirc -}$ and $(R_s-T)^{-\circledcirc }$ (Definition \ref{def:resolvent op}) share the same form over  $\spc,\mathbb H,\O$ (see equalities \eqref{eq:lroq} and \eqref{eq:rroh}). The \textbf{pull-back spectrum, push-forward spectrum} (Definitions  \ref{def:pullback spec} and \ref{def:pushforward sepc}) share the same form over  $\spc,\mathbb H,\O$ (see  Remarks \ref{rem:resop} and  \ref{rem:pull-back spec}).
			\item \textbf{Functional calculi formulas} \eqref{eq:f(T)J}  and \eqref{lifeq:f(T)J} reduce to the quaternionic formulas(Definition 4.10.4 in \cite{colombo2011noncomfunctcalculus}), as every operator $T$ in quaternionic case is \textbf{``\sinv''} (see Remark \ref{rem:sphereinv}) and  associators therein vanish automatically in the quaternionic case. {Similarly, formulas \eqref{eq:f(T)J}  and \eqref{lifeq:f(T)J} reduce to their complex counterparts.}
		\end{enumerate}

		Let $\sigma(T)$ denote the spectrum of a bounded complex linear operator $T$ defined on a complex Banach space, and let $\sigma_S(T)$ denote the S-spectrum of a bounded right quaternionic linear operator $T$ defined on a two-sided quaternionic Banach space, respectively.

		For comparison, we have compiled the definitions and properties of functional calculi over  the division algebras $\spc,\mathbb H$ (see \cite{colombo2011noncomfunctcalculus} and \cite{Dunford1958})  and $\O$ in the Tables \ref{tab:spec}, \ref{tab:formula}, and \ref{tab:prop} below.
		
		
		\begin{table}[htbp]
			\centering
			\caption{Spectra over  Division Algebras}
			\label{tab:spec}
			\renewcommand{\arraystretch}{4.7}  
			\setlength{\tabcolsep}{4pt}        
			\begin{tabular}{|c|>{\Centering\arraybackslash}m{2.78cm}|>{\Centering\arraybackslash}m{10.2cm}|}
				\hline
				{Algebra} & {Spectrum} & {Definition of Spectrum} \\
				\hline
				$\mathbb{C}$
				& spectrum $\sigma(T)$
				& $\displaystyle \sigma(T):=\spc\setminus\rho(T) \qquad\text{ where }\newline
				\rho(T) :=\{s\in \spc: s\mathcal I-T \text{ is bounded invertible}\}	$ \\
				\hline
				{$\mathbb{H}$}
				& S-spectrum $\sigma_S(T)$
				& $\displaystyle \sigma_S(T):={\mathbb{H}}\setminus\rho_S(T), \qquad\text{ where }\newline
				\rho_S(T) :=\{s\in \mathbb H: T^2 - 2\re [s] T + \abs{s}^2\mathcal I \text{ is bounded invertible}	\}$ \\
				\hline
				\multirow{3}{*}{$\mathbb{O}$}
				& pull-back spectrum $\sigma^*(T)$
				& {$\displaystyle \sigma^*(T):=\bigcup_{J\in\mathbb S}\sigma_J^*(T),\quad \sigma_J^*(T):=\mathbb C_J\setminus\rho_J^*(T)$}\newline
				$\displaystyle \rho_J^*(T):=\{s\in\mathbb C_J:R_s-T\text{ is invertible in }\mathscr{B}_{\mathbb R}(V)\text{ and }(R_s-T)^{-1}\text{ is }\mathbb C_J\text{-extendable power associative}\}$ \\ \cline{2-3}
				& push-forward spectrum $\sigma_*(T)$
				& {$\displaystyle \sigma_*(T):=\bigcup_{J\in\mathbb S}{\sigma_*}_J(T),\quad {\sigma_*}_J(T):=\mathbb C_J\setminus{\rho_*}_J(T)$}\newline
				$\displaystyle {\rho_*}_J(T):=\{s\in\mathbb C_J:R_s-T\text{ is invertible in }\mathscr{B}_{\mathbb R}(V)\text{ and }(R_s-T)^{-1}\text{ is }\mathbb C_J\text{-liftable power associative}\}$ \\ \cline{2-3}
				&spectrum $\sigma(T)$ &$\displaystyle \sigma(T):=\sigma^*(T)\cup \sigma_*(T)$\\
				\hline
			\end{tabular}
		\end{table}

		\begin{table}[htbp]
			\centering
			\caption{Functional Calculi Formulas over  Division Algebras}
			\label{tab:formula}
			\renewcommand{\arraystretch}{3.4}  
			\setlength{\tabcolsep}{4pt}        
			\begin{tabular}{|c|>{\Centering\arraybackslash}m{2.78cm}|>{\Centering\arraybackslash}m{10.2cm}|}
				\hline
				{Algebra} & {Function} & {Definition of Functional Calculi $\Phi$, $\Psi$ } \\
				\hline
				$\mathbb{C}$
				& $f\in\mathcal{A}(\sigma(T))$
				& $\displaystyle  \Phi_T (f):=f(T)	=\dfrac{1}{2\pi \sqrt{-1}} \int_{\gamma} 	(s\mathcal I-T)^{ -1} f(s)ds 	$ \\
				\hline
				\multirow{2}{*}{$\mathbb{H}$}
				& $f\in \mathcal{SR}^L(\sigma_S(T))$
				& $\displaystyle   \Phi_T (f):=f(T)	=\dfrac{1}{2\pi} \int_{\partial (U\cap \spc_J)} 	(R_s-T)^{\circledcirc -} ({d}s_Jf(s)) 	$ \\
				\cline{2-3}
				& $f\in \mathcal{SR}^R(\sigma_S(T))$
				& $\displaystyle  \Psi_T (f):=	f(T)	=\dfrac{1}{2\pi} \int_{\partial (U\cap \spc_J)}( f(s){d}s_J)	(R_s-T)^{-\circledcirc } $ \\
				\hline
				\multirow{2}{*}{$\mathbb{O}$}
				& $\displaystyle f=\sum_{i=0}^7f_{(i)}\bullet^L e_i\in\mathcal{SR}^L(\sigma^*(T))$
				& $\displaystyle (\Phi_T (f))(J):=f^*(T)_J	=\dfrac{1}{2\pi} \int_{\partial (U\cap \spc_J)} 	(R_s-T)^{\circledcirc -}\odot ({d}s_Jf(s)) + {	\dfrac{1}{2\pi}}\sum_{i=1}^{7} \int_{\partial (U\cap \spc_J)} \big[(R_s-T)^{\circledcirc -},{d}s_Jf_{(i)}(s),e_i\big]_{\mathscr{B}_{\mathcal{RO}}(V)}	$ \\		\cline{2-3}
				& $\displaystyle f=\sum_{i=0}^7e_i \bullet^R f_{(i)}\in\mathcal{SR}^R(\sigma_*(T))$
				& $\displaystyle (\Psi_T (f))(J):=	f_*(T)_J	=\dfrac{1}{2\pi} \int_{\partial (U\cap \spc_J)} (f(s){d}s_J)\odot 	(R_s-T)^{-\circledcirc } - 	\dfrac{1}{2\pi}\sum_{i=1}^{7} \int_{\partial (U\cap \spc_J)} \big[e_i,f_{(i)}(s){d}s_J,(R_s-T)^{-\circledcirc }\big]_{\mathscr{B}_{\mathcal{RO}}(V)}$ \\
				\hline
			\end{tabular}
		\end{table}

		In the Table 3, given an open subset $U\subset\spc$  we denote by $\mathcal{A}(U)$ the set of all  holomorphic functions on  $U$.
		
		\begin{table}[htbp]
			\centering
			\caption{Functional Calculi over  Division Algebras}
			\label{tab:prop}
			\renewcommand{\arraystretch}{1.6}  
			\begin{tabular}{|c|>{\Centering}m{1.93cm}|c|c|>{\Centering}m{2cm}|>{\Centering}m{3.3cm}|}
				\hline
				Algebra & Operator $T$ & Resolvent Operator & Spectrum & Function Class & Properties of Functional Calculi $\Phi$, $\Psi$ \\
				\hline
				$\mathbb{C}$
				& $\mathbb{C}$-linear
				& $(s\mathcal I - T)^{-1}$
				& $\sigma(T)$
				& $\mathcal{A}(\sigma(T))$
				& $\mathbb{C}$-linear algebraic homomorphism \\
				\hline
				\multirow{2}{*}{$\mathbb{H}$}
				& \multirow{2}{*}{{right $\mathbb{H}$-linear}}  
				& $(R_s - T)^{\circledcirc -}$
				& $\sigma_S(T)$
				& $\mathcal{SR}^L(\sigma_S(T))$
				& right $\mathbb{H}$-linear; satisfies \eqref{eq:leftH} \\
				\cline{3-6}
				& & $(R_s - T)^{-\circledcirc}$
				& $\sigma_S(T)$
				& $\mathcal{SR}^R(\sigma_S(T))$
				& left $\mathbb{H}$-linear; satisfies \eqref{eq:rightH} \\
				\hline
				\multirow{2}{*}{$\mathbb{O}$}
				& {\multirow{2}{2cm}{right $\mathbb{O}$-para-linear; power-associative}}  
				& $(R_s - T)^{\circledcirc -}$
				& $\sigma^*(T)$
				& $\mathcal{SR}^L(\sigma^*(T))$
				& right $\mathbb{O}$-para-linear; satisfies \eqref{eq:reslicef-*Tg*T} \\
				\cline{3-6}
				& & $(R_s - T)^{-\circledcirc}$
				& $\sigma_*(T)$
				& $\mathcal{SR}^R(\sigma_*(T))$
				& left $\mathbb{O}$-para-linear; satisfies \eqref{eq:ref-*Tsliceg*T} \\
				\hline
			\end{tabular}
		\end{table}
		\bigskip
		
		\subsection{Pull-back and Push-forward Methods in Octonionic Case}	
		From these  tables, we can observe that the two types of functional calculi defined over
		$\mathbb{H}$ have, formally,  identical. Furthermore, their respective proof procedures are also highly analogous \cite{colombo2011noncomfunctcalculus}. In contrast, while the two kinds of functional calculi over
		$\mathbb{O}$ follow a similar constructive framework, they differ fundamentally in their intrinsic properties. The study of the octonionic case uncovers two crucial bijections—$\lif$, $\ext$ (see Definition \ref{def:lif and ext})—which play a key role in the octonionic version of the Riesz-Dunford theory.	
		Roughly speaking, there are essentially two approaches:
		\begin{enumerate}
			\item \textbf{Pull-back}: first we restrict to the real part $\re V$ or $\spc_J$-slice $\spc_J(V)$, then \textbf{extend} to the whole space via the bijection $\ext$;
			
			\item \textbf{Push-forward}: first we confine to the results obtained after projecting onto the real part, then we \textbf{lift} to the entire target space via the bijection $\lif$.
		\end{enumerate}
		Moreover, these two methods are also applicable to the quaternionic case, see for example Remark \ref{rem:resop}.
		
		\subsection{Future research}
		We  provide some lines for future research.
		\begin{enumerate}
			\item {Remove or weaken the requirement of \textbf{power-associativity} on $T$.}
			
			The key point of the Riesz-Dunford theory is the definition of resolvent operator.	In view of Theorem \ref{thm:Rs-T}, the left resolvent operator can be defined as
			\begin{eqnarray}
				\ext (R_s-T)^{-1}(\mathcal I+\alpha(s,T))
			\end{eqnarray}
			and the right resolvent operator can be defined as
			\begin{eqnarray}
				\lif (\mathcal I+\beta(s,T))(R_s-T)^{-1}
			\end{eqnarray}
			respectively. Hence, we may assume suitable properties for
			$\alpha(s,T)$ and $\beta(s,T)$ 		to develop the Riesz-Dunford theory.
			For example, if $T$ is  nilpotent, then $\alpha(s,T)$ and $\beta(s,T)$ can be calculated directly and hence the resolvent operator can be given precisely.
			\item Generalize the slice regular functions in the Riesz-Dunford theory—originally defined on axially symmetric domains—to \textbf{non-axially symmetric} domains.
			
			This question makes sense since, in contrast with the S-spectrum in the quaternionic case, the octonionic pull-back spectrum or push-forward spectrum is not necessarily axially symmetric (see Example \ref{eg:non sphere}).
			Thus, there is no necessity to require slice regular functions to be defined on axially symmetric domains. However, for the sake of simplicity —and given the absence of a slice Cauchy integral formula for slice regular functions defined on non-axially symmetric domains— we only study functions defined on an axially symmetric domain that contains the spectrum.
			However, we may consider non-axially symmetric slice topology-domain \cite{Dou2023JEMS} in the definition of $T$-admissible sets (Definition \ref{def:T-admiss}).
			\item Conduct in-depth further study on the \textbf{non-{\sinv}} operator.
			This is a new phenomenon in the octonionic case. It may be relevant for studying the structure of  $\mathbb S^6$.
			\item Investigate the spectral mapping theorem and composition theorem.
			\item Generalize the results in this article to unbounded operators.
		\end{enumerate}

		\bigskip

		\bibliographystyle{plain}

		\bibliography{ref}
		
		\bigskip	\bigskip	\bigskip	\bigskip

	\end{document}